\documentclass[reqno]{amsart}
\usepackage{amssymb,color,epsf}
\usepackage [cmtip,arrow]{xy}
\xyoption{all}
\usepackage {pb-diagram,pb-xy}

\ifx\pdfpageheight\undefined
   \usepackage[dvips,colorlinks=true,linkcolor=blue,citecolor=red,%
      urlcolor=green]{hyperref}
   \usepackage[dvips]{graphicx}
   \makeatletter
   \edef\Gin@extensions{\Gin@extensions,.mps}
   \makeatother
\else
 \usepackage[pdftex]{graphicx}
   \usepackage[bookmarksopen=false,pdftex=true,breaklinks=true,%
      backref=page,pagebackref=true,plainpages=false,%
      hyperindex=true,pdfstartview=FitH,colorlinks=true,%
      pdfpagelabels=true,colorlinks=true,linkcolor=blue,%
      citecolor=red,urlcolor=green,hypertexnames=false%
      ]%
   {hyperref}
\fi
\usepackage{bm}

\usepackage{float}

\ifx\pdfpageheight\undefined
   \usepackage[dvips,colorlinks=true,linkcolor=blue,citecolor=red,%
      urlcolor=green]{hyperref}
   \makeatletter
   \edef\Gin@extensions{\Gin@extensions,.mps}
   \makeatother
\else
 \usepackage[pdftex]{graphicx}
   \usepackage[bookmarksopen=false,pdftex=true,breaklinks=true,%
      backref=page,pagebackref=true,plainpages=false,%
      hyperindex=true,pdfstartview=FitH,colorlinks=true,%
      pdfpagelabels=true,colorlinks=true,linkcolor=blue,%
      citecolor=red,urlcolor=green,hypertexnames=false%
      ]%
   {hyperref}
\fi
\usepackage{bm}

\newtheorem{theorem}{Theorem}[section]
\newtheorem{lemma}[theorem]{Lemma}
\newtheorem{corollary}[theorem]{Corollary}
\newtheorem{proposition}[theorem]{Proposition}

\theoremstyle{definition}
\newtheorem{definition}[theorem]{Definition}
\newtheorem{newdescription}[theorem]{Description}
\newtheorem{example}[theorem]{Example}

\newtheorem{notation}[theorem]{Notation}
\newtheorem{property}[theorem]{Property}
\newtheorem{algorithm}{\sc Algorithm}

\theoremstyle{remark}
\newtheorem{remark}[theorem]{Remark}

\definecolor{DarkBlue}{rgb}{0,0.1,0.55}
\definecolor{DarkGrey}{rgb}{0.2,0.2,0.2}

\numberwithin{equation}{section}

\newcommand{\hide}[1]{}

\newcommand{\defeq}{\;{\stackrel{\text{\tiny def}}{=}}\;}
\newcommand{\eop}{\hfill $\Box$}

\def\C {\ensuremath{\mathrm{C}}}
\def\CC {\ensuremath{\mathcal{C}}}

\def\R {\ensuremath{\mathrm{R}}}

\def\eps {\ensuremath{\varepsilon}}

\def\la {\ensuremath{\langle}}
\def\ra {\ensuremath{\rangle}}

\def\Ext{\ensuremath{{\rm Ext}}}

\def\ZZ{\ensuremath{{\rm Zer}}}
\def\RM{\ensuremath{{\rm RM}}}
\def\BGRM{\ensuremath{{\rm BGRM}}}

\def\Cr{\ensuremath{{\rm Cr}}}
\def\D{\ensuremath{{\rm D}}}
\def\Def{\ensuremath{{\rm Def}}}

\def\y {\ensuremath{\mathbf{y}}}

\begin{document}
\title[A baby step-giant step roadmap algorithm]
{A baby step-giant step roadmap algorithm for general algebraic sets}
\author{S. Basu}
\address{Department of Mathematics,
Purdue University, West Lafayette, IN 47907, USA}
\email{sbasu@math.purdue.edu}

\author{M-F. Roy}
\address{IRMAR (URA CNRS 305)
Universit\'e de Rennes 1
Campus de Beaulieu
35042 Rennes, cedex
France}
\email{marie-francoise.roy@univ-rennes1.fr}

\author{M. Safey El Din}
\address{
Sorbonne Universit\'es, UPMC, Univ. Paris 06, UMR CNRS 7606, LIP6\\
INRIA Paris-Rocquencourt Center PolSys Pro\-ject\\
Institut Universitaire de France, France}
\email{Mohab.Safey@lip6.fr}

\author{\'E. Schost}
\address{Computer Science Department, The University of Western Ontario, London, ON, Canada}
\email{eschost@uwo.ca}

\thanks{Communicated by Teresa Krick.}
\thanks{The first author was supported in part by NSF grants
  CCF-0915954,  CCF-1319080 and DMS-1161629. The first and the second authors did part of the work
  during a research stay in Oberwolfach as part of the Research in
  Pairs Programme. The third author is member of Institut
  Universitaire de France and supported by the French National
  Research Agency EXACTA grant (ANR-09-BLAN-0371-01) and GeoLMI grant
  (ANR-2011-BS03-011-06). The fourth author was supported by an NSERC
  Discovery Grant and by the Canada Research Chairs program.}


\begin{abstract}
  Let $\R$ be a real closed field and $\D \subset \R$ an ordered domain.  We
  give an algorithm that takes as input a polynomial $Q \in
  \D[X_1,\ldots,X_k]$, and computes a description of a roadmap of the
  set of zeros, $\ZZ(Q,\R^k)$, of $Q$ in $\R^k$.  The complexity of
  the algorithm, measured by the number of arithmetic operations in the 
  ordered
  domain $\D$, is bounded by 
  $d^{O(k \sqrt{k})}$, where $d = \deg(Q)\ge 2$.
  As a consequence, there exist algorithms for 
computing the number of 
 semi-algebraically connected components
of a real algebraic set, $\ZZ(Q,\R^k)$,
whose complexity is also   bounded by $d^{O(k \sqrt{k})}$,
where $d = \deg(Q)\ge 2$.
The best previously known 
algorithm for constructing a roadmap of a real algebraic
subset of $\R^k$ defined by a polynomial of degree $d$ 
has
complexity $d^{O(k^2)}$.
\end{abstract}

\maketitle
\section{Introduction}
\label{sec:intro}

The problem of designing efficient algorithms for deciding whether two
points belong to the same semi-algebraically connected component of a
 semi-algebraic set, as well as counting the number of
 semi-algebraically connected components of a given semi-algebraic set
$S \subset \R^k$ where $\R$ is a real closed field (for example the
field of real numbers), is a very important problem in algorithmic
semi-algebraic geometry.  

The first algorithm for solving this problem
\cite{SS} was based on the technique of cylindrical algebraic
decomposition \cite{Col,BPRbook2}, and consequently had doubly
exponential complexity in $k$. 

Algorithms with singly exponential complexity in $k$
for solving this problem was first introduced by Canny in
\cite{Canny87}, and then successively completed and refined in
\cite{GV90,GHRSV90,HRS90,HRS93,GR92,GV92,BPR99}.  They are all based
on a geometric idea introduced by Canny, the construction of a
one-dimensional s-a (i.e. semi-algebraic) subset of the given s-a set
$S$, called a {\em roadmap of $S$}, which has the property that it is
non-empty and s-a connected inside every s-a (i.e. semi-algebraically)
connected component of $S$.

In the papers mentioned above, the construction of a roadmap of a
s-a set $S$ depends on recursive calls to itself on several
(in fact, singly exponentially many) $(k-1)$-dimensional slices of
$S$, each obtained by fixing the first coordinate. For constructing
the roadmap of a real algebraic variety defined by a polynomial $Q
\in \R[X_1,\ldots,X_k]$ with $\deg(Q) \leq d$, this technique gave
an algorithm with complexity $d^{O(k^2)}$. The exponent in 
the upper bound on the complexity,
$O(k^2)$, is due to the fact that the depth of the
recursion in these algorithms could be as large as $k$.  This exponent
is not satisfactory since the total number of s-a
connected components is 
$(O(d))^k$
and so there is room for trying to
improve it.  However, this has turned out to be a rather difficult
problem with no progress until very recently.

A new construction for computing roadmaps, with an improved recursive
scheme of baby step - giant step type, has been proposed, and applied
successfully in the case of smooth bounded real
algebraic hypersurfaces in \cite{Mohab-Schost2010}. In this new
recursive scheme, the dimension drops by $\sqrt{k}$ in each recursive
call. As a result, the depth of the recursive calls in this new
algorithm is at most $\sqrt{k}$, and consequently the algorithm has a
complexity of $d^{O(k \sqrt{k})}$.  The proof of correctness of the
algorithm in \cite{Mohab-Schost2010} depends on certain results from
commutative algebra and complex algebraic geometry, in order to prove
smoothness of polar varieties corresponding to generic projections of
a non-singular hypersurface.  Choosing generic coordinates in the
algorithm is necessary since the non-singularity of polar varieties
does not hold for all projections, but only for a Zariski-dense set of
projections.  This is an important restriction, since there is no
known method for making such a choice of generic coordinates
deterministically within this improved complexity bound. As a result,
the authors obtain a randomized (rather than a deterministic)
algorithm for computing roadmaps: there might be cases where the
algorithm terminates and gives a wrong result.

In contrast to these techniques which depend on complex algebraic
geometry, the algorithm for constructing roadmaps described in
\cite{BPRbook2} depends 
mostly
on arguments which are semi-algebraic in
nature. The greater flexibility of semi-algebraic geometry (as opposed
to complex geometry) makes it possible to avoid genericity
requirements for coordinates.  More precisely, we apply the technique used in
\cite{BPRbook2} to make an infinitesimal deformation of the given
variety so that the original coordinates are good. Since the
infinitesimal deformation uses only one
infinitesimal, it does not affect the asymptotic complexity class of
the algorithm.

The goal of this paper is to obtain a {\em deterministic algorithm}
for computing the roadmap of a {\em general algebraic set}, combining a
baby step - giant step 
recursive scheme similar to that used in
\cite{Mohab-Schost2010} and extending techniques coming from \cite{BPRbook2}.

\smallskip
We start by recalling the precise definition of what is meant by a roadmap.

\begin{definition}
Let $S \subset \R^k$ be a s-a set.  A {\em roadmap} for $S$
is a s-a set $\RM(S)$ of dimension at most one contained in
$S$ which satisfies the following roadmap conditions:
\begin{enumerate}
\item $\RM_1$ For every s-a connected component $C$ of
  $S$, $C \cap \RM(S)$ is s-a connected.
\item $\RM_2$ For every $x \in {\R}$ and for every s-a
  connected component
  $C'$ of $S_x$, $C' \cap \RM(S) \neq \emptyset$,
  where we denote by $S_x$ the set $S \cap \pi_1^{-1}(x)$ for
  $x\in\R$,  
with
$\pi_1:\R^k\rightarrow\R$ the projection map onto the
  first coordinate.
\end{enumerate}
Let $\mathcal{M}\subset \R^k$ be a finite set of points. A roadmap for
$(S,\mathcal{M})$ is a s-a set $\RM(S, \mathcal{M})$ such
that $\RM(S, \mathcal{M})$ is a roadmap of $S$ and $\mathcal{M}\subset
\RM(S, \mathcal{M})$.
\end{definition}

We illustrate this definition by the picture of a torus in $\mathbb{R}^3$ and 
a roadmap of it.

\begin{figure}[h]
  \includegraphics[scale=0.5]{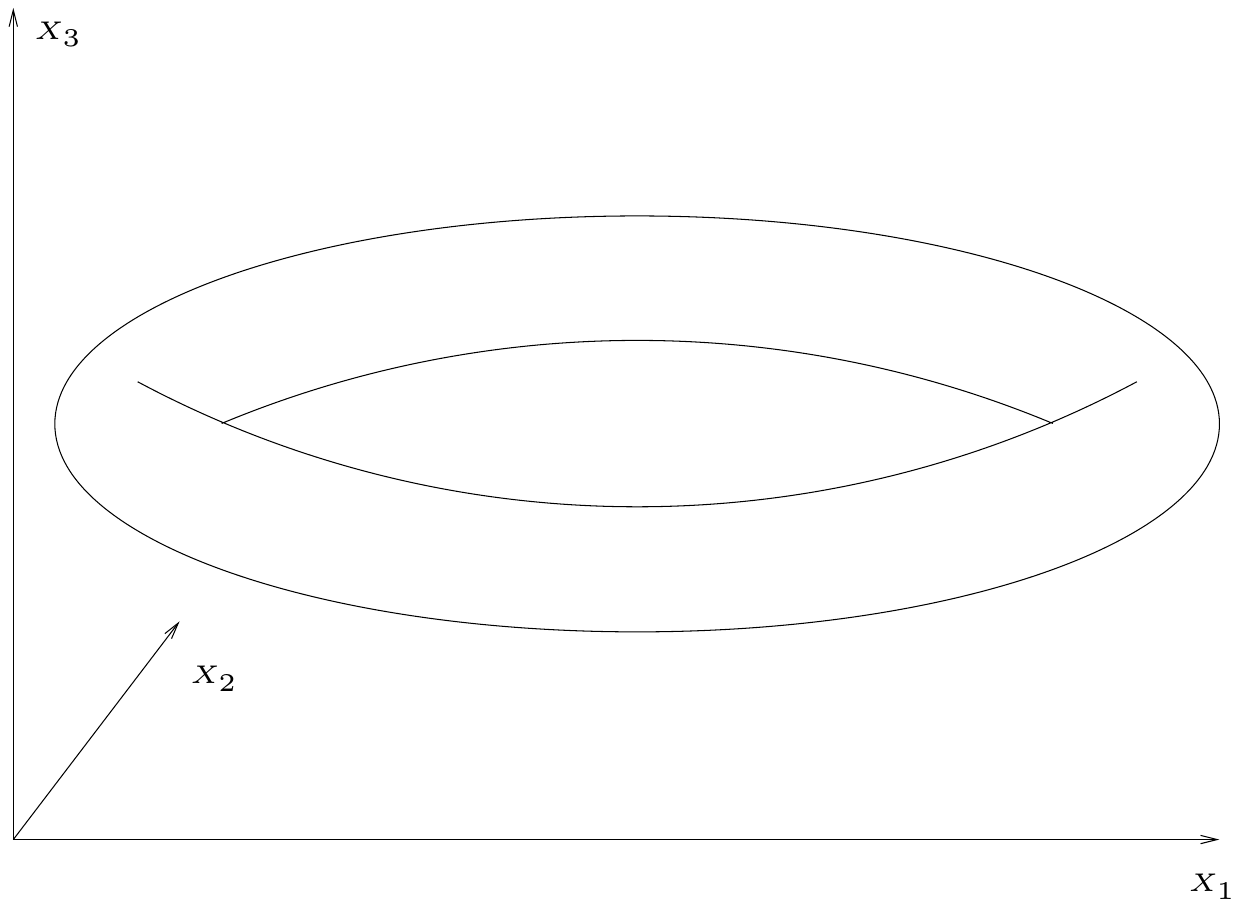}
  \vspace{-1.3in}
\caption{Torus in $\mathbb{R}^3$}
\end{figure}

\begin{figure}[h]
  \includegraphics[scale=0.5]{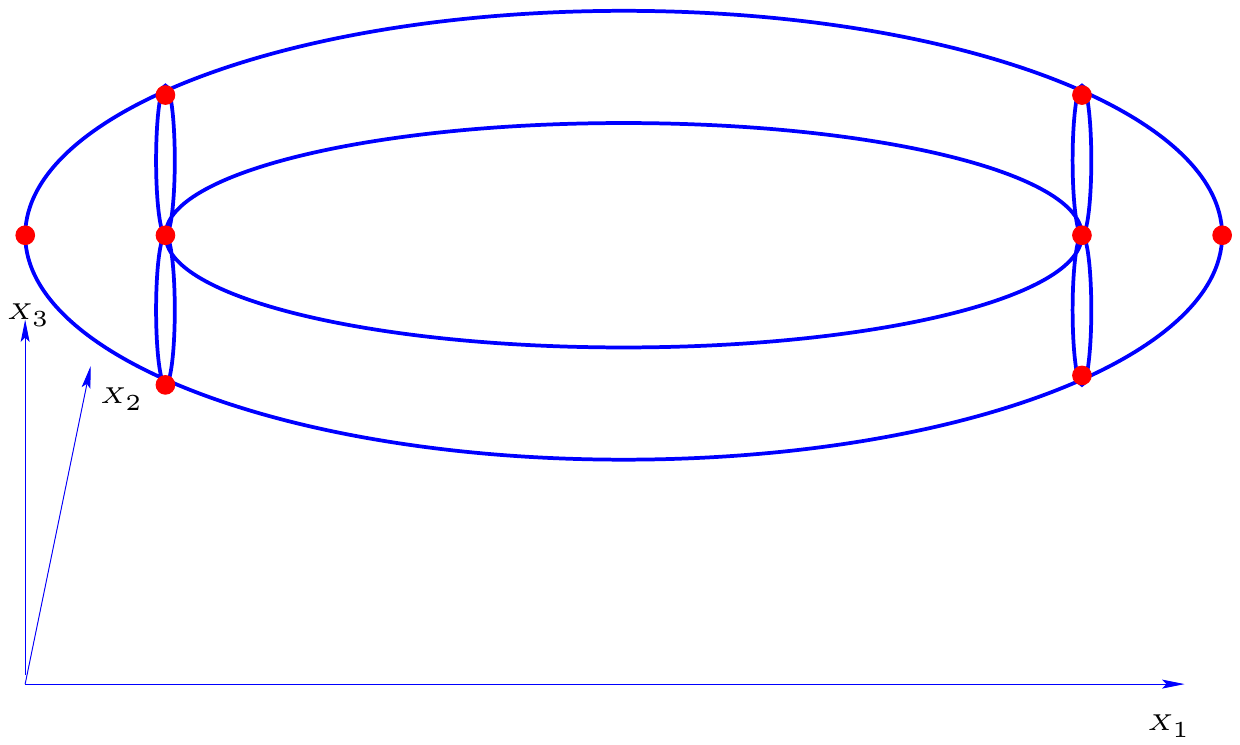}
\vspace{-1.5in}
\caption{A roadmap of the torus in $\mathbb{R}^3$}
\end{figure}

The main result of the paper is the following theorem. The notion of
real univariate representations used in the following statements to represent finite sets
of point in $\R^k$ is explained in Section \ref{sec:def}.
  
\begin{theorem}
 \label{the:babygiant}
 Let $Z \subset \R^k$ be an algebraic set defined as the set of zeros
 of a polynomial of degree at most $d \ge 2$ in $k$ variables with
 coefficients in an ordered domain $\D$ contained in a real closed
 field $\R$.
\begin{enumerate}
\item[a)] There exists an algorithm for constructing a roadmap for $Z$
  using $d^{O(k\sqrt{k})}$ arithmetic operations in $\D$.
\item[b)] Moreover, there exists an algorithm that given a finite set
  of points $\mathcal{M}_0 \subset Z$, with cardinality $\delta$, and
  described by real univariate representations of degree at most
  $d^{O(k)}$, constructs a roadmap for $(Z,\mathcal{M}_0)$ using
  $\delta^{O(1)} d^{O(k\sqrt{k})}$ arithmetic operations in $\D$.
\end{enumerate}
\end{theorem}

The following corollary is an immediate consequence of b).

\begin{corollary}
 \label{cor:babygiant}
  Let $Z \subset \R^k$ be an algebraic set defined as the set of zeros
  of a polynomial of degree at most $d \ge 2$ in $k$ variables with
  coefficients in an ordered domain $\D$ contained in a real closed
  field $\R$.
\begin{enumerate}
\item[a)] There exists an algorithm for counting the number of
  s-a connected components of $Z$ which uses $d^{O(k
    \sqrt{k})}$ arithmetic operations in $\D$.
\item[b)] There exists an algorithm for deciding whether two given
  points, described by real univariate representations of degree at
  most $d^{O(k)}$, belong to the same s-a connected
  component of $Z$ which uses $d^{O(k\sqrt{k})}$ arithmetic operations
  in $\D$.
\end{enumerate}
\end{corollary}

\begin{remark}
\label{rem:number}
We can always suppose without loss of generality that the zero set of
a family of polynomials of degree at most $d$ is defined by one single
polynomial of degree at most $2d$ by replacing the input polynomials
by their sum of squares.
\end{remark}

\begin{remark}
\label{rem:rcf}
Even if the input is a polynomial with coefficients in the field of
real numbers, the deformation techniques by infinitesimal elements we
use make it necessary to perform computations on polynomials with
coefficients in some non-archimedean real closed field.  This is the
reason why general real closed fields provide a natural framework for
our work.
\end{remark}


\section{Outline}
\label{sec:outline}

We outline below the classical construction of a roadmap
$\RM(\ZZ(Q,\R^k))$ for a bounded algebraic set $\ZZ(Q,\R^k)$, defined
as the zero set of a polynomial $Q$
contained in 
$\R^k$.  The geometric
ideas yielding this construction are due to Canny.  The description
below is similar to the one in \cite[Chapter 15, Section
  15.2]{BPRbook2}.

A key ingredient of the algorithm is the construction of a particular
finite set of points intersecting every s-a connected
component of $\ZZ(Q,\R^k)$.  In the case of a bounded and non-singular
real algebraic set in $\R^k$, these points are
nothing but the set of critical points of the projection to the
$X_1$-coordinate on $\ZZ(Q,\R^k)$.  In more general situations, the
points we consider are called $X_1$-pseudo-critical points, since they
are obtained as limits of the critical points of the projection to the
$X_1$-coordinate of a bounded nonsingular algebraic hypersurface
defined by a particular infinitesimal deformation of the polynomial
$Q$.  Their projections on the $X_1$-axis are called pseudo-critical
values.

We first construct the ``silhouette'' which is the set of
$X_2$-pseudo-critical points on $\ZZ(Q,\R^k)$ along the $X_1$-axis by
following continuously, as $x$ varies on the $X_1$-axis, the
$X_2$-pseudo-critical points on $\ZZ(Q,\R^k)_{x}$.  
Note that in case $\ZZ(Q,\R^k)$ is a non-singular hypersurface, then the 
``silhouette'' described above is the set of critical points of the projection map to the coordinates
$X_1$ and $X_2$. However, we are not assuming here that $\ZZ(Q,\R^k)$ is a non-singular hypersurface.
This results in
curves and their endpoints on $\ZZ(Q,\R^k)$. The curves are continuous
s-a curves parametrized by open intervals on the $X_1$-axis
and their endpoints are points of $\ZZ(Q,\R^k)$ above the
corresponding endpoints of the open intervals.  Since these curves and
their endpoints include for every $x\in\R$ the $X_2$-pseudo-critical
points of $\ZZ(Q,\R^k)_{x}$, they meet every s-a
connected component of $\ZZ(Q,\R^k)_{x}$.  Thus, the set 
consisting of these
curves and
their endpoints, already satisfy $\RM_2.$ However, it is clear that
this set might not be s-a connected in a
s-a connected component and so $\RM_1$ might not be
satisfied.

In order to ensure property $\RM_1$ we need to add more curves to the
roadmap. For this purpose, we define the set of {\em distinguished
  values} $\mathcal{D}$ as the union of the $X_1$-pseudo-critical
values, and the first coordinates of the endpoints of the curves
described in the previous paragraph.  A {\em distinguished hyperplane}
is a hyperplane defined by $X_1=v$, where $v$ is a distinguished
value. The input points, the endpoints of the curves, and the
intersections of the curves with the distinguished hyperplanes define
the set of {\em distinguished points}, $\mathcal{M}$.

Let the distinguished values be $v_1<\ldots <v_N.$ Note that amongst
these are the $X_1$-pseudo-critical values. Above each interval $(v_i,
v_{i+1})$ we have constructed a collection of curves $\mathcal{C}_i$
meeting every s-a connected component of
$\ZZ(Q,\R^k)_v$ for every $v \in (v_i, v_{i+1})$. Above each
distinguished value $v_i$ we have a set of distinguished points
$\mathcal{M}_i$.
  Each curve in $\mathcal{C}_i$ has an endpoint in
$\mathcal{M}_i$,
and another one in   $\mathcal{M}_{i+1}$.  
Moreover, the
union of the 
$\mathcal{M}_i$ contains $\mathcal{M}$.  
We denote by
$\mathcal{C}$ the union of the $\mathcal{C}_i$.

The following key connectivity result is proved in \cite[Lemma
15.9]{BPRbook2posted2}.
 
\begin{proposition}
  \label{prop:ordinary}  
  Let $\mathcal{R} =\mathcal{C} \cup {\rm Zer} (Q,
  \R^k)_{\mathcal{D}}$.  If $P$ is a s-a connected
  component of~${\rm Zer} (Q, \R^k),$ then $\mathcal{R} \cap P$ is
  s-a connected.
\end{proposition}

Thus, in order to construct a roadmap of $\ZZ(Q,\R^k)$ it suffices to
repeat the same construction in each distinguished hyperplane $H_i$
defined by $X_1=v_i$ with input $Q(v_i,X_2,\ldots,X_k)$ and the
distinguished points in $\mathcal{M}_{v_i}$ by making recursive calls
to the algorithm.  The following proposition is proved in
\cite[Proposition 15.7]{BPRbook2}.

\begin{proposition}
\label{prop:rm}
The s-a set $\RM(\ZZ(Q,\R^k),\mathcal{M})$ obtained by this
construction is a roadmap for $\ZZ(Q,\R^k)$ containing $\mathcal{M}$.
\end{proposition}

To summarize, classical roadmap algorithms based on Canny's
construction proceed by first considering the ``silhouette'',
consisting of curves in the $X_1$-direction, and then making recursive
calls to the same algorithm at certain hyperplane sections of
$\ZZ(Q,\R^k)$, so that the dimension of the ambient space drops by $1$
at each recursive call.

The main difference between classical roadmap algorithms and the
algorithms described in \cite{Mohab-Schost2010} and in the current
paper is that instead of considering curves in the $X_1$-direction
and making recursive calls to the same algorithm at certain hyperplane
sections of $\ZZ(Q,\R^k)$ corresponding to special values of $X_1$, so
that the dimension of the ambient space drops by 1, we consider a
$p$-dimensional subset $W^{(p)}$ of $\ZZ(Q,\R^{k})$ where $1 \leq p \leq k$,
and make recursive calls at certain $(k-p)$-dimensional fibers of
$\ZZ(Q,\R^{k})$, so that the dimension of the ambient space drops by
$p$.

The main topological result, analogous to Proposition
\ref{prop:ordinary}, is that if the set $\ZZ(Q,\R^k)$ satisfies
certain conditions, such as having only isolated singular points, then
the s-a set which is the union of $W$ and these fibers are s-a
connected.  This is proved in Section \ref{sec:connectivity}.  It
follows, though not immediately, that for a general real algebraic
set, $\ZZ(Q,\R^k)$, in order to produce a roadmap of $\ZZ(Q,\R^k)$ it
suffices to compute:
\begin{enumerate}
\item
a roadmap of a certain s-a
subset $W$ of an infinitesimal deformation of $\ZZ(Q,\R^k)$
passing through a carefully chosen set of points, 
and taking the limit of the curves so obtained by letting the perturbarion
variable go to $0$;
\item
roadmaps of certain fibers in a $(k-p)$-dimensional ambient space, using
recursive calls.
\end{enumerate}

The fact that in the new algorithm we are fixing a whole block of $p$
variables at a time necessitates introducing a new kind of algebraic
representation which we call ``real block representation''. This
notion is defined in Section \ref{sec:def}, where we also explain how
to represent curves
over such blocks.

In Section \ref{sec:lowspecial}, we consider the case when $W$ is 
low-dimensional and described by equations having a special structure.
Adapting an
algorithm from \cite[Algorithm 15.3]{BPRbook2} we show how to
compute a roadmap of $W$ in this case with complexity depending in a crucial way
on the dimension of $W$.
The general case, requiring the use of a deformation technique
and a limit process, is described in Section \ref{sec:lowgeneral}.
 
Finally, we obtain in Section \ref{sec:main} a 
baby step - giant step
roadmap
algorithm for a general algebraic set. We prove its correctness, as
well as the improved complexity bound.

The algorithm for computing efficiently limits of curve segments is
quite
technical. Since this 
technicality
can obscure the ideas
behind the main algorithm, for the sake of readability we have
postponed the details behind taking limits of curves to a separate
section (Section \ref{sec:details}).

Throughout the paper, we use as a basic reference \cite{BPRbook2}. We cite
~\cite{BPRbook2posted2} instead when the precise statements needed 
are not included in \cite{BPRbook2}.
 

\section{Connectivity results}
\label{sec:connectivity}

In this section we prove a topological result about connectivity which
will be used in proving the correctness of our algorithm later.  The
statement of the result, as well as the main ideas of the proof, is
influenced by \cite[Theorem 14]{Mohab-Schost2010}. It is a
direct generalization of Proposition \ref{prop:ordinary} to the case
of projection onto more than one variable.

We denote by $\R$ a real closed field.

\begin{notation}
\label{not:projection-and-fibers}
For $ 1 \leq p < k$, we denote by $\pi_p$ the projection 
$$(x_1, \ldots, x_k)\mapsto x_p.$$
For $ 1\leq q \leq p < k$, we denote by $\pi_{[q,p]}:
\R^{k}\rightarrow \R^{p-q+1}$ the projection
$$(x_1, \ldots, x_k)\mapsto (x_q, \ldots, x_p).$$

For any pair of s-a subsets $S \subset \R^k$, and $T \subset
\R^{p}$, we denote by $S_T$ the s-a set
$\pi_{[1,p]}^{-1}(T) \cap S$, and $S_{y}$ rather than $S_{\{y\}}$, for
$y\in \R^{p}$. We also 
write
$S_{<a}$ and $S_{\le a}$ rather than
$S_{(-\infty,a)}$ and $S_{(-\infty,a]}$, for $a\in \R$.
\end{notation}

We denote as before by $\ZZ(Q,\R^k)$ the algebraic set of zeros of a
polynomial $Q\in \R[X_1,\ldots,X_k]$  
contained in
$\R^k$. 
Note that for a non-constant polynomial $Q$, the real dimension of $\ZZ(Q,\R^k)$ is not necessarily equal to $k-1$, and indeed $\ZZ(Q,\R^k)$ might even be empty.
In fact,
over any real closed field, algebraic sets defined by one equation
coincide with general algebraic sets since replacing several equations
by their sum of squares does not modify the zero set.  A $Q$-{\bf
  singular point} is a point $x \in \R^k$ such that
$$Q(x)=\frac{\partial Q}{\partial X_1}(x)=\ldots=\frac{\partial Q}{\partial X_k}(x)=0.
$$ Note that this is an algebraic property related to the  
polynomial
$Q$
rather than a geometric property of the underlying set $\ZZ(Q,\R^k)$:
two equations can define the same algebraic set but have different
sets of singular points.

Similarly a $Q$-{\bf critical point} of $\pi_1$ is a point $x\in\R^k$ such
that
$$
Q(x)=\frac{\partial Q}{\partial X_2}(x)=\cdots=\frac{\partial Q}{\partial X_k}(x)=0.
$$
To simplify notation, when there will be no ambiguity regarding $Q$, we
will simply refer to singular/critical points.

In this paper, we will be using constantly 
the notion of
s-a connected components of a s-a set
\cite[Section 5.2]{BPRbook2}.  Note that, in particular, a
s-a connected component is 
always
non-empty by
definition
(see \cite[Theorem 5.21]{BPRbook2posted2} and the definition of s-a connected components
following it). In particular, the empty set has no s-a
connected component.

\begin{property}
\label{property:special0}
We now consider for $1\leq p < k$ a tuple 
$$
\displaylines{
\left(V,
\mathcal{M},
W^{(p)},\mathcal{M}^{(p)},\mathcal{D}^{(p)}
\right)
}
$$ 
with the following properties:
\begin{enumerate}
\item $V \subset \R^k$ is the union of certain bounded
  s-a connected components of an algebraic set
  $\ZZ(Q,\R^k)\subset \R^k$, such that 
   the $Q$-critical points of the map $\pi_1$ 
  on $V$ 
  (which contains in particular the $Q$-singular points contained in $V$)
  form the finite set $\mathcal{M} \subset V$;
\item $W^{(p)} \subset V$ is a closed s-a set of dimension $p$, such
  that for each $y \in \R^{p}$, $W_y^{(p)}$ (cf. Notation
  \ref{not:projection-and-fibers}) is a finite set of points having
  non-empty intersection with every s-a connected component of $V_y$;
\item $\mathcal{M}^{(p)} \subset V$ is a finite subset  such that the
  intersection of $\mathcal{M}^{(p)} $ with every s-a
  connected component of $W^{(p)}_{a}$ is non-empty, for
  $a\in{\mathcal{D}^{(p)}=\pi_1(\mathcal{M}^{(p)}})$. Moreover, for every
  interval $[a,b]$ and $c\in [a,b]$ with $\{c\}\supset \mathcal{D}^{(p)}
  \cap [a,b]$, if $C$ is a s-a connected component of
  $W^{(p)}_{[a,b]}$, then $C_c$ is a s-a connected component
  of $W^{(p)}_c$.
\end{enumerate}
A tuple 
$$
\displaylines{
\left(V,
\mathcal{M},
W^{(p)}, \mathcal{M}^{(p)},\mathcal{D}^{(p)}
\right)
}
$$ 
is said to satisfy Property \ref{property:special0} if it satisfies the above
properties (1) to (3).
\end{property}

Now we state the main result of this section. It generalizes
Proposition \ref{prop:ordinary} as well as \cite[Theorem
  14]{Mohab-Schost2010}, 
  in the special case when Property \ref{property:special0} holds.

\begin{proposition}
\label{prop:connectivity}
Let
$$
\displaylines{
\left(V,
\mathcal{M},
W^{(p)},\mathcal{M}^{(p)},\mathcal{D}^{(p)}
\right)
}
$$ satisfy Property
\ref{property:special0}, 
$$
\mathcal{N}=\pi_{[1,p]}(\mathcal{M}),
\mathcal{N}^{(p)}=\pi_{[1,p]}(\mathcal{M}^{(p)}), 
$$
and
$$
\displaylines{
{\mathcal{S}}  =  W^{(p)} \cup V_{\mathcal{N}\cup\mathcal{N}^{(p)}}.
}
$$
For every s-a connected component  $C$  of $V$, $C
\cap \mathcal{S}$ is non-empty and s-a connected.
\end{proposition}

\begin{remark}
  In order to understand the situation, the following example of a
  tuple satisfying Property \ref{property:special0} can be useful:
\begin{enumerate}
 \item the torus $V\subset \R^3$ defined as the set of zeros of the
   equation
 $$Q=36(X_1^2+\left(\frac{12
   X_2+5X_3}{13}\right)^2)-
(X_1^2+X_2^2+X_3^2+8)^2,$$ (\cite{BCR}, page
 40, figure 2.5), $p=1$, the four critical points $\mathcal{M} \subset V$
 of the map $\pi_1$ restricted to $V$;
\item the silhouette $W^{(1)} \subset V$ defined by $$Q=\frac{\partial
  Q}{\partial X_3}=0;$$
\item the six critical values $\mathcal{D}^{(1)} \subset \R$ of the map
  $\pi_1$ restricted to the silhouette $W$, and the intersection
  $\mathcal{M}^{(1)}$ of the corresponding six fibers with the silhouette
  $W^{(1)}$.
\end{enumerate}
The tuple
$$ \left(V,
 \mathcal{M},
W^{(1)},\mathcal{M}^{(1)},  \mathcal{D}^{(1)} \right)
$$ 
satisfies Property \ref{property:special0}. 

Finally, ${\mathcal{S}}$ is the union of the silhouette and the
intersection of the torus with the six curves which are the fibers of
$V$ at the distinguished values $\mathcal{D}^{(1)}$.
\end{remark}

\medskip 

The rest of this section is devoted to prove
Proposition~\ref{prop:connectivity}. We need preliminaries about
non-archimedean extensions of the base real closed field~$\R$.

\begin{remark}\label{rem:tarski-seidenberg}
  A typical non-archimedean extension of $\R$ is the field
  $\R\la\eps\ra$ of algebraic Puiseux series with coefficients in
  $\R$, which coincide with the germs of s-a continuous
  functions (see \cite[Chapter 2, Section 6 and Chapter 3, Section
    3]{BPRbook2}).  An element $x\in \R\la \eps\ra$ is bounded over
  $\R$ if $\vert x \vert \le r$ for some $0\le r \in \R$.  The subring
  $\R\la\eps\ra_b$ of elements of $\R\la\eps\ra$ bounded over $\R$
  consists of the Puiseux series with non-negative exponents.  We
  denote by $\lim_{\varepsilon}$ the ring homomorphism from~$\R
  \langle \varepsilon \rangle_b$ to $\R$ which maps $\sum_{i \in
    \mathbb{N}} a_i \varepsilon^{i / q}$ to $a_0$. So, the mapping
  $\lim_{\varepsilon}$ simply replaces $\varepsilon$ by $0$ in a
  bounded Puiseux series.  Given $S\subset \R\la \eps \ra^k$, we
  denote by $\lim_\eps(S)\subset \R^k$ the image by $\lim_\eps$ of the
  elements of $S$ whose coordinates are bounded over $\R$.

  More generally, let $\R'$ be a real closed field extension of $\R$.
  If $S\subset \R^ k$ is a s-a set, defined by a boolean
  formula $\Phi$ with coefficients in $\R$, we denote by $\Ext(S,\R')$
  the extension of $S$ to $\R'$, i.e. the s-a subset of
  $\R'^k$ defined by $\Phi$.  The first property of $\Ext(S,\R')$ is
  that it is well defined, i.e. independent of the formula $\Phi$
  describing $S$ \cite[Proposition 2.87]{BPRbook2}.  Many properties
  of $S$ can be transferred to $\Ext(S,\R')$: for example $S$ is
  non-empty if and only if $\Ext(S,\R')$ is non-empty;  also $S$ is
  s-a connected if and only if $\Ext(S,\R')$ is
  s-a connected \cite[Proposition 5.24]{BPRbook2}.

  Moreover, if Property \ref{property:special0} (2) holds for $V,W$,
  i.e. for every $y \in \R^{p}$, $W_y$ is a finite set of points
  having non-empty intersection with every s-a
  connected component of $V_y$, then Property \ref{property:special0}
  (2) holds for $\Ext(V,\R'),\Ext(W,\R')$, i.e for each $y' \in
  \R'^{p}$, $\Ext(W,\R')_{y'}$ is a finite set of points having
  non-empty intersection with every s-a connected
  component of $\Ext(V,\R')_{y'}$.  Indeed, by Hardt's s-a
  triviality theorem \cite[Theorem 5.45]{BPRbook2}, one can find a
  finite partition of $\R^{p}$ into s-a sets $T_i$,
  $i=1,\ldots,r$, a finite partition of $V_{T_i}$ into s-a
  sets $S_{i,j}$ and an integer $n_i>0$ such that $S_{i,j}$ is
  s-a homeomorphic to $T_i\times (S_{i,j})_{y_i}$ for
  some $y_i\in T_i$, and for all $y\in T_i$, the s-a
  connected components of $V_y$ are $(S_{i,j})_y$ and $W_y$ has $n_i$
  points. By Tarski-Seidenberg's transfer principle \cite[Theorem
    2.80]{BPRbook2}, $\Ext(S_{i,j},\R')$ is s-a
  homeomorphic to $\Ext(T_i,\R')\times \Ext(S_{i,j},\R')_{y_i}$, and
  for all $y'\in \R'^{p}$, there exists $i$ such that $y'\in
  \Ext(T_i,\R')$, the sets $\Ext(S_{i,j},\R')_{y'}$ are the
  s-a connected components of $\Ext(V,\R')_{y'}$ and
  the intersection of $\Ext(W,\R')_{y'}$ and $\Ext(S_{i,j},\R')_{y'}$
  has exactly $n_i$ points.
\end{remark}

Let $V$ be the union
of a subset of the bounded s-a connected components of an
algebraic set $\ZZ(Q,\R^k)\subset \R^k$.  Suppose also that the set
$\mathcal{M}$ of points which are singular points of $\ZZ(Q,\R^k)$ or critical points
of $\pi_1$ on $\ZZ(Q,\R^k)$ and which belong to $V$ is finite. 
We now prove two preliminary results (Lemma \ref{lem:forb} and Lemma \ref{lem:fora} below)
about the pair $(V,\mathcal{M})$ which will be needed in the proof of 
Proposition \ref{prop:connectivity}.

In this paper a {\em s-a path} is a s-a
continuous function $\gamma$ from a closed interval $[a,b]\subset \R$
to $\R^k$.  Note that a s-a set is s-a
connected if and only if it is s-a path connected
\cite[Theorem 5.23]{BPRbook2}.

\begin{lemma}
\label{lem:forb}
Let $C$ be a s-a connected component of $V_{\le b}$ 
such that $C \cap V_b$ is not empty.
\begin{enumerate}
\item  
If $\dim(C) = 0$, $C$ is a point contained in $\mathcal{M}$.
\item  
If $\dim(C) \not= 0$, $C_{<b}$ is non-empty.  Let $B_1,\ldots, B_s$ be
the s-a connected components of $C_{<b}$.  Then,
\begin{enumerate}
\item  
for each $i, 1 \leq i \leq s$, $\overline{B_i}\cap \mathcal{M}\not =
\emptyset$;
\item 
if there exist $i,j, 1 \leq i<j \leq s$ such that $\overline{B_i}\cap
\overline{B_j}\not = \emptyset$, then $\overline{B_i}\cap
\overline{B_j}\subset \mathcal{M}$;
\item  
$\cup_{i=1}^s \overline{B_i}=C$, and hence $\cup_{i=1}^s
  \overline{B_i}$ is s-a connected.
 \end{enumerate}
\end{enumerate}
 \end{lemma}
\begin{proof}
Part 1 follows immediately from \cite[Proposition
  7.3]{BPRbook2posted2}.  We prove Part 2.  Since $\mathcal{M}$
is finite, there is a non-singular point 
$x \in C_b$ 
which is
non-critical for $\pi_1$ on $V$. Let $T_x V$ denote the tangent space
to $V$ at $x$.  So $T_x V$ is not orthogonal to the $X_1$ axis, and
the s-a implicit function theorem~\cite[Theorem
  3.25]{BPRbook2} implies that $C_{<b}$ is non-empty.

Part 2) a) and 2 b) are immediate consequences of Proposition 7.3 in
\cite{BPRbook2posted2}.

We prove 2) c). Clearly, $\cup_{i=1}^s \overline{B_i}\subset C$.
Moreover since $C_{<b}$ is non-empty, $\cup_{i=1}^s \overline{B_i}$ is also non-empty.
Suppose that $x \in A=C \setminus \cup_{i=1}^s \overline{B_i}$.  
For $r> 0$ sufficiently small, 
$\mathcal{B}_k(x,r) \cap C_{<b}=\emptyset$ (where
$\mathcal{B}_k(x,r)$ is the $k$-dimensional open ball of center $x$ and radius
$r$).  Note that $\pi_1(x) = b$, since otherwise $x$ belongs to
$C_{<b}$, and thus to one of the $B_i$'s.

Applying \cite[Proposition 7.3]{BPRbook2posted2}, we deduce from the
fact that $\mathcal{B}_k(x,r)\cap C_{< b}=\mathcal{B}_k(x,r)_{< b}
\cap C=\emptyset$ that $x$ is either a $Q$-singular point, or a
$Q$-critical point of $\pi_1$ on $V$. In other words $x \in
\mathcal{M}$. But since by assumption $\mathcal{M}$ is finite, this
implies that $A$ is a finite
set and is closed.  Since $C$ is s-a connected, and $\cup_{i=1}^s \overline{B_i}$ closed and non-empty, $A$ must be empty. 
\end{proof}

\begin{lemma}
\label{lem:fora}
Suppose that $b \not \in \pi_1(\mathcal{M})$.  Let $C$ be a s-a
connected component of $V_{\leq b}$.  If $a<b$ and $(a,b] \cap
\pi_1(\mathcal{M})$ is empty, then $C_{\leq a}$ is a s-a connected
component of $V_{\leq a}$.
\end{lemma}

\begin{proof} 
We first prove that $C_{\leq a}$ is non-empty. 
Since $C$ is non-empty because it is a 
s-a connected component of $V_{\leq b}$, there must exist 
$a'\in\R$, with $a'\leq b$, such that $C_{\leq a'}$ is non-empty, but
$C_{\leq a''}$ is empty for all $a''<a'$. In this case, 
using Lemma \ref{lem:forb},
$a' \in \pi_1(\mathcal{M})$,
since $(a,b]\cap \pi_1(\mathcal{M})$ is empty, $a'\le a$ and $C_{\leq a}$ is non-empty.

We now show that $C_{\leq a}$ is s-a connected. 
This together
with the fact shown above implies that $C_{\leq a}$ 
a s-a connected component of $V_{\leq a}$.
Let $x$ and $y$ be two points of $C_{\leq a}$ and $\gamma: [0,1] \to C$ be a
s-a path connecting $x$ to $y$ inside $C$.  We want to
prove that there is a s-a path connecting $x$ to $y$ inside
$C_{\le a}$.
 
If ${\rm Im}(\gamma) \subset C_{\le a}$, there is nothing to prove.

If ${\rm Im}(\gamma) \not\subset C_{\le a}$, 
then there exists $c\in \R$ such that for all $d\in\R$ such that
$a<d<c$, 
$ {\rm Im}(\gamma)\cap {\rm Zer} (Q)_{d} \not=\emptyset$. 

Let $\eps$ be a positive infinitesimal.
Then $$\Ext(\gamma([0,1]),\R\la \eps \ra) \cap {\rm Zer} (Q, \R\la
\eps \ra^k)_{a+\eps} \not= \emptyset$$ using \cite[Proposition
  3.17]{BPRbook2} .  Since $$\{u\in [0,1]\subset \R\la \eps\ra \mid
\Ext(\gamma,\R\la \eps \ra)(u)\in {\rm Zer} (Q, \R\la \eps \ra^k)_{<
  a+\eps}\}$$ and $$\{u\in [0,1]\subset \R\la \eps\ra \mid
\Ext(\gamma,\R\la \eps \ra)(u)\in {\rm Zer} (Q, \R\la \eps
\ra^k)_{[a+\eps, b]}\}$$ are s-a subsets of $[0,1]\subset
\R\la \eps \ra$,  there exists by \cite[Corollary 2.79]{BPRbook2} a
finite partition $\mathfrak{P}$ of $[0,1]\subset \R\la \eps \ra$
such that for each open interval $(u,v)$ of $\mathfrak{P}$,
$\Ext(\gamma,\R\la \eps \ra)(u,v)$ is either contained in the set ${\rm Zer}
(Q, R\la \eps \ra^k)_{< a+\eps}$, or in the set ${\rm Zer} (Q, \R\la \eps
\ra^k)_{[a+\eps, b]}$, with $\gamma(u)$ and $\gamma(v)$ in
$\Ext(C,\R\la\eps\ra)_{a+\varepsilon}$.
    
If $\Ext(\gamma,\R\la \eps \ra)(u,v)$ is contained in ${\rm Zer} (Q,
\R\la \eps \ra^k)_{[a+\eps, b]}$, we can replace $\gamma$ by a
s-a path $\gamma'_{[a,b]}$ connecting $\gamma(u)$ to
$\gamma(v)$ inside 
$\Ext(C,\R\la\eps\ra)_{a+\eps}$.  
Note that there is no
$Q$-critical point of $\pi_1$ in $\Ext(V,
\R\la\varepsilon\ra)_{[a+\varepsilon, b]}$ and $\Ext(V,
\R\la\varepsilon\ra)_{[a+\eps, b]}$ contains no $Q$-singular point by
\cite[Proposition 3.17]{BPRbook2} while $\Ext(V, \R\la\eps\ra)\subset
\ZZ(Q, \R\la\eps\ra^k)$ by \cite[Proposition 2.87]{BPRbook2} .
    
By ~\cite[Proposition 15.1 b]{BPRbook2posted2},  if $C'$ is a s-a
connected component of 
$\Ext(V, \R\la\eps\ra)_{[a+\eps, b]}$,
$C'_{a+\eps}$ is a s-a connected component of $\Ext(V,
\R\la\eps\ra)_{a+\eps}$.

Construct a s-a path $\gamma'$ from $x$ to $x'$ inside
$\Ext(C,\R\la\eps\ra)_{\le a+\eps}$, 
obtained by concatenating pieces of $\gamma$ inside
${\rm Zer} (Q, \R\la \eps \ra^k)_{< a+\eps}$ and the paths
$\gamma'_{(u,v)}$ connecting $\gamma(u)$ to $\gamma(v)$ for $(u,v)$
such that $$\Ext(\gamma,\R\la \eps \ra)(u,v)\subset \Ext(V,
\R\la\eps\ra)_{[a+\eps, b]}.$$ Note that such a s-a
connected path $\gamma'$ is closed and bounded.  Applying
\cite[Proposition 12.43]{BPRbook2}, ${\rm lim}_\eps(\gamma'([0,1]))$
is s-a connected, contains $x$ and $x'$ and is
contained in 
${\rm lim}_\eps(\Ext(C,\R\la\eps\ra)_{\le a+\eps})=C_{\le a}$. 
This proves the lemma.
\end{proof}

\begin{notation}
\label{not:ccofx}
If $S \subset \R^k$ is s-a set and $x \in S$, then we denote by 
$\CC(S,x)$ the s-a connected component of $S$ 
containing $x$.
\end{notation}

We are now ready to prove Proposition \ref{prop:connectivity}.

\begin{proof}[Proof of Proposition \ref{prop:connectivity}]
For every $a \in \R$, we say that property ${\bf P}(a)$ holds if:
for any s-a connected component $C$ of $V_{\le a}$, 
$C \cap \mathcal{S}$ is s-a connected.

We prove that for all $a$ in $\R$, the property ${\bf P}(a)$ holds. Since $V$
is bounded, the proposition follows from the property ${\bf P}(a)$ for any $a \ge
\max_{x\in V} \pi_{1}(x)$.

Let $\mathcal{D}=
\pi_{1}(\mathcal{M} \cup \mathcal{M}^{(p)})$.

The proof uses two intermediate results:

\noindent {\bf Step 1}: For every $a\in {\mathcal{D}}$, 
property
${\bf P}(a)$
implies 
property
${\bf P}(b)$ if
$b \in \R$ with $(a,b]\cap {\mathcal{D}}=\emptyset$.

\noindent {\bf Step 2}: For every $b\in {\mathcal{D}}$, if 
property
${\bf P}(a)$ holds for all $a < b$, then 
property
${\bf P}(b)$ holds.

Since for $a < \min_{x\in V} \pi_{1}(x)$, the property ${\bf P}(a)$ holds
vacuously, and the combination of these two results gives by an easy
induction 
the property
${\bf P}(a)$ for all $a$ in $\R$.

We now prove the two steps.

\noindent {\bf Step 1.}  
We suppose that $a\in {\mathcal{D}}$, and that 
the property
${\bf P}(a)$ holds. Take $b \in
\R$, $a<b$ with $(a,b]\cap {\mathcal{D}}=\emptyset$ and prove that
the property
  ${\bf P}(b)$ holds.  Let $C$ be a s-a connected
  component of $V_{\leq b}$. We have to prove that $ C \cap
  \mathcal{S}$ is 
non-empty and 
s-a connected.

Since $(a,b]\cap {\mathcal{D}}=\emptyset$, it follows that
$\mathcal{M}_{(a,b]} = \emptyset$, and $C_{\le a}$ is a s-a connected
component of $V_{\le a}$ using Lemma \ref{lem:fora}.  So, using
property ${\bf P}(a)$, we see that $C_{\le a} \cap \mathcal{S}$ is
non-empty and s-a connected.

If $C_{\le a}\cap \mathcal{S}=C\cap \mathcal{S}$, there is nothing to
prove.  Otherwise, let $x \in C \cap \mathcal{S}$ such that $x \not
\in C_{\le a}$.  We prove that $x$ can be s-a connected
to a point in $C_{\le a}\cap \mathcal{S}$ by a s-a path in
$C \cap \mathcal{S}$, which is enough to prove that $C \cap
\mathcal{S}$ is s-a connected.

Since $\pi_1(x)\in (a,b]$ and $(a,b]\cap \mathcal{D}=\emptyset$, we
deduce that $\pi_1(x)\not \in \mathcal{D}$ and $x\not \in
V_{\mathcal{N}\cup \mathcal{N}^{(p)}}$. So, from $x \in \mathcal{S}$,
we get $x \in W^{(p)}$.  We note that $\CC(W^{(p)}_{[a,b]}, x) \subset
C$.  By Property \ref{property:special0} (3) applied to
$\CC(W^{(p)}_{[a,b]}, x)$ (noting that $(a,b] \cap \mathcal{D}_2
\subset (a,b] \cap \mathcal{D} = \emptyset$) we have that $a \in
\pi_{1}(\CC(W^{(p)}_{[a,b]}, x))$ and $\CC(W^{(p)}_{[a,b]}, x)_{a}$ is
non-empty. Hence there exists a s-a path connecting $x$ to a point in
$\CC(W^{(p)}_{[a,b]}, x)_{a}$ inside $\CC(W^{(p)}_{[a,b]}, x)$.  Since
$\CC(W^{(p)}_{[a,b]}, x) \subset W^{(p)}\subset \mathcal{S}$ and
$\CC(W^{(p)}_{[a,b]}, x)\subset C$, it follows that $
\CC(W^{(p)}_{[a,b]}, x) \subset C \cap \mathcal{S}$ and we are done.

\noindent {\bf Step 2.}  We suppose that $b \in {\mathcal{D}}$, and
that the property
${\bf P}(a)$ holds for all $a < b$. We prove that 
the property
${\bf P}(b)$ holds.

Let $C$ be a s-a connected component of $V_{\le b}$.
If $C_b = \emptyset$, there is nothing to prove.  Suppose that $C_b$ is
non-empty; we have to prove that $C\cap \mathcal{S}$ is
s-a connected.

If $\dim(C) = 0$, $C$ is a point, belonging to $\mathcal{M}\subset
\mathcal{S}$ by Lemma \ref{lem:forb}.  So $C \cap \mathcal{S}$ is
s-a connected.

Hence, we can assume that $\dim(C) > 0$, so that $C_{<b}$ is non-empty
by Lemma \ref{lem:forb}.

Our aim is to prove that $C\cap \mathcal{S}$ is s-a connected.  We do
this in two steps. We prove the following statements:
\begin{enumerate}
\item[{\bf(a)}] If $B$ is a s-a connected component of $C_{<b}$, then
  $\overline{B} \cap \mathcal{S}$ is s-a connected,
  \item[]   and, using {\bf (a)},
  \item[{\bf (b)}] $C\cap \mathcal{S}$ is non-empty and s-a connected.
\end{enumerate}

\noindent {\bf Proof of (a)} We prove that if $B$ is a
s-a connected component of  
$C_{<b}$, 
then $\overline{B}
\cap \mathcal{S}$ is s-a connected.

Since 
$\overline{B}$ 
contains a point of $\mathcal{M}$ it follows
that $\overline{B}\cap \mathcal{S}$ is not empty.

Note that if $\overline{B}\cap \mathcal{S}=B\cap \mathcal{S}$, then
there exists $a$ with
$$
\max(\{\pi_1(x)\mid x \in B \cap \mathcal{S}\})<a<b,
$$
with $B\cap \mathcal{S}=(B\cap \mathcal{S})_{\le a}$ and
$B_{\le a}$ s-a connected using Lemma \ref{lem:fora}.
So $B\cap \mathcal{S}$  is s-a connected 
since 
the property
${\bf P}(a)$ holds.

We now suppose that $(\overline{B}\setminus B)\cap \mathcal{S}$ is
non-empty.  Taking $x \in (\overline{B}\setminus B)\cap \mathcal{S}$,
we are going to show that $x$ can be connected to a point $z$ in $B
\cap \mathcal{S}$ by a s-a path $\gamma$ inside
$\overline{B}\cap \mathcal{S}$. Notice that $\pi_1(x)=b$.

We first prove that we can assume without loss of generality that $x
\in W^{(p)}$.  Otherwise, since $x \in \mathcal{S}$ and ${\mathcal{S}} = W^{(p)}
\cup V_{\mathcal{N}\cup\mathcal{N}^{(p)}}$, we must have that $x \in
V_y$ with $y=\pi_{[1,p]}(x)$, and $V_y \subset \mathcal{S}$.  Let $A =
\CC(V_y \cap \overline{B},x)$.  We now prove that $A \cap W^{(p)}_y \neq
\emptyset$.  Using the Curve Selection Lemma 
\cite[Theorem 2.5.5]{BCR}
 choose a s-a
path $\gamma: [0,\eps] \rightarrow \Ext(\overline{B},\R\la\eps\ra)$
such that $\gamma(0) = x$, $\lim_\eps \gamma(\eps) = x$ and
$\gamma((0,\eps]) \subset \Ext(B,\R\la\eps\ra)$.  Let $y_\eps =
  \pi_{[1,p]}(\gamma(\eps))$ and
\[
A(\eps) = \CC(\Ext(B,\R\la\eps\ra)_{y(\eps)},\gamma(\eps)).
\] 
Note that $x \in \lim_\eps A(\eps) \subset A$.

By Remark \ref{rem:tarski-seidenberg}, $\Ext(B, \R\la\eps\ra)$ is a
s-a connected component of  
$\Ext(V_{<b}, \R\la\eps\ra)$
which implies that $A(\eps)$ is a s-a connected
component of $\Ext(V, \R\la\eps\ra)_{y(\eps)}$.  
By Property \ref{property:special0} (2) and Remark
\ref{rem:tarski-seidenberg}, $\Ext(W^{(p)},\R\la\eps\ra)_{y(\eps)}
\cap A(\eps) \neq \emptyset$.  Then, since
$\Ext(W^{(p)},\R\la\eps\ra)_{y(\eps)} \cap A(\eps)$ is bounded over
$\R$, we deduce that 
\[
\lim_\eps (\Ext(W^{(p)},\R\la\eps\ra)_{y(\eps)} \cap A(\eps))
\]
is a non-empty subset of $W^{(p)}_y \cap A$.

Now connect $x$ to a point  $x' \in W^{(p)}_y$ by a s-a path
whose image is contained in $A\subset \overline{B}_y\subset
(\overline{B}\setminus B)\cap \mathcal{S}$.  Thus, replacing $x$ by
$x'$ if necessary we can assume that $x \in W^{(p)}$ as 
claimed.

There are four cases, namely
\begin{enumerate}
\item
$x \in \mathcal{M} \cup \mathcal{M}^{(p)}$;
\item
$x \not\in \mathcal{M} \cup \mathcal{M}^{(p)}$ and 
$\CC(W^{(p)}_b,x) \not\subset \overline{B}$;
\item
$x \not\in \mathcal{M} \cup \mathcal{M}^{(p)}$,
$\CC(W^{(p)}_b,x) \subset \overline{B}$ and $b \in \mathcal{D}^{(p)}$;
\item
$x \not\in \mathcal{M} \cup \mathcal{M}^{(p)}$,
$\CC(W^{(p)}_b,x) \subset \overline{B}$ and $b \not\in \mathcal{D}^{(p)}$;
\end{enumerate}
that we consider now.

\begin{enumerate}
\item $x \in \mathcal{M} \cup \mathcal{M}^{(p)}$: \\ Define
  $y=\pi_{[1,p]}(x)\in \R^{p}$, and note that $V_y\subset
  \mathcal{S}$.  Since $x \in \overline{B}$, 
  and $B$ is bounded,
  $y\in
  \pi_{[1,p]}(\overline{B})= \overline{\pi_{[1,p]}(B)}$.  Now let
  $\eps > 0$ be an infinitesimal.  By applying the Curve Selection
  Lemma \cite[Theorem 2.5.5]{BCR} to the set $B$ and $x \in \overline{B}$, and then projecting
  to $\R^{p}$ using $\pi_{[1,p]}$ we obtain that there exists $y(\eps)
  \in \R\langle\eps\rangle^{p}$ infinitesimally close to $y$ with
$\pi_1(y(\eps)) < \pi_1(y)=b$, 
and $x \in
  \lim_\eps\Ext(V,\R\langle\eps\rangle)_{y(\eps)}$. 
  Let $x(\eps) \in \Ext(V,\R\langle\eps\rangle)_{y(\eps)}$ be such that 
$\lim x(\eps)=x$.
  Moreover, by
  Property \ref{property:special0} (2) and Remark
  \ref{rem:tarski-seidenberg} we have that
  $\Ext(W^{(p)},\R\langle\eps\rangle)_{y(\eps)}$ is non-empty and meets every
  s-a connected component of
  $\Ext(V,\R\langle\eps\rangle)_{y(\eps)}$. 
Note that
\[
\Ext(V,\R\la\eps\ra)_{y(\eps)} = \Ext(V_{\leq b},\R\la\eps\ra)_{y(\eps)} = \Ext(V_{< b},\R\la\eps\ra)_{y(\eps)}.
\]
Let
\[
x'(\eps) \in \Ext(W^{(p)},\R\langle\eps\rangle)_{y(\eps)}\cap
\CC(\Ext(B,\R\la\eps\ra)_{y(\eps)},x(\eps)),
\]
and   
$x' = \lim_\eps x'(\eps)$.  
Since 
$\lim_\eps x(\eps)=x$
and
$\lim_\eps \;\CC(\Ext(B,\R\la\eps\ra)_{y(\eps)},x(\eps))$ is
s-a connected,
$$ \lim_\eps \CC(\Ext(B,\R\la\eps\ra)_{y(\eps)},x(\eps))
\subset\CC(\overline{B}_y,x).
$$ Now choose a s-a path $\gamma_1$ connecting $x$ to $x'$
inside $\CC(\overline{B}_y,x)$ (and hence inside $\mathcal{S}$ since $
\CC(\overline{B}_y,x) \subset V_y\subset \mathcal{S}$).  Since
$x'(\eps)$ has coordinates which are algebraic Puiseux
series in $\eps$, there exists a positive element
$t_0 \in \R$, and a s-a curve $\gamma_2:[0,t_0] \rightarrow \R^k$
defined over $\R$, such that $x' = \gamma_2(0)$, and $x'(\eps) \in
\Ext(\gamma_2,\R\la\eps\ra)(\eps)$ (see \cite[Theorem
3.14]{BPRbook2}).
 
The concatenation of
$\Ext(\gamma_1,\R\la\eps\ra),\Ext(\gamma_{2},\R\la\eps\ra)|_{[0,\eps]}$
gives a s-a path $\gamma$
having the required property, after replacing $\eps$ 
by a 
sufficiently small 
positive element of $\R$.

\item $x \not\in \mathcal{M} \cup \mathcal{M}^{(p)}$ and
  $\CC(W^{(p)}_b,x) \not\subset \overline{B}$: \\ There exists $x' \in
  \CC(W^{(p)}_b,x)$, $x' \not \in \overline{B}$ and a s-a path
  $\gamma:[0,1]\rightarrow \CC(W^{(p)}_b,x)$, with $\gamma(0) = x,
  \gamma(1) = x'$. Since $x' \not\in \overline{B}$, it follows that
  for $t_1 = \max \{0 \leq t < 1\;\mid\; \gamma(t) \in
  \overline{B}\}$, $\gamma(t_1) \in \mathcal{M}$.  To see this observe
  that by Lemma \ref{lem:forb} (2c), it follows that
  $\gamma(t_1) \in \overline{B}\cap\overline{B'}$, where $B'$ is a s-a
  connected component of $C_{<b}$ distinct from $B$.  It then follows
  from Lemma \ref{lem:forb} (2b) that $\gamma(t_1)
  \in \mathcal{M}$.  We can now connect $\gamma(t_1)$ to a point in $B
  \cap \mathcal{S}$ by a s-a path inside $\overline{B}\cap
  \mathcal{S}$ using (1).
\item
$x \not\in \mathcal{M} \cup \mathcal{M}^{(p)}$, $\CC(W^{(p)}_b,x) \subset
  \overline{B}$ and $b \in \mathcal{D}^{(p)}$: \\ Since $b \in
  \mathcal{D}^{(p)}$ by Property \ref{property:special0}  
(3)
there exists
  $x' \in \CC(W^{(p)}_b,x) \cap \mathcal{M}^{(p)}$. Thus, there exists a
  s-a path connecting $x$ to $x' \in \mathcal{M}^{(p)}$ with
  image contained in $\overline{B} \cap W^{(p)} \subset \overline{B} \cap
  \mathcal{S}$.  We can now connect $x'$ to a point in $B \cap
  \mathcal{S}$ by a s-a path inside $\overline{B}\cap
  \mathcal{S}$ using (1).

\item
$x \not\in \mathcal{M} \cup \mathcal{M}^{(p)}$, $\CC(W^{(p)}_b,x) \subset
  \overline{B}$ and $b \not\in \mathcal{D}^{(p)}$: \\ Since $b \not\in
  \mathcal{D}^{(p)}$, for all $a < b$ such that $[a,b] \cap \mathcal{D}^{(p)}
  = \emptyset$, $\CC(W^{(p)}_{[a,b]},x)_b = \CC(W^{(p)}_b,x)$ and
  $\CC(W^{(p)}_{[a,b]},x)_a \neq \emptyset$ by Property
  \ref{property:special0} (3).  
  Let $x'\in \CC(W^{(p)}_{[a,b]},x)_a$.
  We can choose a s-a path
  $\gamma: [0,1] \rightarrow \CC(W^{(p)}_{[a,b]},x)$ with $\gamma(0) = x,
  \gamma(1) = x'$. Let 
  \[
  t_1 = \max \{0 \leq t < 1 ];\mid\;\gamma(t) \in W^{(p)}_b\}.
  \]  
  Then, either $\gamma(t_1) \in \mathcal{M}$ and we can
  connect $\gamma(t_1)$ to a point in $B \cap \mathcal{S}$ by a
  s-a path inside $\overline{B}\cap \mathcal{S}$ using
  (1). Otherwise, by Lemma \ref{lem:forb} (2 b), for all small enough
  $r >0$, $\mathcal{B}_k(\gamma(t_1),r) \cap C_{<b}$ is non-empty and contained
  in $B$. Then, there exists $t_2 \in (t_1,1]$ such that $\gamma(t_2)
    \in B\cap W^{(p)} \subset B \cap \mathcal{S}$, and the s-a
    path $\gamma|_{[0,t_2]}$ gives us the required path in this case.
\end{enumerate}

Take $x$ and $x'$ in $\overline{B}\cap \mathcal{S}$.
They can be  connected to points $z$ and $z'$ in $B
\cap \mathcal{S}$ by s-a paths $\gamma$ and $\gamma'$ inside
$\overline{B}\cap \mathcal{S}$ such that, without loss of generality, 
$\pi_1(z)=\pi_1(z')=a$
with $a<b$, and $b-a$ arbitrarily small. Since $B$ is a s-a connected
component of $C_{< b}$, it follows from 
Hardt's s-a triviality theorem \cite[Theorem 5.45]{BPRbook2}
that 
for all $a<b$ with $b-a$ sufficiently small, $B_{\leq a}$ is non-empty and connected,
and hence $B_{\leq a}$ is a s-a connected component of $C_{\leq a}$.
Now, using 
property
${\bf P}(a)$, we conclude that 
property
${\bf P}(b)$
holds.

\noindent {\bf Proof of (b)} We have to prove that $C\cap \mathcal{S}$
is 
non-empty and 
s-a connected.

Since we suppose that $\dim(C)>0$, $C_{<b}$ is non-empty by Lemma
\ref{lem:forb} (2). Let $B$ be a s-a connected component of 
$C_{<b}$. We have from {\bf (a)} that $\overline{B}\cap \mathcal{S}$ is non-empty,
and since $\overline{B}\subset C$, it follows that $C \cap \mathcal{S}$ is not empty. 

We now prove that $C \cap \mathcal{S}$ is s-a connected.
Let $x$ and $x'$ be in $C \cap \mathcal{S}$.  We prove that it is
possible to connect them by a s-a path inside $C \cap
\mathcal{S}$.
Using Lemma \ref{lem:forb} (2c), let $B_i$
(resp. $B_j$) be a s-a connected component of $C_{<b}$
such that $x \in \overline{B_i}$ (resp. $x' \in \overline{B_j}$).

If $i=j$, $x$ and $x'$ both lie in $\overline{B}_i\cap \mathcal{S}$
which is s-a connected by {\bf (a)}. Hence, they can be
connected by a s-a connected path in
$\overline{B}_i\cap\mathcal{S}\subset C\cap\mathcal{S}$.

So let us suppose that $i\not=j$. Note that:
\begin{itemize}
\item by Lemma \ref{lem:forb} (2a), $\overline{B_i} \cap
  \mathcal{M}$ and $\overline{B_j} \cap \mathcal{M}$ are not
  empty,
\item by {\bf (a)} $\overline{B_i}\cap \mathcal{S}$ and $\overline{B_j}\cap
  \mathcal{S}$ are s-a connected,
\item by definition of $\mathcal{S}$, $\mathcal{M} \subset
  \mathcal{S}$.
\end{itemize}
Then, one can connect $x$ (resp. $x'$) to a point in
$\overline{B}_i\cap\mathcal{M}$
(resp. $\overline{B}_j\cap\mathcal{M}$).  This shows that one can
suppose without loss of generality that $x\in
\overline{B}_i\cap\mathcal{M}$ and $x'\in
\overline{B}_j\cap\mathcal{M}$.

Let $\gamma:[0,1]\to C$ be a s-a path that connects $x$ to
$x'$, and let $G=\gamma^{-1}(C \cap \mathcal{M})$ and
$H=[0,1]\setminus G$.

Since $\mathcal{M}$ is finite, we can assume without loss of
generality that $G$ is a finite set of points, and $H$ is a union of a
finite number of open 
or half-open
intervals.

Since $\gamma(G) \subset \mathcal{M} \subset \mathcal{S}$, it
suffices to prove that if $t$ and $t'$ are the end points of an
interval in $H$, then $\gamma(t)$ and $\gamma(t')$ are connected by a
s-a path inside $C \cap \mathcal{S}$.

Notice that $\gamma((t,t')) \cap \mathcal{M}=\emptyset$, so that
$\gamma(t)$ and $\gamma(t')$ belong to the same $\overline{B}_\ell$ by
Lemma \ref{lem:forb} (2b). Recall now that $\gamma(t)$ and
$\gamma(t')$ both lie in $\overline{B}_\ell\cap\mathcal{S}$ and that
$\overline{B}_\ell\cap\mathcal{S}$ is s-a connected by
{\bf (a)}.  Consequently, $\gamma(t)$ and $\gamma(t')$ can be connected by a
s-a path in $\overline{B}_\ell\cap\mathcal{S}\subset
C\cap\mathcal{S}$.
\end{proof}

We are going to need the following corollary.

\begin{corollary}
\label{cor:connectivity}
Let
$$
\displaylines{
\left(V,
\mathcal{M},
W^{(p)},\mathcal{M}^{(p)},\mathcal{D}^{(p)}
\right)
}
$$ satisfy Property
\ref{property:special0}, 
$$
\mathcal{N}=\pi_{[1,p]}(\mathcal{M}),
\mathcal{N}^{(p)}=\pi_{[1,p]}(\mathcal{M}^{(p)}),
$$
and $\mathcal{N}'\subset \R^{p}$ a finite set containing
$\mathcal{N}\cup\mathcal{N}^{(p)}$.  For every s-a
connected component $C$ of $V$,
$$C \cap (W^{(p)} \cup V_{\mathcal{N}'})$$ is 
s-a connected.
\end{corollary}

\begin{proof}
Follows immediately from Proposition \ref{prop:connectivity} and
Property \ref{property:special0} b).
\end{proof}


\section{Block representations and curve segments}
\label{sec:def}

We denote  by $\D$ an ordered domain contained in a real closed field
$\R$ and by $\C$ the algebraically closed field 
$\R[\sqrt{-1}]$.
All the
polynomials in the input and output of our algorithms have
coefficients in $\D$ and the complexity of our algorithms is measured
by the number of arithmetic operations (addition, multiplication, sign
determination) in $\D$.

In this section, we first define certain representations of points,
as well as of s-a curves, that are going to be used in the
inputs and outputs of our algorithms. Several of these representations
share the common property that a certain initial number of coordinates
are fixed by a triangular system of equations, along with certain 
Thom encodings 
and the remaining coordinates are defined by rational functions
to be evaluated at a fixed real root of another polynomial
(see Definitions \ref{def:rur} and \ref{def:curves}  below).
The structure
of these representations reflect the recursive structure of our
main algorithms described in Section \ref{sec:main}.

After defining these representations,  we recall the input, output and
an upper bound on the
complexity of a key algorithm, Algorithm \ref{alg:curvesegments}(Curve
Segments), which is described in full detail in \cite{BPRbook2posted2}.
Algorithm \ref{alg:curvesegments} accepts as input a polynomial
defining a bounded real algebraic variety (with some coordinates
fixed by a triangular system as mentioned above), and outputs a
s-a partition of the first (non-fixed) coordinate, as well
as descriptions of s-a curve segments (as well as points)
parametrized by this coordinate satisfying certain properties --
which are the key to the construction of the main roadmap
algorithm. Indeed, the curve segments appearing in the output of the main roadmap
algorithm (Algorithm \ref{alg:babygiant}) are limits of the curve segments output by the various
calls to Algorithm \ref{alg:curvesegments}.
 
We begin with a few definitions.
\begin{definition}
\label{def:rur}
A {\em Thom encoding $f,\sigma$ representing an element $\alpha\in
  \R$} consists of
\begin{enumerate}
\item a polynomial $f\in \D[T]$ such that $\alpha$ is a root of $f$ in
  $\R$,
\item a sign condition $\sigma$ on the set ${\rm Der}(f)$ of
  derivatives of $f$, such that $\sigma$ is the sign condition
  satisfied by ${\rm Der}(f)$ at $\alpha$.
\end{enumerate}

If $(f,\sigma)$ is a Thom encoding representing an element $\alpha\in \R$, we will
sometimes abuse notation slightly and say that $\sigma$ is the Thom encoding of
the real root $\alpha$ of $f$.  

Distinct roots of $f$ in $\R$ correspond to distinct Thom encodings
\cite[Proposition 2.28]{BPRbook2}.

A {\em real univariate representation $g,\tau,G$ representing $x \in \R^k$} consists of
\begin{enumerate}
\item a Thom encoding $g,\tau$ representing 
an element
$\beta \in \R$,
\item  $G=(g_0,g_{1},\ldots,g_k)\in \D[T]^{k+1}$
where $g$ and $g_0$ are co-prime and such that 
 $$ x=\left(\frac{g_{1}(\beta)}{g_{0}(\beta)},\ldots,
\frac{g_{k}(\beta)}{g_{0}(\beta)}\right) \in \R^k.
$$
\end{enumerate}
\end{definition}

\subsection{Block representations}
\label{subsec:block}
In our algorithms, we make recursive calls, where we fix blocks of
several coordinates.  This makes necessary the following rather
technical definitions.

\begin{definition}
\label{triangthom}
 A {\em triangular Thom encoding
   $\mathcal{F}=(f_{[1]},\ldots,f_{[m]}),\sigma$ representing
   $t=(t_1,\ldots,t_m)$ in $\R^m$} consists of
\begin{enumerate}
\item a triangular system $\mathcal{F}=(f_{[1]},\ldots,f_{[m]})$, i.e.
  $f_{[i]}\in \D[T_1,\ldots,T_i]$ for $i=1,\ldots, m$, such that the
  zero set of $\mathcal{F}$ in $\C^m$ is finite;
\item a list, $\sigma=(\sigma_1,\ldots, \sigma_m)$, where for
  $i=1,\ldots, m$, $\sigma_i$ is the Thom encoding of the root $t_i$
  of $f_{[i]}(t_1,\ldots,t_{i-1},T_i)$.
\end{enumerate}
A triangular system  $\mathcal{F}=(f_{[1]},\ldots,f_{[m]})$ is {\em quasi-monic} if the leading coefficient of 
 $f_{[i]}\in \D[T_1,\ldots,T_i]$
with respect to $T_i$
 is a strictly positive element in $\D$
and  ${\rm deg}_{T_i}(f_{[j]})<{\rm deg}_{T_i}(f_{[i]})$, $j>i$.
A triangular Thom encoding  $\mathcal{F}=(f_{[1]},\ldots,f_{[m]}),\sigma$  is {\em quasi-monic} if
 $\mathcal{F}=(f_{[1]},\ldots,f_{[m]})$  is quasi-monic.

Let $\mathcal{F}=(f_{[1]},\ldots,f_{[m]})$ be a quasi-monic triangular system.
 Let $c\in \D$, $c$ strictly positive and  $g\in \D[T_1,\ldots,T_m]$, we say that {\em $cg$ has a reduction in $\D$ modulo $\mathcal{F}$}
 if there exists a polynomial
$\bar g \in \D[T_1,\ldots,T_m]$ such that
 ${\rm deg}_{T_i}(\bar g)<{\rm deg}_{T_i}(f_{[i]})$ for $i=1,\ldots,m$ and
$$c g = \bar g \mod {\rm I}(\mathcal{F}),$$ where 
${\rm I}(\mathcal{F})$ is the ideal of $\D[T_1,\ldots,T_m]$ generated by 
$f_{[1]},\ldots,f_{[m]}$.
The polynomials $\bar g$ is unique with these properties since a quasi-monic triangular system 
is a Groebner basis with respect to the
lexicographical ordering.
We say that the couple $(c,\bar g)$ is a {\em pseudo-reduction of $g$}.
 
Note that at the zeros of $\mathcal{F}$ the signs of $g$ 
and $\bar g$ coincide.
\end{definition}
\begin{remark}
\label{rem:complexityofreduction}
If $g\in \D[T_1,\ldots,T_m]$ is a polynomial of degree $D$, and $d$ is a bound on the degree of the $f_i$ 
with respect to the $T_i$,
the complexity  of computing 
a pseudo-reduction $(c,\bar g)$ of $g\in \D[T_1,\ldots,T_m]$ modulo $\mathcal{F}$
is $(D d)^{O(m)}$ 
(see Section \ref{sec:details} Proposition \ref{prop:compring} a)).
\end{remark}

\begin{definition}
\label{blockblock}
A {\em real block representation $\mathcal{F},\sigma,L,F$ representing 
$y\in \R^\ell$}
consists of
\begin{enumerate}
\item a triangular Thom encoding
  $\mathcal{F}=(f_{[1]},\ldots,f_{[m]}),\sigma$ representing a root
  $t=(t_1,\ldots,t_m)$ of $\mathcal{F}$ in $\R^m$;
\item a list of natural numbers $L=(\ell_1,\ldots,\ell_m)$ such that
 $$\ell=\ell_1 + \cdots + \ell_m;$$
\item a list of polynomials $F=(F_{[1]},\ldots,F_{[m]})$, where
 \[
F_{[i]} = (f_{[i]0},\ldots,f_{[i]\ell_i}), f_{[i]j} \in 
\D[T_1,\ldots,T_i], 0 \leq j\leq \ell_i,
\]
with
$f_{[i]}(t_1,\ldots,t_{i-1},T_i),f_{[i]0}(t_1,\ldots,t_{i-1},T_i)$
coprime (as polynomials in $T_i$), such that
$$y = (y_{[1]},\ldots,y_{[m]}) \in \R^{\ell},
$$ with
$$\displaylines{ y_{[i]} =
  \left(\frac{f_{[i]1}(t_1,\ldots,t_i)}{f_{[i]0}(t_1,\ldots,t_i)},
  \ldots,
  \frac{f_{[i]\ell_i}(t_1,\ldots,t_i)}{f_{[i]0}(t_1,\ldots,t_i)}
  \right), 1 \leq i \leq m.  }
$$ 
\end{enumerate}

In the case $\ell_1 = \cdots = \ell_m =p $  we will write
\begin{equation}
\label{underbrace}
 L = [p^m].
\end{equation}
\end{definition}

\begin{notation} [Substituting a real block representation in a polynomial] 
\label{not:subst}    
Let $\mathcal{F},\sigma,L,F$ be a real block representation
representing $y\in\R^\ell$, and let $t \in \R^m$ be represented by
$\mathcal{F},\sigma$.

Let

$$\displaylines{
\bar{f}_{[i]} (T_1,\ldots, T_i)= \left( \frac{f_{[i]1}(T_1,\ldots,T_i)}{f_{[i]0}(T_1,\ldots,T_i)}, \ldots,   \frac{f_{[i] \ell_i}(T_1,\ldots,T_i)}{f_{[i]0}(T_1,\ldots,T_i)}
\right).
}
$$ 
Given $Q \in \D [X_1, \ldots, X_k]$ with $\ell \le k$, we 
set
$T=(T_1,\ldots,T_m)$, and define
$Q_F \in \D[T,X_{\ell+1},\ldots,X_k]$ by

\begin{equation}
\label{eqn:subst} 
Q_F := \bar{f}_0(T) Q\left(\bar{f}_{[1]} (T_1), \ldots, \bar{f}_{[m]} (T_1,\ldots, T_m),X_{\ell+1},\ldots, X_k \right),
\end{equation}
where 
$$\bar{f}_0(T)=\prod_{i=1}^m f_{[i]0}(T_1,\ldots,T_i)^{e_i},$$
and $e_i$ is the smallest even number $\ge \deg_{X_{[i]}}(Q)$, where
$X_{[i]}$ is the block of variables
$X_{\ell_1+\cdots+\ell_{i-1}+1},\ldots,X_{\ell_1+\cdots+\ell_i}$.
  
Note that
$$Q_F(t,X_{\ell+1},\ldots,X_k)= \bar{f}_0(t) Q(y,X_{\ell+1},\ldots,X_k),$$
with $\bar{f}_0(t)>0$.
More generally , for any family of polynomials $\mathcal{Q} \subset \D [X_1, \ldots, X_k]$ with $\ell \le k$, we will denote 
$\mathcal{Q}_F = \{Q_{F} \mid Q \in \mathcal{Q}\}$. 
\end{notation}

\begin{notation} [Substituting a real block representation in a parametrized
univariate representation] 
\label{not:subst2}    
Let $\mathcal{F},\sigma,L,F$ be as
above and let 
$g,G$ with $G= (g_0,g_{\ell+1},\ldots,g_k)$,
be a parametrized univariate representation, 
i.e.
$g,g_0,g_{\ell+1},\ldots,g_k \in \D[X_1,\ldots,X_\ell,U]$, where
$X_1,\ldots,X_\ell$ are the parameters, and
$U$ a single variable.

We denote by $G_F$ the tuple $(g_{0,F},\ldots,g_{k,F})$,
where
each $g_{i,F} \in \D[T,U]$ and is defined by
\begin{equation}
\label{eqn:subst2} 
g_{i,F} := \bar{f}_0(T) g_{i}\left(\bar{f}_{[1]} (T_1), \ldots, \bar{f}_{[m]} (T_1,\ldots, T_m),X_{\ell+1},\ldots, X_k \right),
\end{equation}
where 
$$
\bar{f}_0(T)=\prod_{i=1}^m f_{[i]0}(T_1,\ldots,T_i)^{e_i},$$
and $e_i$ is the smallest 
even
number $\ge \max_{j} \deg_{X_{[i]}}(g_{j})$.
\end{notation}

\begin{definition}
\label{def:triangular-thom-encoding}
 Let $t\in \R^m$ be represented by a triangular Thom encoding
 $\mathcal{F},\sigma$.

A {\em Thom encoding $g,\tau$ representing $\beta$ over $t$} consists
of (using the same notation as above)
\begin{enumerate}
\item a polynomial $g \in \D[T_1,\ldots,T_m,T]$ such that
  $g(t,\beta)=0$,
\item a sign condition $\tau$ on $\mathrm{Der}_T(g)$ such that $\tau$
  is the sign condition satisfied by the set $\mathrm{Der}_T(g(t,T))$
  at $\beta$.
\end{enumerate}

A {\em real univariate representation 
  representing $x\in \R^{k}$ over $t$} consists of
\begin{enumerate}
\item a Thom encoding $g,\tau$ representing $\beta$ over $t$,
\item $G=(g_0,g_{1},\ldots,g_k) \in \D[T_1,\ldots,T_m,U]^{k+1}$ such
  that 
  $g(t,U),g_0(t,U)$ are coprime, and such
  that
$$
x=\left(\frac{g_{1}(t,\beta)}{g_{0}(t,\beta)},\ldots,
\frac{g_{k}(t,\beta)}{g_{0}(t,\beta)}\right) \in \R^{k}.
$$
\end{enumerate}
A real univariate representation over $t$ is {\em quasi-monic} if the leading monomial of
 $g$  with respect to $U$ is in $\D$.

A {\em triangular Thom encoding representing $z=(z_1,\ldots,z_r)$ over
  $t$ with variables $Z_1,\ldots,Z_r$} consists of
\begin{enumerate}
\item a triangular system  $\mathcal{H}=(h_{1},\ldots,h_{r})$, with
$$h_{i}\in \D[T_1,\ldots,T_m,Z_{1},\ldots,Z_{i}]$$ for $i=1,\ldots,
  r$, such that the zero set of 
$\mathcal{H}(t,Z_1,\ldots,Z_r)$ in $\C^r$ is finite;
\item a list, $\rho=(\rho_1,\ldots, \rho_r)$, where for $i=1,\ldots,
  r$, $\rho_i$ is the Thom encoding of the root $z_i$ of
  $h_{i}(t,z_1,\ldots,z_{i-1},Z_{i})$.
\end{enumerate}
\end{definition}

\subsection{Curve segments}
\label{subsec:curves}

\begin{definition}
\label{def:curves}
Let $t\in \R^m$ be represented by a triangular Thom encoding
$\mathcal{F},\sigma$.  A {\em curve segment with parameter $X_j$ over
  $t$ on $(\alpha_1,\alpha_2)$ in $\R^{k}$},
$$f_1,\sigma_1, f_2,\sigma_2,g,\tau,G$$ is given by
\begin{enumerate}
\item $\alpha_1, \alpha_2 \in \R$ represented by Thom encodings $f_1,
  \sigma_1$ and $f_2, \sigma_2$ over $t$;
\item a parametrized univariate representation with parameter $X_{j}$, i.e.
  \[ g, G=(g_0, g_{1} ,\ldots, g_k ), \]
  with $g_{j}=X_j g_0 $ and $g, g_{0},\ldots,g_k$ in
  $\D[T_1,\ldots,T_m,X_j,U]$;
\item a sign condition $\tau$ on ${\rm Der}_U (g)$ such that for every
  $x_j \in (\alpha_1, \alpha_2)$ there exists a real root $u(x_j)$ of
  $g (t,x_j,U)$ with Thom encoding $\tau$, and $g_0 (t, x_j, u(x_j))
  \ne 0$.
\end{enumerate}
The {\em curve represented by $f_1,\sigma_1,f_2,\sigma_2,g, \tau,G$}
is the image of the smooth injective s-a function $\gamma$
which maps a point $x_j$ of $(\alpha_1, \alpha_2)$ to the point of
$\R^{k}$ defined by
\[ 
\gamma(x_j) = \left(\frac{g_{1} (t,x_j, u(x_j))}{g_0 (t,x_j, u(x_j))},
   \ldots, \frac{g_k (t,x_j, u(x_j))}{g_0 (t,x_j, u(x_j))} \right) . 
\]
   \end{definition}

Let $Q \in \R[X_1,\ldots,X_k]$.  For 
$0 \leq \ell <k$ 
and $y \in
\R^{\ell}$, we 
write
\begin{equation}
\label{notQy}
 Q(y,-)\defeq Q(y_1,\ldots,y_\ell,X_{\ell+1},\ldots,X_k).
\end{equation}

\begin{remark}
\label{rem:abuse}
Abusing notation slightly, we will occasionally identify
$$\ZZ (Q(y,-),\R^{k-\ell})\subset \R^{k-\ell}$$ 
with 
$$\{y\} \times \ZZ (Q(y,-),\R^{k-\ell})=\ZZ (Q,\R^{k})_y\subset \R^k.$$
More generally, for a s-a set $A\subset \R^k$, 
we will occasionally identify $\{x\in \R^{k-\ell}\mid (y,x)\in A_y\}$ with $A_y\subset \R^k$.
\end{remark}

We now recall the input, output and complexity of \cite[Algorithm 15.2
  (Curve segments)]{BPRbook2posted2}.
\begin{algorithm} {\bf [Curve Segments]}
\label{alg:curvesegments}
\begin{itemize}
\item[{\sc Input.}]
  \begin{enumerate}
  \item a point $t \in \R^m$ represented by a triangular Thom encoding
    $\mathcal{F},\sigma$;
     \item a polynomial $P \in \D [T_1 \ldots, T_{m},X_1,\ldots,X_{k}]$
    for which  $\ZZ (P(t,-), \R^k)$ is bounded;
  \item a finite set of points contained in $\ZZ(P(t,-),\R^{k})$
    represented by real univariate representations $\mathcal{U}$ over
    $t$.
  \end{enumerate}
  Moreover, all the polynomials describing the input are with
  coefficients in $\D$.
\item[{\sc Output.}]
  \begin{enumerate}
  \item An ordered list of points $c_1< \cdots <c_{N}$ of $\R$ with
each $c_i$
    $i=1,\ldots,N$ represented by a Thom encoding $g_i,
    \tau_i$ over $t$.  The $c_i$'s are called {\em distinguished
      values}.
  \item For every $i = 1, \ldots, N$, a finite set of real univariate
    representations $\mathcal{D}_i$ 
over $t$
    representing a finite number of points 
in $\R^k$, called {\em distinguished
      points}.
  \item For every $i = 1, \ldots, N-1$ a finite set of curve segments
    $\mathcal{C}_i$ defined on $(c_i,c_{i+1})$ with parameter $X_{1}$,
    over $t$.  The represented curves are called {\em distinguished
      curves}.
  \item For every $i = 1, \ldots, N-1$ a list of pairs of elements of
    $\mathcal{C}_i$ and $\mathcal{D}_i$ (resp. $\mathcal{D}_{i+1}$)
    describing the adjacency relations between distinguished curves
    and distinguished points.
      
    The distinguished curves and points are contained in $\ZZ
    (P(t,-),\R^{k})$.  The sets of distinguished values, distinguished
    curves, and distinguished points satisfy the following properties.
    \begin{itemize}
    \item[$\mathrm{CS}_1$.] If $c \in \R$ is a distinguished value, the set of 
distinguished points in the  output intersect every
      s-a connected component of
      \[
      \ZZ(P(t,c,-),\R^{k-1}).
      \]

      If $c\in \R$ is not distinguished, the distinguished  curves in the
      output intersect every s-a connected
      component of  
      \[
      \ZZ(P(t,c,-),\R^{k-1}).
      \]
    \item[$\mathrm{CS}_2$.] For each distinguished curve in the output over
      an interval with endpoint a given distinguished value, there
      exists a distinguished point over this distinguished value which
      belongs to the closure of the curve.
    \end{itemize}
  \end{enumerate}
\item[{\sc Complexity.}] If $d = \deg_X(P)\ge 2$, $\deg_T(P)=D$, and
  the degree of the polynomials in $\mathcal{F}$ and the number of
  elements of $\mathcal{U}$ are bounded by $D$, the number of
  arithmetic operations in $\D$ is bounded by $D^{O(m)}d^{O(mk)}$.
  Moreover, the degree in $T_i$ 
of the polynomials appearing in 
the output is bounded by $D d^{O(k)}$.
\end{itemize}
\end{algorithm}

\section{Low dimensional roadmap in a special case}
\label{sec:lowspecial}

In this section we describe an algorithm
for computing the roadmap of a variety described by equations having
a special structure. 
Although, this algorithm is very similar to 
   \cite[Algorithm 15.3 (Bounded Algebraic Roadmap)]{BPRbook2}, 
the complexity analysis differs because of the special 
structure assumed for the input.

Let $Q \in \R[X_1,\ldots,X_k]$ and suppose that $V =\ZZ(Q,\R^{k})$ is bounded.

For   
$0 \leq \ell <k$,  
$0 \le p \leq k-\ell$, and  $y \in \R^{\ell}$, 
we 
write
\begin{equation}
\label{critical}
\Cr_{\ell+p}(Q)(y,-)
\defeq 
\left(
Q(y,-),\frac{\partial Q}{\partial X_{\ell+p+2}}(y,-),
\ldots, \frac{\partial Q}{\partial X_{k}}(y,-)
\right).
\end{equation}

We assume that $Q$ satisfies the following property.
\begin{property}
\label{property:special1}
For every 
$\ell, 0 \leq \ell < k$,  
$0 \le p < k-\ell$, 
and  $y \in \R^{\ell}$, 
the 
algebraic set
$$
\displaylines{
W_{y}^{(p)} =\ZZ(\Cr_{\ell+p}(Q)(y,-),\R^{k-\ell})
}
$$ 
is of dimension $p$ or empty.
\end{property}

\begin{remark}
\label{rem:crit}
Note that for every $y \in \R^{\ell}$, $z \in \R^{r}$,
$(W_{y}^{(r)})_{z}=W_{(y,z)}^{(0)}$ has a finite number of points and,
since $V$ is bounded, intersects
every s-a connected component of $V_{(y,z)}$ by \cite[ Proposition
7.4]{BPRbook2posted2}.
\end{remark}

\begin{proposition}
\label{specialisspecial}
Suppose that $V$ is bounded and $Q$ satisfies Property \ref{property:special1}. Let
\begin{enumerate}
 \item 
$\mathcal{M}=W_{y}^{(0)} \subset V_y$;
\item  $\mathcal{D}^{(p)} \subset \R$ the set of  pseudo-critical values  
(see \cite[Definition 12.41]{BPRbook2}) of $\pi_{\ell+1}$
on 
$W_y^{(p)}$, 
and  $\mathcal{M}^{(p)}$ a set of points such 
that for every 
$c \in \mathcal{D}^{(p)}$,
$\mathcal{M}^{(p)}$ 
intersects every  s-a connected component 
$D$ of 
$(W_y^{(p)})_c$.
\end{enumerate}
Then,
$$
\displaylines{
\left(V_y,\mathcal{M},
W_{y}^{(p)},\mathcal{M}^{(p)},
\mathcal{D}^{(p)}
\right)
}
$$ 
satisfies  Property \ref{property:special0}.
\end{proposition}

\begin{proof}
  Note that by Property \ref{property:special1}, $W_{y}^{(p)}$ is of
  dimension $p$ or empty. In particular,
  $W_{y}^{(0)}$ is finite, i.e. zero-dimensional or empty, and satisfies Property
  \ref{property:special0} 2): for each $z \in \R^{p}$,
  $(W_{y}^{(p)})_{z}=W_{(y,z)}^{(0)}$ is a finite set of
  points having non-empty intersection with every s-a connected
  component of $V_{(y,z)}$ by Remark \ref{rem:crit}.  Moreover,
  $\ZZ(Q(y,-),\R^{k-\ell})$ is clearly bounded (since $\ZZ(Q,\R^k)$ is
  bounded), and the finite set $\mathcal{M}= W_{y}^{(0)}$ is the union
  of the $Q(y,-)$-singular points of $\ZZ(Q(y,-),\R^{k-\ell})$ and the
  $Q(y,-)$-critical points of the map $\pi_{\ell+1}$ on $V_y =
  \ZZ(Q(y,-),\R^{k-\ell})$.  Thus, $\mathcal{M}$ satisfies Property
  \ref{property:special0} 1).

Note also that $\mathcal{M}^{(p)}$
  satisfies Property \ref{property:special0} 3). Indeed, the intersection of $\mathcal{M}^{(p)}$ with
  every s-a connected component of
$(W_y^{(p)})_{c}$
is non-empty, by \cite[Proposition 12.42]{BPRbook2}.
Moreover for every interval $[a,b]$ and $c\in [a,b]$
    such that  $[a,b]$ contains no point of $\mathcal{D}^{(p)}$, except maybe $c$, and for every
    s-a connected component $C$ of 
$( W_y^{(p)})_{[a,b]}$, 
$C_{(y,c)}$
    is a s-a connected component of 
    $(W_{y}^{(p)})_{c}$, by  \cite[Proposition 15.4]{BPRbook2}.
\end{proof}

We are going to describe below, in the special case where $Q$
defines
  a bounded real algebraic set and satisfies Property \ref{property:special1}, an algorithm directly adapted from
\cite[Algorithm 15.3]{BPRbook2} for computing a roadmap of certain
subvarieties of $\ZZ(Q,\R^k)$ of dimension at most $p$: this is
Algorithm \ref{alg:lowspecial} (Roadmap for Lower Dimensional Special
Algebraic Sets).

\begin{remark}
 In all our algorithms, the roadmaps output are represented by a finite number of real univariate representations
 and curve segments over a point defined by a triangular Thom encoding
 (see Definitions \ref{blockblock}, \ref{def:triangular-thom-encoding}, and \ref{def:curves} above).
\end{remark}

\begin{algorithm} {\bf [Roadmap for Lower Dimensional Special
Algebraic Sets]}
\label{alg:lowspecial}
\begin{itemize}
\item [{\sc Input.}]
\begin{enumerate}
\item a polynomial 
$Q\in\D[X_1,\ldots,X_k]$
satisfying Property \ref{property:special1}, and 
for which $V=\ZZ(Q,\R^{k})\subset \mathcal{B}_k(0,1/c)$ (where $c\in \R$);
\item
natural numbers $p,m,r \geq 0$ satisfying 
$0 \leq k- mp \leq p, 0 \leq r < p$;
\item  $y \in \R^{m p}$ represented 
by a real block representation
$\mathcal{F},\sigma,[p^m],F$ (see (\ref{underbrace})) with $t\in \R^m$ represented by 
a quasi-monic triangular system
$\mathcal{F},\sigma$;
 \item $z\in \R^r$ represented by a triangular Thom encoding $\mathcal{H},\rho$ over $t$, with variables
 $X_{mp+1},\ldots,X_{m p+r}$;
\item a finite set of points
 $\mathcal{M}_{0}$ contained in
$$W_{(y,z)}^{(p-r)}=\ZZ(\Cr_{(m+1)p}(Q)(y,z,-),\R^{k-(mp+r)})$$
represented by
 real univariate representations  $\mathcal{U}_{0}$, over $(t,z)$
(using the notation of Property \ref{property:special1}).
\end{enumerate}
\item[{\sc Output.}] a roadmap  
$\RM(W_{(y,z)}^{(p-r)},\mathcal{M}_{0})$ for
$(W_{(y,z)}^{(p-r)},\mathcal{M}_{0})$
represented as a union of curve segment representations and real univariate 
representations  
over points defined by triangular Thom encodings.
The adjacencies between the images of the associated curves and points
are also 
part of the 
output.

\item[{\sc Complexity.}]
$D^{O(m+p)} d^{O((m+p)k)}$
where $d = \deg(Q)\ge 2$, 
and $D$ is a bound on the degree of $\mathcal{H},\mathcal{F},F$ and
the number and degrees of the elements in $\mathcal{U}_{0}$.

\item[{\sc Procedure.}]
\item[Step 0.]  If $mp+r \geq k$ then exit. Else do the following.
\item[Step 1.] Denote 
$$P:=\sum_{A\in \Cr_{(m+1)p}(Q)_F} A^2\in\D[T_1,\ldots,T_m,X_{mp+1},\ldots,X_k]$$ 
using Notation \ref{not:subst} and (\ref{critical}).
Call Algorithm \ref{alg:curvesegments} (Curve Segments)
 with input 
 $$(\mathcal{F},\mathcal{H}),(\sigma,\rho),P, \mathcal{U}_0.$$
Compute a pseudo-reduction of the output 
of the call to Algorithm \ref{alg:curvesegments} (Curve Segments) in the previous step
modulo $\mathcal{F}$ 
(using Proposition \ref{prop:compring}),
and place the result in the description of $\RM(W_{(y,z)}^{(p-r)},\mathcal{M}_{0})$.
\item[Step 2.]  
Using the notation in the output of Algorithm \ref{alg:curvesegments},
for every
$j = 1, \ldots, N$, 
define
$$
\displaylines{
z:=(z,c_j), \cr
\mathcal{H}:=(\mathcal{H},g_j(T_1,\ldots,T_m,X_{1},\ldots,X_{r+i})), \cr
\rho:=(\rho,\tau_j),  \cr
\mathcal{U}_{0} := \mathcal{D}_{j},
}
$$
and call Algorithm \ref{alg:lowspecial} 
(Roadmap for Lower Dimensional Special
Algebraic Sets)
recursively, with input 
$$
(Q, (p,m,r+1),(\mathcal{F},\sigma,[p^m],F) ,(\mathcal{H},\rho), \mathcal{U}_{0}).
$$
Include in the description of $\RM(W_{(y,z)}^{(p-r)},\mathcal{M}_{0})$, the output of 
the call.
\end{itemize}
\end{algorithm}

\vspace{.1in}
\noindent
{\sc Proof of correctness.}
Notice that
$$W_{(y,z)}^{(p-r)}=\ZZ(P(t,z,-)),\R^{k-(m p+r)}).$$
The correctness of the algorithm
then follows from  the correctness of Algorithm \ref{alg:curvesegments} 
(Curve Segments) 
and Proposition \ref{prop:rm}.
The only additional fact that needs to be checked is that when 
the recursion ends with 
$r=k-mp\leq p$, 
the algebraic variety 
$\ZZ(P((t,z,z'),-),\R^{k- p(m+1)})$
is finite, where
$z' = (c_{j_1},\ldots,c_{j_{p-r}}) \in \R^{p-r}$ and
the various $c_{j_i} \in \R$ are associated to the Thom encodings
computed in Step 1 of the algorithm. 
This is the case because 
$ \ZZ(P((t,z,z'),-),\R^{k-p (m+1)}) = W_{(y,z,z')}^{(0)}$,
and
$W_{(y,z,z')}^{(0)}$
is finite 
by Property \ref{property:special1}. 
\eop

\vspace{.1in}
\noindent
{\sc Complexity analysis.}
The depth $i$ of the recursion is bounded by $p-r$,  and the total number of recursive calls at depth $i$ is bounded by $d^{O(ik)}$.
Thus, there are at most $d^{O((p - r) k)}$ calls to  
Algorithm \ref{alg:curvesegments}
(Curve Segments).

In each of the calls to Algorithm \ref{alg:curvesegments} 
(Curve Segments),
the number of arithmetic operations in 
$\D$ is bounded by 
$D^{O(m+p)} d^{O ((m + p) k)}$
using the complexity analysis of 
Algorithm \ref{alg:curvesegments} ( Curve Segments).
Moreover the number of arithmetic operations needed for each pseudo-reduction
is $(D d^{k})^{O(m)}$ 
since the degree in $T_i$ of the output of Algorithm \ref{alg:curvesegments} 
(Curve Segments)
is $D d^{O(k)}$
using Remark \ref{rem:complexityofreduction}.

Thus, the total number of arithmetic
operations in $\D$ for Algorithm  \ref{alg:lowspecial} 
is bounded by $D^{O(m+p)}d^{O((m+p)k)}$.
\eop

\section{Low dimensional roadmap in general}
\label{sec:lowgeneral}
In this section, we first explain how to perform an infinitesimal deformation 
of any given polynomial $Q \in \R[X_1,\ldots,X_k]$ such that
the deformed polynomial satisfies
Property \ref{property:special1}.

We then 
sketch
how to compute the limit of a curve. We also describe 
how to compute the limits of 
roadmaps of 
certain algebraic sets 
which are the critical locus 
of dimension $p$
of certain projection maps restricted to
the algebraic hypersurfaces  obtained after performing an
infinitesimal deformation.

\subsection{Deformation}
\label{subsec:deformation}
We consider a bounded algebraic set defined by a non-negative 
polynomial $Q$. 
Our aim is to define an infinitesimal deformation of $Q$ such that the deformed polynomial satisfies 
Property \ref{property:special1}.
Suppose that the polynomial $Q \in  \R[X_1,\ldots, X_k]$ and 
the tuple
$(d_1, \ldots, d_k)$ satisfy the following 
additional
conditions:
\begin{enumerate}
  \item $d_1 \ge d_2 \ge \cdots \ge d_k$,
  
  \item $\deg (Q) \le d_1,$ $\mathrm{tDeg}_{X_i} (Q) \le d_i$, for $i = 2,
  \ldots, k$,  
where $\mathrm{tDeg}_{X_i} (Q)$ is the maximum degree amongst all monomials in $Q$ containing $X_i$.
\end{enumerate}
Let $\bar{d}_i$ be an even number $> d_i, i = 1, \ldots, k$, and $\bar{d} = (
\bar{d}_1, \ldots, \bar{d}_k)$.
Let 
$$
G_k ( \bar{d}) = X_1^{\bar{d}_1} + \cdots + X_k^{\bar{d}_k} + X_2^2 +
\cdots + X_k^2 + X_{k+1}^2
+ 2k,
$$ and note that 
for all $x\in\R^{k+1}$
$G_k(\bar{d})(x) > 0$.

We denote
for ever $\ell, 0\leq\ell \leq k$ 
and every  $y \in \R^\ell$: 
\begin{notation}
\label{not:defandcr}
\begin{eqnarray*}
\mathrm{Def} (Q, \eps) & = & - \eps G_k ( \bar{d}) + Q,\\
V(\eps)_{y}&=&\ZZ(\Def(Q,\eps)(y,-),\R\langle\eps\rangle^{k+1-\ell}),\\
\mathrm{Cr}_{\ell+p} (Q, \eps)(y,\cdot) \hspace{-0.2cm}&\hspace{-0.2cm}=&\hspace{-0.4cm}\left(
\mathrm{Def}(Q, \eps)(y,\cdot),\frac{\partial  \mathrm{Def}(Q, \eps)}{\partial X_{p + \ell+1}}(y,\cdot),
\ldots, \frac{\partial  \mathrm{Def}(Q, \eps)}{\partial X_{k+1}}(y,\cdot)
\right),\\
W(\eps)_{y}^{(p)}&=&\ZZ(\mathrm{Cr}_{\ell+p} (Q, \eps)(y,-),\R\langle\eps\rangle^{k+1-\ell}).
\end{eqnarray*}
\end{notation}

\begin{proposition}
\label{prop:spespe}
For every 
$\ell, 0 \leq \ell \leq k$ and every $y \in \R^\ell$:
\begin{itemize}
\item[a)]  
$\mathrm{Def} (Q, \eps)$ 
satisfies Property \ref{property:special1};
\item[b)]  $\lim_\eps$  
induces
a  1-1 correspondence between the bounded s-a
connected components of 
 $$V(\eps)_{y}=\ZZ(\Def(Q,\eps)(y,-),\R\langle\eps\rangle^{k+1-\ell})$$
and the s-a connected components of 
$$V_y=\ZZ(Q(y,-),\R^{k-\ell}).$$
\end{itemize}
\end{proposition}

\begin{proof}
a) follows from  \cite[Proposition 12.44]{BPRbook2} and b) from \cite[Lemma 15.6]{BPRbook2}.
\end{proof}

We are going to describe in Section
\ref{subsec:generallow}
an algorithm for computing the limit
of  a roadmap of the critical locus of dimension $p$, 
$W(\eps)_{y}^{(p)}$, of  
$V(\eps)_{y}$. 
In order to achieve this,  we first need to compute
limits of curves, 
which is the purpose of Section \ref{subsec:limits}.

\subsection{Limits of points and curve segments}
\label{subsec:limits}

The general problem of computing the image of a s-a
set $S \subset \R\la\eps\ra^k$ which is bounded over $\R$ under
the $\lim_\eps$ map reduces to the problem of computing 
the closure of a one-parameter family of s-a sets, which
can be done using  quantifier elimination algorithms 
(see, for example, \cite[p. 556]{BPRbook2}). 
However, the complexity of this general algorithm, $d^{k^{O(1)}}$, is
not good enough for our purposes in this paper. 
Fortunately, 
we need efficient algorithms for computing limits 
only in two very special situations, 
where we can do better than 
in the general case.

These two special cases are the following:
\begin{enumerate}
 \item when the set 
consists of
a point represented by a real univariate
representation,
\item when the set 
consists of 
a curve represented by curve segments. 
\end{enumerate}

We 
now describe the input, output and give an upper bound on the complexity 
of  Algorithm \ref{alg:blocklimits} (Limit of a Bounded Point) and Algorithm \ref{alg:limitofacurve} (Limit of a Curve).
A full description of these algorithms, their correctness and complexity analysis
appear in Section \ref{sec:details}.

\begin{algorithm} {\bf [Limit of a Bounded Point]}
\label{alg:blocklimits}
\begin{itemize}
\item[{\sc Input.}]
\begin{enumerate}
 \item  a quasi-monic triangular Thom encoding
$\mathcal{F},\sigma,$ with coefficients in $\D$, representing a point 
$t\in \R^m$; 
\item a real univariate representation
$g(\eps),\tau(\eps),G(\eps)$ over $t$ with coefficients in $\D[\eps]$,
 representing a  point $z(\eps) \in \R\la \eps \ra^p$ 
bounded over $\R$.
 \end{enumerate}
\item[{\sc Output.}] 
\begin{enumerate}
\item a quasi-monic triangular Thom encoding
$\mathcal{F}',\sigma'$, representing the point 
$t \in \R^{m}$;
\item a quasi-monic real univariate representation 
$(h,\tau,H)$
representing
$$z = \lim_\eps z(\eps) \in \R^p.$$
 \end{enumerate}
\item[{\sc Complexity.}]
If $D_1$ (resp. $D_2$) is a bound on the degrees of the polynomials in $\mathcal{F},g(\eps)$ and $G(\eps)$ 
with respect to $T_1,\ldots,T_{m}$ (resp. $\eps,U$),  then
 $D_1$ (resp. $D_2$) is a bound on the degrees of the polynomials  appearing in the output,
and 
the number of arithmetic operations in $\D$
is bounded by $D_1^{O(m)}D_2^{O(1)}$.
\end{itemize}
\end{algorithm}

\begin{remark}
Note that there is a possibly new representation  
of the point $t\in \R^m$ specified in the input, 
in the output of Algorithm
\ref{alg:blocklimits} (Limit of a Bounded Point).
The reason for this peculiarity is explained in Section \ref{sec:details}
(cf. Proposition \ref{prop:compring} b)).
\end{remark}

\begin{algorithm} {\bf [Limit of a Curve]}
\label{alg:limitofacurve}
\begin{itemize}
\item[{\sc Input.}] 
\begin{enumerate}
 \item  a 
 quasi-monic
 triangular Thom encoding $\mathcal{F},\sigma$ with coefficients in $\D$
 representing $t\in \R^m$;
 \item  a triangular Thom encoding 
$\mathcal{H}(\eps),\rho(\eps)$ 
over $t$ with coefficients in
 $\D [\eps]$ representing $z(\eps) \in \R\la \eps \ra^r$ over $t$;
 \item 
 a curve segment
with parameter 
$X_{r+1}$
and coefficients in $\D[\eps]$ over $(t,z(\eps))$, representing a
curve 
$S(\eps)$ in $\R\la\eps\ra^k$
bounded over $\R$.
\end{enumerate}
\item[{\sc Output.}] 
\begin{enumerate}
\item a
real univariate representation 
$p_z,\rho_z,P_z$
of $z=\lim_\eps(z(\eps))$,
with $u$ the root of $p_z$ with Thom encoding $\rho_z$;
\item
a finite set 
$\mathcal{D}= \{d_1,\ldots,d_{N-1}\}$
where each $d_i$ is 
a real univariate representation
over $(t,u,c_i)$, and $c_i$ is given by a Thom encoding over $t$ fixing
$X_{m(i)}$;
\item 
a finite set
$\mathcal{W} = \{w_1,\ldots,w_{N}\}$ 
of curve segments over $(t,u)$
with $w_i$ parametrized by
$X_{\ell(i)}$.

Moreover,
the union of the curves represented by 
$\mathcal{W}$,
together with 
the points represented by 
$\mathcal{D}$
define a partition
of 
$S=\lim_\eps(S(\eps))$.
All the coefficients of the polynomials in the output belong to $\D$.
\end{enumerate}
\item[{\sc Complexity}]
If
the polynomials occurring in 
the input have degrees bounded by 
$D$, then the complexity of the algorithm is bounded by 
$k^{O(1)} D^{O(m+r)}$.
\end{itemize}
\end{algorithm}

\subsection{Low dimensional roadmap algorithm}
\label{subsec:generallow}
We are going to describe 
an algorithm computing the limit of 
a roadmap of the critical locus of dimension $p$, 
$W(\eps)_{y}^{(p)}$, of the deformation 
$V(\eps)_{y}=\ZZ(\Def(Q,\eps)(y,-),\R\langle\eps\rangle^{k+1-\ell})$ of $V_y=\ZZ(Q(y,-),\R^{k-\ell})$.
The algorithm proceeds by first calling 
Algorithm \ref{alg:lowspecial} (Roadmap for Lower Dimensional Special
Algebraic Sets)
 in order to
compute a roadmap for $W(\eps)_{y}^{(p)}$, and then
computes the image of the resulting roadmap
under the $\lim_\eps$ map.
Note that this limit is not necessarily a roadmap of 
$V_y$, since a s-a connected component of $V_y$ might
contain 
the limit
of 
more than one s-a connected 
component of $W(\eps)_{y}^{(p)}$.

\begin{algorithm} {\bf 
[Limit of Roadmaps of Special Low Dimensional Varieties]}
\label{alg:lowgeneral}
\begin{itemize}
\item[{\sc Input.}]
\begin{enumerate}
\item natural numbers $p\le k$;
\item
a polynomial $Q\in\D[X_1,\ldots,X_k]$
for which $Z=\ZZ(Q,\R^{k})\subset \mathcal{B}_k(0,1/c)$ (with $c\in \R$);
\item $y \in \R^{mp}$ represented by the real block  representation 
$$\mathcal{F},\sigma,[p^m],F,$$ (see (\ref{underbrace})) with coefficients in $\D$, such that $t\in \R^m$ is represented by
a quasi-monic triangular Thom encoding
$\mathcal{F},\sigma$; 
\item
a finite set of points 
$\mathcal{N}(\eps)
\in \{y\} \times \R\langle\eps\rangle^{p}$
represented by 
quasi-monic
real univariate representations $\mathcal{V}(\eps)$, 
over $t$.
\end{enumerate}

\item[{\sc Output.}] 
Real univariate representations and curve segments  
representing the
set of points 
$$
\mathcal{R}=(\pi_{[1,k]}\circ \lim_\eps)(\RM(W(\eps)_{y}^{(p)},
{\mathcal A}(\eps)))
$$ 
where
$W(\eps)_{y}^{(p)}$ is the zero set of
$ \mathrm{Cr}_{\ell+p} (Q, \eps)(y,-)$,
\begin{equation}
\label{notationW}
{\mathcal A}(\eps)=\bigcup_{z(\eps)\in \mathcal{N}(\eps)}
W(\eps)^{(0)}_{y,z(\eps)},
\end{equation}
and 
$\RM(W(\eps)_{y}^{(p)} ,{\mathcal A}(\eps)))$ is a roadmap for 
$(W(\eps)_{y}^{(p)} ,{\mathcal A}(\eps)))$.
\item[{\sc Complexity.}]
$D^{O(m+p)} d^{O((m+p)k}$
where $d = \deg(Q)\ge 2$, and $D$ is a bound on the degrees of $\mathcal{F},F$ 
and
the number and degrees 
(including that in $\eps$)
of the elements in 
$\mathcal{N}(\eps)$.
\item[{\sc Procedure.}]
\item[Step 1.]
Let $T=(T_1,\ldots,T_m)$, and using  Notation \ref{not:subst}
and Notation \ref{not:defandcr},
$$P=\sum_{A \in \Cr_{(m+1)p}(Q_,\eps)_F} A^2 \in 
\D[\eps,T,X_{mp+1},\ldots,X_k].$$
Call 
\cite[Algorithm 12.18 (Parametrized Bounded Algebraic Sampling)]{BPRbook2}
with input $P$ and parameters 
$\eps,T,X_{mp+1},\ldots, X_{(m+1)p}$
and  output a set of parametrized univariate representations with variable $U$.

Compute a pseudo-reduction of the output modulo $\mathcal{F}$
(using Pro\-position \ref{prop:compring})
and place the result in 
$\mathcal{U}(\eps)$.

For every 
$(h(\eps),H(\eps)) \in \mathcal{U}(\eps)$, 
and 
every 
$z(\eps) \in \mathcal{N}(\eps)$ 
represented by a
real univariate representation 
$(g(\eps),\tau,G(\eps)) \in \mathcal{V}(\eps)$, 
use \cite[Algorithm 12.20 (Triangular Thom Encoding)]{BPRbook2}
with input the triangular system $(\mathcal{F},g(\eps),h(\eps))$ to compute
the Thom encodings of
the real roots of $h(\eps)(y,z(\eps),U)$.
Let
$\mathcal{U}'(\eps)_{z(\eps)}$ be the set of real univariate representations
over $y,z(\eps)$ so obtained.
Define
$$
\displaylines{
\mathcal{U}'(\eps) = \bigcup_{z(\eps) \in \mathcal{N}(\eps)} 
\mathcal{U}'(\eps)_{z(\eps)}.
}
$$
The set of  points represented by 
$\mathcal{U}(\eps)'$ is  ${\mathcal A}(\eps)$ (see (\ref{notationW})).
Compute the limit of  ${\mathcal A}(\eps)$ 
using
 Algorithm \ref{alg:blocklimits} (Limit of a Bounded Point).

\item[Step 2.]
Call Algorithm 
\ref{alg:lowspecial} (Roadmap for Lower Dimensional Special
Algebraic Sets)
with input 
${\mathrm{Def}}(Q,\eps)$ (cf. Notation \ref{not:defandcr}),
the triple $(p,m,0)$,
the real block representation 
$\mathcal{F},\sigma,[p^m],F$, 
and $\mathcal{U}'(\eps)$ 
(note that since $r=0$ in the input there is no triangular Thom encoding $H,\rho$ specified in the input to this call).
The output of Algorithm 
\ref{alg:lowspecial} 
consists of a set of real univariate representations
and curve segments
over triangular Thom encodings. 
Each such curve segment, 
$\gamma(\eps) = (f_1,\sigma_1,f_2,\sigma_2,g,\tau,G)$,
is defined over some 
$(t,z_\gamma(\eps))$ with $r_\gamma < p$ and 
$z_\gamma(\eps) \in \R\la\eps\ra^{r_\gamma}$, represented
by a triangular system $\mathcal{F},\mathcal{H}_\gamma(\eps)$.

\item[Step 3.]
For each such curve segment $\gamma(\eps)$
over $(t,z_\gamma(\eps))$, output in the previous step, 
call  Algorithm \ref{alg:limitofacurve} (Limit of a Curve) 
with input the triangular system 
$\mathcal{F}, \mathcal{H}_\gamma(\eps)$ and 
$\gamma(\eps)$.
Finally, project to  $\R^k$ by forgetting the last coordinate.
\end{itemize}
\end{algorithm}

\begin{remark}
The role played by the 
set of points 
 ${\mathcal A}(\eps)$
whose limit is computed by
Algorithm
 \ref{alg:lowgeneral}
 (Limit of Roadmaps of Special Low Dimensional Varieties),
 will become clear in 
 the proof of correctness of
 Algorithm \ref{alg:babygiant_bounded}
(Baby-giant Roadmap for
Bounded Algebraic Sets) (see (\ref{reasonforincludingpoints})).
\end{remark}

\vspace{.1in}
\noindent
{\sc Proof of correctness.}
First note that it follows from Proposition \ref{prop:spespe}
that
$\Def(Q,\eps)_F$ 
satisfies Property \ref{property:special1}, and hence 
 ${\mathcal A}(\eps)$
is a finite set of points.
The correctness of the algorithm now follows from the correctness of
Algorithm 
\ref{alg:lowspecial} (Roadmap for Lower Dimensional Special
Algebraic Sets)
and
Algorithm \ref{alg:limitofacurve} (Limit of Curve).
\eop

\vspace{.1in}
\noindent
{\sc Complexity analysis.}
The number of arithmetic operations performed in 
$\D[\eps]$ in Step 1  is bounded by 
$D^{O(m+p)}d^{O((m+p) k)}$ arithmetic operations in $\D[\eps]$
according to the complexity analysis of 
\cite[Algorithm 12.18 (Paramet\-ri\-zed Bounded Algebraic Sampling)]{BPRbook2}
and
\cite[Algorithm 12.20 (Triangular Thom Encoding)]{BPRbook2}.
Since the degree in $\eps$ in the output of \cite[Algorithm 12.18 (Parametrized Bounded Algebraic Sampling)]{BPRbook2}
is $d^{O(k)}$
and does not change during the pseudo-reduction, 
the number of arithmetic operations in $\D$ in Step 1  
and hence the complexity
is bounded by 
$D^{O(m+p)}d^{O((m+p) k)}$.

The number of arithmetic operations performed in $\D[\eps]$ in Step 2 
is bounded by 
$D^{O(m+p)}d^{O((m+p) k)}$
according to the complexity analysis of Algorithm 
\ref{alg:lowspecial} (Roadmap for Lower Dimensional Special
Algebraic Sets).
Moreover the degree in $\eps$ is bounded by $O(d)^{k}$.
To see this one has to observe that the arithmetic operations in  
$\D[\eps]$ in the call to Algorithm \ref{alg:curvesegments} (Curve Segments) 
coincide with those performed by
\cite[Algorithm 15.10 (Parametrized Curve Segments)]{BPRbook2} 
with $\eps$ treated as a parameter.
It follows from the complexity analysis of 
\cite[Algorithm 15.10 (Parametrized Curve Segments)]{BPRbook2} that the 
degree in $\eps$ is bounded by $O(d)^{k}$.

So the number of arithmetic operations in $\D$ 
in Step 2 and hence the complexity
is bounded by $D^{O(m+p)}d^{O((m+p)k)}$.

The complexity of Step 3 is also bounded by 
$D^{O(m+p)}d^{O((m+p) k}$
according to the complexity analysis of 
Algorithm 
\ref{alg:limitofacurve} (Limit of Curve).

Thus the total complexity of the algorithm is 
$D^{O(m+p)}d^{O((m+p) k)}$.
\eop

\section{Main result}
\label{sec:main}
We now describe our main result
Algorithm \ref{alg:babygiant}  (Baby-giant Roadmap for General
Algebraic Sets).
It is based on
Algorithm \ref{alg:babygiant_bounded} (Baby-giant Roadmap for
Bounded Algebraic Sets),
computing a 
baby step - giant step roadmap algorithm for 
a bounded algebraic set. 
The algorithm for computing roadmaps of
general (i.e. not necessarily bounded) algebraic sets,
Algorithm \ref{alg:babygiant}  (Baby-giant Roadmap for General
Algebraic Sets)
is then obtained from 
Algorithm \ref{alg:babygiant_bounded} (Baby-giant Roadmap for
Bounded Algebraic Sets) following a
method similar to the one in \cite{BPRbook2} to go from the bounded case to the general case.

Algorithm \ref{alg:babygiant_bounded} 
(Baby-giant Roadmap for Bounded Algebraic Sets) 
proceeds roughly as follows. 
We denote by $y$ the vector of coordinates which are fixed.
If the number of non-fixed coordinates is too
small (i.e. less than the number $p$ which is prescribed in the input),
then we compute the roadmap using 
\cite[Algorithm 15.3 (Bounded Algebraic Roadmap)]{BPRbook2}.
Otherwise, we compute representations of points in 
 $\mathcal{N}_{y}\subset \R^p$
defining the  fibers at which we make recursive calls
to the same algorithm;
these are the giant steps.

For the 
baby steps, the algorithm uses 
Algorithm \ref{alg:lowgeneral} 
(Limit of Roadmaps of Special Low Dimensional Varieties) 
to compute the limit (under the $\lim_\eps$ map) of the roadmap of 
the critical set $W^{(p)}_{y,\eps}$ 
going through a well chosen finite set of points.

We are now ready to proceed to the description of 
Algorithm \ref{alg:babygiant_bounded}.

We first introduce some notation to be used in the description
of Algorithm \ref{alg:babygiant_bounded} and 
the proof of its correctness.

\begin{notation}
\label{notationalgocor}
 The input of the algorithm involves
 \begin{enumerate}
\item a polynomial $Q\in\D[X_1,\ldots,X_k]$
 such that $V=\ZZ(Q,\R^{k})\subset \mathcal{B}_k(0,1/c)$ (where $c\in \R$);
\item  
$y\in \R^{m p}$  
\item a finite set of points $\mathcal{M}_{0}\subset V_{y}= \ZZ(Q(y,-), \R^{k-mp})$.
\end{enumerate}

Let as in Notation \ref{not:defandcr}
\begin{eqnarray*}
 V(\eps)_{y}&
 =
 &\ZZ(\Def(Q,\eps)(y,-),\R\langle\eps\rangle^{k-pm+1}),\\
W(\eps)_{y}^{(p)}&
=&
\ZZ(\mathrm{Cr}_{\ell+p} (Q, \eps)(y,-),\R\langle\eps\rangle^{k-pm+1}),
\end{eqnarray*}

and  define
\begin{enumerate}
 \item 
$\mathcal{M}(\eps)=W(\eps)_{y}^{(0)} \subset V(\eps)_y$;
\item  
$\mathcal{D}(\eps)^{(p)} \subset \R\la \eps \ra$ the set of  pseudo-critical values  
(see \cite[Definition 12.41]{BPRbook2}) of $\pi_{\ell+1}$
on 
$W(\eps)_y^{(p)}$ 
and  $\mathcal{M}(\eps)^{(p)}$ a set of points such 
that for every 
$c \in \mathcal{D}(\eps)^{(p)}$,
$\mathcal{M}(\eps)^{(p)}$ 
intersects every  s-a connected component 
$D$ of 
$(W(\eps)_y^{(p)})_c$.
\end{enumerate}

It follows from Proposition 
\ref{prop:spespe}
and Proposition \ref{specialisspecial} that
$$
\displaylines{
\left(V(\eps)_{y},
\mathcal{M}(\eps),
W(\eps)_{y}^{(p)},\mathcal{M}(\eps)^{(p)},
\mathcal{D}(\eps)^{(p)}
\right)
}
$$ 
satisfies  Property \ref{property:special0}.

We also define
\begin{eqnarray*}
\mathcal{N}(\eps)&=&\pi_{[mp+1,(m+1)p]}(\mathcal{M}(\eps)),\\
\mathcal{N}&=&\lim_\eps(\mathcal{N}(\eps)),\\
\mathcal{N}(\eps)^{(p)}&=&\pi_{[mp+1,(m+1)p]}(\mathcal{M}(\eps)^{(p)}),\\
\mathcal{N}^{(p)}&=&\lim_\eps(\mathcal{N}(\eps)^{(p)}),\\
\mathcal{N}_{0}&=&\pi_{[mp+1,(m+1)p]}(\mathcal{M}_{0}),\\
\mathcal{N}' &=&\mathcal{N}_{0} \cup \mathcal{N} \cup \mathcal{N}^{(p)},\\
\mathcal{N}'(\eps) &=&
 \mathcal{N}_{0} \cup \mathcal{N}(\eps) \cup \mathcal{N}(\eps)^{(p)}.
\end{eqnarray*}
\end{notation}

\begin{algorithm} {\bf [Baby-giant Roadmap for
Bounded Algebraic Sets]}
\label{alg:babygiant_bounded}
\begin{itemize}
\item[{\sc Input.}]
\begin{enumerate}
\item a polynomial $Q\in\D[X_1,\ldots,X_k]$
 such that $V=\ZZ(Q,\R^{k})\subset \mathcal{B}_k(0,1/c)$ (where $c\in \R$);
\item  
$y\in \R^{m p}$ represented by a real block representation
$$\mathcal{F},\sigma,[p^m],F,$$ (see (\ref{underbrace})) such that $t\in \R^m$ is represented by
a quasi-monic triangular Thom encoding 
$\mathcal{F},\sigma$; 
\item a finite set of points $\mathcal{M}_{0}$ in
$V_{y}= \ZZ(Q(y,-), \R^{k-mp})$ 
represented by
quasi-monic
 real univariate representations  $\mathcal{U}_{0}$
over $t$.
 All the coefficients of the input polynomials are in $\D$.
\end{enumerate}
\item[{\sc Output.}]
a representation of a roadmap, 
$\BGRM(V_{y},\mathcal{M}_{0})$,
for
$(V_{y},\mathcal{M}_{0})$.
\item[{\sc Complexity.}] $d^{O( k^2/p + pk)}$ operations in $\D$ where 
$d = \deg(Q)\ge 2$ and the degrees of the polynomials in 
$\mathcal{F},F$, as well as the degrees of the polynomials
and the number of elements in $\mathcal{U}_{0}$ are all bounded by 
$d^{O(k)}$.
\item[{\sc Procedure.}]

\item[Step 1.]
If $(m+1)p \geq  k$ call 
\cite[Algorithm 15.3 (Bounded Algebraic Road\-map)]{BPRbook2} with input 
\begin{enumerate}
\item the quasi-monic triangular Thom encoding  $\mathcal{F},\sigma$ representing
$t \in \R^m$,
\item the polynomial $Q_F$, using  Notation \ref{not:subst},
\item 
the 
finite set of points $\mathcal{M}_{0}$ in
$V_{y}= \ZZ(Q_F(t,-), \R^{k-mp})$ 
represented by  real univariate representations  $\mathcal{U}_{0}$ over $t$.
\end{enumerate}

Otherwise, 
do the following.
\item[Step 2.]
Determine the finite set of points 
$\mathcal{N}$, and their representation $\mathcal{U}$
used in the recursive call as follows.

Let $T=(T_1,\ldots,T_m)$, and using Notation \ref{not:subst}
and \ref{not:defandcr}),
$$
P=\sum_{A \in \Cr_{(m+1)p}(Q,\eps)_F} A^2 \in 
\D[\eps,T,X_{mp+1},\ldots,X_k].$$

\item[Step 2 a).]
Call 
\cite[Algorithm 12.18 (Parametrized Bounded Algebraic Sampling)]{BPRbook2}
with input $P$ and parameters 
$\eps,T$,
and output a set of para\-metrized univariate representations 
 with variable $U$.
 Compute a pseudo-reduction of the output modulo $\mathcal{F}$
(using Proposition \ref{prop:compring})
 and place the result in 
$\mathcal{U}(\eps)$.

For every $(h(\eps),H(\eps)) \in \mathcal{U}(\eps)$,
use \cite[Algorithm 12.20 (Triangular Thom Encoding)]{BPRbook2}
with input the triangular system $(\mathcal{F},h(\eps))$ to compute
the Thom encodings of
the real roots of $h(\eps)(y,U)$.
Let
$\mathcal{U}'(\eps)$ be the set of real univariate representations
over $y$ so obtained.
Let $\mathcal{M}(\eps)  \subset V(\eps)_{y}$ 
be the set of points represented by 
$\mathcal{U}'(\eps)$.

Projecting $\mathcal{U}'(\eps)$, by forgetting  
its components corresponding to the last $k-(m+1)p$ coordinates
obtain
a set of 
quasi-monic
real univariate representations
$\mathcal{V}(\eps)$
representing $\mathcal{N}(\eps)$
over $t$.
Then apply 
Algorithm \ref{alg:blocklimits}
(Limit of a Bounded Point)
with $\mathcal{V}(\eps)_{\y}$ as input to obtain 
a set of 
quasi-monic
real univariate representations
$\mathcal{V}$
representing $\mathcal{N}$
over $t$.

\item[Step 2 b).]
Perform Algorithm \ref{alg:curvesegments} (Curve Segment) with  
input
$P$
and the 
triangular Thom encoding 
$\mathcal{F},\sigma$
and retain the set of univariate representations,
$\mathcal{U}(\eps)^{(p)}$,
representing
$\mathcal{M}(\eps)^{(p)}\subset V(\eps)_{y}$, 
which are the distinguished points in the output. 

Projecting $\mathcal{U}(\eps)^{(p)}$, by forgetting 
its components corresponding to the last $k-(m+1)p$ coordinates
obtain
a set of real univariate representations
$\mathcal{V}(\eps)^{(p)}$
representing $\mathcal{N}(\eps)^{(p)}$.
Then apply Algorithm \ref{alg:blocklimits} 
(Limit of a Bounded Point)
with $\mathcal{V}(\eps)^{(p)}_{y}$ as input to obtain 
a set of 
quasi-monic
real univariate representations
$\mathcal{V}^{(p)}$
representing $\mathcal{N}^{(p)}$.

\item[Step 2 c).]
Projecting $\mathcal{U}_{0}$, by  
its components corresponding to the last $k-(m+1)p$ coordinates
obtain
a set of quasi-monic  real univariate representations
$\mathcal{V}_{0}$
representing $\mathcal{N}_{0}$
over $t$.

Let 
\[
\mathcal{V}' =
 \mathcal{V}_{0} \cup \mathcal{V}
\cup \mathcal{V}^{(p)},
\]
and
\[
\mathcal{V}'(\eps) =
 \mathcal{V}_{0} \cup 
\mathcal{V}(\eps) \cup \mathcal{V}(\eps)^{(p)}.
\]

\item[Step 3.]
Call Algorithm \ref{alg:lowgeneral} (Limit of Roadmaps of Special Low Dimensional Varieties)
with 
input
$p$, $Q$,
the real block representation 
$\mathcal{F},\sigma,[p^m],F$,
and 
$\mathcal{V}'(\eps)_{y}$ 
and note that 
the one-dimensional s-a set described by the output 
contains the image under the $(\pi_{[1,k]}\circ \lim_\eps)$ map of the finite set
$$
\mathcal{A}(\eps)= \bigcup_{z(\eps)\in \mathcal{N}'(\eps)}W(\eps)^{(0)}_{y,z(\eps)}.
$$
Place the result in the output.
\item[Recursive call.]
For every element $u = ((\mathcal{F},h), (\sigma,\tau),(F,H)) \in 
\mathcal{V}'$, 
representing
$(y,z)\in \mathcal{N}'
\subset \R^{(m+1)p}$, 
determine a set, 
$\mathcal{U}_{z}$,
of 
quasi-monic
univariate representations over
$u$
representing
\begin{equation}
 \label{toremember}
\mathcal{M}_z = 
(\pi_{[1,k]}\circ \lim_\eps) \left( \bigcup_{z(\eps)\in \mathcal{N}'(\eps), \lim_\eps(z(\eps))=z}W(\eps)^{(0)}_{y,z(\eps)}\right)
\end{equation}
using  
Algorithm \ref{alg:blocklimits} (Limit of a Bounded Point).

Call Algorithm \ref{alg:babygiant_bounded} (Baby-giant Roadmap for
Bounded Algebraic Sets) recursively with input
$$
Q,
(\mathcal{F},h),
(\sigma,\tau),
[p^{m+1}],
(F,H), 
$$
and 
$\mathcal{U}_{z}\cup (\mathcal{U}_{0})_{z}$, 
where
$(\mathcal{U}_{0})_{z}$ 
is a set of
quasi-monic
real univariate representations representing
 $(\mathcal{M}_{0})_{z}$.
\end{itemize}
\end{algorithm}

\begin{remark}
Algorithm \ref{alg:babygiant_bounded} would have been much simpler
if we could make recursive calls to 
Algorithm \ref{alg:babygiant_bounded} at the fibers over the points
in 
$\mathcal{N}_{\eps}$,
and thus obtain a roadmap first of 
$V(\eps)_{y}$,
and finally take the image of the resulting roadmap under the
$\lim_\eps$ map. In this case the proof of correctness of the algorithm
would be an immediate consequence of the main connectivity result,
Corollary \ref{specialisspecial}, and the fact that the image
under $\lim_\eps$ of a bounded, s-a connected 
set is also s-a connected.

However, we are unable to compute efficiently 
limits of s-a curves given by curve segments over a real
block representation which depend on $\eps$, if the number of blocks and
their sizes are large. More precisely,  
if the number of blocks as well as the sizes of the blocks are
proportional to $\sqrt{k}$, 
then the procedure that we use to compute limits of curve segments
could produce polynomials with degrees as large as $d^{c k^2}$ in
$\eps$ where $c$ is some constant $>0$. This is unacceptable since we
are aiming for a roadmap algorithm having complexity 
$d^{O(k \sqrt{k})}$.

We overcome this difficulty by making recursive calls to
Algorithm \ref{alg:babygiant_bounded}, not at the fibers over the points in
$\mathcal{N}(\eps)$, but at the fibers over 
$\mathcal{N} = \lim_\eps (\mathcal{N}(\eps))$, so that the algebraic sets
specified in the input to the various recursive calls are then $V_{(y,z)}$ 
for $z \in \mathcal{N}$. In this approach, the only limits of
curve segments that are computed are those of the roadmap of
$W(\eps)_{y}^{(p)}$, and we can compute the limits of these curve
segments without spoiling the complexity,
as they are not defined over real block representations
depending on $\eps$. However,
since the recursive calls are made with fibers
of $V_y$ (instead of $V(\eps)_{y}$),
Corollary \ref{specialisspecial} is not directly applicable, and we need
to be more careful about choosing the set of points in the input to the
recursive calls. It also makes the proof of correctness 
more complicated.
\end{remark}

\vspace{.1in}
\noindent
{\sc Proof of correctness.}

\noindent {\bf Base case.}

If $\lceil (k - mp)/p \rceil = 1$ then
the correctness of the algorithm is a consequence of the correctness
of \cite[Algorithm 15.3 (Bounded Algebraic Roadmap)]{BPRbook2}.

\noindent {\bf General case.}

Suppose that $\lceil (k - mp)/p \rceil > 1$. 

Denote by 
$\BGRM(V_{y},\mathcal{M}_{0})$ 
the union of  the curve segments output
by Algorithm \ref{alg:babygiant_bounded} (Baby-giant Roadmap for
Bounded Algebraic Sets).

We have that
$$
\displaylines{
\BGRM(V_{y},\mathcal{M}_{0})
= \mathcal{R} \cup 
\bigcup_{(y,z) \in  \mathcal{N}'}
\BGRM(V_{(y,z)}, 
\mathcal{M}_{z} \cup (\mathcal{M}_{0})_z),
}
$$
with 
$$
\mathcal{R}=(\pi_{[1,k]} \circ \lim_\eps)
(\RM(W(\eps)_{y}^{(p)}, \mathcal{A}(\eps))),
$$ 
denoting 
as before 
by
$W(\eps)_{y}^{(p)}$ the zero set of
$\mathrm{Cr}_{(m+1)p}(Q, \eps)_F$ and 
$$
\mathcal{A}(\eps)= \bigcup_{z(\eps)\in \mathcal{N}'(\eps)}W(\eps)^{(0)}_{y,z(\eps)}
$$
(see Notation \ref{notationalgocor}).

\noindent {\bf Proof of $\mathcal{M}_{0} \subset \BGRM(V_{y},\mathcal{M}_{0})$.}

The proof is by induction on  $\lceil (k - mp)/p \rceil$.

We suppose by induction 
that for every $(y,z) \in  \mathcal{N}$ 
\[
\mathcal{M}_{z} \cup (\mathcal{M}_{0})_z \subset 
\BGRM(V_{(y,z)}, \mathcal{M}_{z} \cup (\mathcal{M}_{0})_z).
\]

Since
$
\displaystyle{
\mathcal{M}_{0} \subset \bigcup_{(y,z) \in  \mathcal{N}}
(\mathcal{M}_{0})_z
}
$,
and by induction hypothesis we have that 
\[
(\mathcal{M}_{0})_z) \subset 
\BGRM(V_{(y,z)}, 
\mathcal{M}_{z} \cup (\mathcal{M}_{0})_z),
\]
it is clear that 
$\BGRM(V_{y},\mathcal{M}_{0})$ contains $\mathcal{M}_{0}$.

\noindent {\bf Proof of $\RM_1$.}

The property $\RM_1$ of 
$\BGRM(V_{y},\mathcal{M}_{0})$ is also proved
by induction on  $\lceil (k - mp)/p \rceil$. 

Let  $C$ be a s-a connected component of $V_{y}$,
and
$C'=
\BGRM(V_{y}, \mathcal{M}_{0}) \cap C$.
We want to prove that $C'$
is s-a connected.

Supposing that $x,x' \in C'$, we are going to prove that there exists a s-a path
$ \gamma:[0,1] \rightarrow C'$ 
with $\gamma(0)=x,\gamma(1) = x'$. 

\noindent {\bf Without loss of generality we can suppose that $x\in \mathcal{R}$:}

Since
$$
\displaylines{
\BGRM(V_{y},\mathcal{M}_{0})
= \mathcal{R} \cup 
\bigcup_{(y,z) \in  \mathcal{N}'}
\BGRM(V_{(y,z)}, 
\mathcal{M}_{z} \cup (\mathcal{M}_{0})_z),
}
$$
we have that $x$ (resp. $x'$) belongs to 
$\mathcal{R}$ 
or to some
$
\BGRM(V_{(y,z)}, 
\mathcal{M}_{z} \cup (\mathcal{M}_{0})_z)
$ 
with $(y,z) \in \mathcal{N}'$.

If 
$x \in \BGRM(V_{(y,z)}, 
\mathcal{M}_{z} \cup (\mathcal{M}_{0})_z)$
we show that $x$ can  be connected to a point in $\mathcal{M}_{z}$
inside
\[
\BGRM(V_{(y,z)}, 
\mathcal{M}_{z} \cup (\mathcal{M}_{0})_z)
\]
by a s-a path.
It follows from Proposition \ref{specialisspecial} that
$\Def(Q,\eps)_F$ satisfies Property \ref{property:special0} (2), and
hence  we have that 
$W(\eps)_{(y,z)}^{(0)}$
meets every s-a connected component of 
$V(\eps)_{(y,z)}$.
By  \cite[Lemma 15.6]{BPRbook2posted2}
each s-a connected component of $V_{(y,z)}$ 
is the image under 
$\pi_{[1,k]}\circ \lim_\eps$ 
of a unique
s-a connected component of 
$V(\eps)_{(y,z)}$.
It follows that 
each s-a connected component of 
$V_{(y,z)}$ meets
$$\pi_{[1,k]}\circ\lim_\eps(W(\eps)_{(y,z)}^{(0)}) \subset \mathcal{M}_{z},$$ 
since
$$W(\eps)_{(y,z)}^{(0)}=(W(\eps)_{y}^{(p)})_{z}.$$

Finally, applying the induction hypothesis to 
$\BGRM(V_{(y,z)},\mathcal{M}_{z} \cup (\mathcal{M}_{0})_z)$
we have
that the intersection of 
$\BGRM(V_{(y,z)},\mathcal{M}_{z} \cup (\mathcal{M}_{0})_z)$
with each 
s-a connected component of 
$V_{(y,z)}$ is non-empty and s-a connected,
and meets $\mathcal{M}_{z}$.
Thus, there exists a s-a path 
in
$\BGRM(V_{(y,z)},\mathcal{M}_{z} \cup (\mathcal{M}_{0})_z)$
joining $x$ to a point in 
$\mathcal{M}_{z}$.

Since
$\mathcal{M}_{z} \subset \mathcal{R}$
 we can assume that 
$x$ (and similarly $x'$) is contained in $\mathcal{R}$.

\noindent {\bf Connectivity when $x$ and $x'$ are contained in $\mathcal{R}$.}

Since 
$$
\mathcal{R}=(\pi_{[1,k]} \circ \lim_\eps)
(\RM(W(\eps)_{y,}^{(p)},\mathcal{A}(\eps))),
$$ 
 there exists 
$x(\eps) \in \RM(W(\eps)_{y}^{(p)},\mathcal{A}(\eps))$
 (resp.  $x'(\eps) \in \RM(W(\eps)_{y}^{(p)},\mathcal{A}(\eps))$)
such that $\lim_\eps(x(\eps)) = x$ (resp. $\lim_\eps(x'(\eps)) = x'$). 

Let 
\[
\mathcal{S}(\eps) = W(\eps)_{y}^{(p)} \cup 
(V(\eps) _{y})_{\mathcal{N}'(\eps)},
\]
and
 $C(\eps)$ the unique s-a 
connected component of 
$V(\eps)_{y}$
such that 
$(\pi_{[1,k]} \circ \lim_\eps) (C(\eps)) = C$. 

By Corollary \ref{cor:connectivity}, since
$\mathcal{N}(\eps) \cup \mathcal{N}(\eps)^{(p)}\subset \mathcal{N}'(\eps)$, 
$\mathcal{S}(\eps) \cap C(\eps) $ 
is s-a connected. So there exists a 
s-a path 
$\gamma(\eps) : [0,1] \rightarrow \mathcal{S}(\eps)  \cap C(\eps) $, with
$\gamma(\eps) (0) = x(\eps) , \gamma(\eps) (1) = x'(\eps) $.
Moreover, there exists a partition of $(0,1) \subset \R\langle\eps\rangle$
into a finite number of open intervals and points, such that 
for every open interval $I$ in the partition
one of the following holds :
\begin{itemize}
\item[Case 1:]
\[
\gamma(\eps) (I) \subset W(\eps)_{y}^{(p)}.
\] 
\item[Case 2:]
there exists $z(\eps) \in \mathcal{N}'(\eps)$ such that
\[
\gamma(\eps)(I) \subset V(\eps)_{(y,z(\eps))}.
\] 
\end{itemize}

Since $ W(\eps)_{y}^{(p)} \subset V(\eps)_{y}$,
for each
point $a\in (0,1)$ defining the partition
\begin{equation}
\label{reasonforincludingpoints}
\gamma(\eps)(a) \in \mathcal{A}(\eps)\subset \RM(W(\eps)_{y}^{(p)}, 
\mathcal{A}(\eps)).
\end{equation}

Hence, by definition of $\mathcal{M}_{z}$ (see (\ref{toremember}))
\begin{equation}
\label{reasonforincludingpoints2}
(\pi_{[1,k]}\circ \lim_\eps)(\gamma(\eps) (a))\in \mathcal{M}_{z},
\end{equation}
where $\lim_\eps (z(\eps)) =z$.

In Case 1, we can replace $\gamma(\eps)(I)$ by a s-a
path having the same endpoints and whose image is contained in 
$$\RM(W(\eps)_{y}^{(p)}, W_{\mathcal{N}'(\eps)})$$
using 
$\RM_1$ 
as well as (\ref{reasonforincludingpoints}).
Taking the image under $\pi_{[1,k]}\circ \lim_\eps$  of this new path we obtain
a s-a path 
\[
\gamma:\lim_\eps(I) \rightarrow \mathcal{R}.
\]

In Case 2,  
writing $I = (a,b)$,
$(\pi_{[1,k]}\circ \lim_\eps)(\gamma(\eps )(a)), (\pi_{[1,k]}\circ \lim_\eps) (\gamma(\eps)(b))$
both belong to
$$\BGRM(V_{(y,z)},\mathcal{M}_{z} \cup (\mathcal{M}_{0})_z)$$
using  (\ref{reasonforincludingpoints2}).
Using the induction hypothesis for 
$\BGRM(V_{(y,z)},\mathcal{M}_{z} \cup (\mathcal{M}_{0})_z)$
there exists a s-a path 
\[
\gamma:[\lim(\eps)(a),\lim(\eps)(b)] \rightarrow 
\BGRM(V_{(y,z)},\mathcal{M}_{z} \cup (\mathcal{M}_{0})_z).
\]

Finally, we have constructed a s-a path
$
\gamma:[0,1] \rightarrow C' $ 
with $\gamma(0)=x,\gamma(1) = x'$.

This proves that
$\BGRM(V_{(y,z)},\mathcal{M}_{z} \cup (\mathcal{M}_{0})_z) \cap C$
is 
non-empty and 
s-a 
connected proving $\mathrm{RM}_1$. 

\noindent {\bf Proof of  $\mathrm{RM}_2$.}

Let $c \in \R$ such that $V_{(y,c)}$ is not empty, and
let $C$ be a s-a connected component of 
$V_{(y,c)}$.
We prove that
$$\BGRM(V_{(y,z)},\mathcal{M}_{z} \cup (\mathcal{M}_{0})_z) \cap C$$
is not empty.
It follows from \cite[Lemma 15.6 ]{BPRbook2posted2} that 
there exists 
a s-a connected component $C(\eps)$ 
of $V(\eps)_{(y,c)}$ such that 
$$C = (\pi_{[1,k]} \circ \lim_\eps)(C(\eps)).$$ 
Since
$C(\eps)$ is non-empty, let $x(\eps)\in C(\eps)$ and let
$z(\eps )= \pi_{[mp+1,(m+1)p]}(x(\eps))$. It follows from
Proposition \ref{specialisspecial} 
that $(W(\eps)_{y}^{(p)})_{z(\eps)} = W(\eps)^{(0)}_{(y,z(\eps))}$
meets every s-a connected component of $V(\eps)_{(y,z(\eps))}$.
Since $C(\eps)$ contains a s-a
connected component of $V(\eps)_{(y,z(\eps))}$, we have that
$$W(\eps)^{(0)}_{(y,z(\eps))} \cap C(\eps) \neq \emptyset,$$ and thus  
$C(\eps)$ contains a s-a connected
component of $(W(\eps)_{y}^{(p)})_{c}$ (since 
$W(\eps)_{y}^{(p)} \subset V(\eps)_{y}$).
Now,  since the roadmap 
$$
\RM(W(\eps)_{y}^{(p)},\mathcal{A}(\eps))
$$
satisfies
$\mathrm{RM}_2$,
$\RM(W(\eps)_{y}^{(p)},\mathcal{A}(\eps))$ has a non-empty intersection
with every s-a connected component of  
$(W(\eps)_{y}^{(p)})_{c}$, and in particular with the
one contained in $C(\eps)$.
Taking the image under the map $(\pi_{[1,k]}\circ \lim_\eps)$, we get that
$\mathcal{R} = (\pi_{[1,k]}\circ \lim_\eps) (\RM(W(\eps)_{y}^{(p)},\mathcal{A}(\eps)))$
has a non-empty intersection with $(\pi_{[1,k]}\circ \lim_\eps) (C(\eps)) = C$.
Since
$\BGRM(V_{(y,z)},\mathcal{M}_{z} \cup (\mathcal{M}_{0})_z)_{c}$
contains
$\mathcal{R}$, this finishes the proof. 
\eop

\vspace{.1in}
\noindent
{\sc Complexity analysis.}
We first bound 
the number of arithmetic operations in Step 1.
Since we assume that the degrees of the polynomials in $\mathcal{F},F$ are
bounded by $d^{O(k)}$, it follows from the complexity
analysis
of  \cite[Algorithm 15.3 (Bounded Algebraic Roadmap)]{BPRbook2}, 
and  \cite[Algorithm 15.2 (Curve Segments)]{BPRbook2posted2}, that the
number of arithmetic operations in this step is bounded by 
\[
d^{O(k m)} d^{O((k - m p)^2)} = d^{O(km+ p^2)}
\]
since   $k - mp \le p$.

The number of arithmetic operations in $\D[\eps]$ in 
Step 2  is bounded by $d^{O(m k)}$ and the degree and number of univariate
representations produced is bounded by $O(d)^{k-mp}$.
Moreover the degree in $\eps$ is bounded by $O(d)^{k}$.
To see this one has to observe that the arithmetic operations in  
$\D[\eps]$ in the call to Algorithm \ref{alg:curvesegments} (Curve Segments) 
coincide with those performed in   
\cite[Algorithm 15.10 (Parametrized Curve Segments)]{BPRbook2} 
with $\eps$ treated as a parameter. 
It follows from the complexity analysis of 
\cite[Algorithm 15.10 (Parametrized Curve Segments)]{BPRbook2} that the 
degree in $\eps$ is bounded by $O(d)^{k}$.
So the number of arithmetic operations in $\D$
in Step 2  is bounded by $d^{O(m k)}$.

The complexity of computing  
$\mathcal{R}_{y}$ in 
Step 3 is bounded by 
$d^{O((m+p)k)}$  
given that the number of arithmetic operations  of Algorithm \ref{alg:lowgeneral} (Limit of Roadmaps of Special Low Dimensional Varieties)
 is $d^{O((m+p) k)}$.

The total number of recursive calls at depth $i$ is 
 $d^{O(ki)}$,
and for each such call the number of arithmetic operations in $\D$ 
in Steps 1, 2, and 3 is bounded by 
$d^{O((m+i+p)k + p^2 }$, 
where 
$0 \leq i \leq \lfloor{k/p}\rfloor-m$. 
Since the depth of the recursion is at 
most $\lfloor{k/p}\rfloor-m $, 
we conclude that the total number of 
arithmetic operations in the domain $\D$
is bounded by
$$d^{O(k^2/p)}d^{O((k/p+p)k + p^2)} = d^{O( k^2/p + pk)}.$$
\eop

We now describe Algorithm \ref{alg:babygiant} (Baby-giant Roadmap for
General Algebraic Sets) for computing a roadmap of a general (i.e. not
necessarily bounded algebraic set). This algorithm is essentially the
same algorithm as \cite[Algorithm 15.5 (Algebraic Roadmap)]{BPRbook2},
except that we call Algorithm \ref{alg:babygiant_bounded} (Baby-giant
Roadmap for Bounded Algebraic Sets) after reducing to the bounded case
instead of \cite[Algorithm 15.3 (Bounded Algebraic
Roadmap)]{BPRbook2}.  We first need a notation.  Let $P \in \R[X]$ be
given by
$$P=a_pX^p + \cdots+ a_qX^q, \quad p > q,\quad 
a_qa_p \neq 0.$$ 

\begin{notation}
\label{10:def:bigcauchy}
We 
write
$$
\displaylines{
 c(P)=\left(
\sum_{q\leq i\leq p} \left\vert \frac{a_i}{a_q}
\right\vert  \right)^{-1}.
}
$$
\end{notation}

\begin{algorithm} {\bf [Baby-giant Roadmap for General
Algebraic Sets]}
\label{alg:babygiant}
\begin{itemize}
\item[{\sc Input.}]
\begin{enumerate}
\item a polynomial $Q\in\D[X_1,\ldots,X_k]$;
\item a finite set of points  
$\mathcal{M}_{0}$ 
in 
$\ZZ(Q,\R^k)$, represented by real univariate representations $\mathcal{U}_{0}$.
\end{enumerate}
\item[{\sc Output.}] a roadmap, 
$\BGRM(\ZZ(Q,\R^k),\mathcal{M}_{0})$,
for $(\ZZ(Q,\R^k),\mathcal{M}_{0})$.
\item[{\sc Complexity.}] $d^{O( k^2/p + pk)}$ operations in $\D$.
\item[{\sc Procedure.}]
\item[Step 1.]
Introduce new variables
$X_{k+1}$ and
$\eps$ and replace $Q$ by the polynomial
$$Q(\eps)=Q^2+(\eps^2(X_1^2 +\cdots +X_{k+1}^2)-1)^2.$$
Replace $\mathcal{M}_{0} \subset
\R^k$ by the set of real univariate representations
representing the elements of
$\ZZ(\eps^2(X_1^2 +\cdots +X_{k+1}^2)-1,\R\la \eps \ra^{k+1})$
above the points 
$\mathcal{M}_{0}$
using 
\cite[Algorithm 12.18 (Parametrized Bounded Algebraic Sampling)]{BPRbook2}.

\item[Step 2.]
Call
Algorithm \ref{alg:babygiant_bounded} (Baby-giant Roadmap for
Bounded Algebraic Sets)
with input $Q(\eps)$, $\mathcal{M}_{0}$, $m=0$,  
performing arithmetic operations in the domain $\D[\eps]$.
The algorithm
outputs a roadmap 
$$\BGRM(\ZZ(Q(\eps),\R\la \eps \ra^{k+1} ),\mathcal{M}_{0})$$ composed of
points and curves whose description involves $\eps$.

\item[Step 3.]
Denote by $\mathcal{L}$  the set of polynomials in
 $\D[\eps]$ whose signs have been determined
in the preceding computation
and take
$$a = \min_{P \in \mathcal{L}}c(P)$$
(Notation  \ref{10:def:bigcauchy}).
Replace $\eps$ by
$a$ in the polynomial $Q(\eps)$
to get a polynomial $Q_a$. Replace $\eps$
by
$a$ in the output roadmap to obtain a roadmap
$\BGRM(\ZZ(Q_a,\R^{k+1} ),\mathcal{M}_{0}).$
When projected to
$\R^k$, this roadmap gives a roadmap 
$$\BGRM(\ZZ(Q,\R^k),\mathcal{M}_{0})\cap \mathcal{B}_k(0,1/a).$$

\item[Step 4.]
In order to extend the roadmap outside the ball $B(0,1/a)$
collect all the points $(y_1,\ldots,y_k,y_{k+1}) \in \R\la \eps \ra^{k+1}$
in the roadmap 
\[
\BGRM(\ZZ(Q(\eps),
\R\la \eps \ra^{k+1} ),\mathcal{M}_{0})
\]
which satisfy  $\eps(y_1^2 + \ldots + y_k^2) =1$.
Each such point is described by a  real univariate
representation involving $\eps$. Add
to the roadmap the
curve segment obtained by first
forgetting the last coordinate and
then treating $\eps$ as a parameter which varies
over $(0,a]$ to get a roadmap 
$\BGRM(\ZZ(Q,\R^k),\mathcal{M}_{0})$.
\end{itemize}
\end{algorithm}

\vspace{.1in}
\noindent
{\sc Proof of correctness.} 
The proof of correctness follows  from the proof of correctness of
Algorithm \ref{alg:babygiant_bounded} (Baby-giant Roadmap for
Bounded Algebraic Sets).
\eop

\vspace{.1in}
\noindent
{\sc Complexity analysis.}
The complexity is dominated by the complexity of Step 2.
\eop

\begin{proof}[Proofs of Theorem \ref{the:babygiant} and 
Corollary \ref{cor:babygiant}]
Follows directly from the correctness and complexity analysis
of Algorithm \ref{alg:babygiant} (Baby-giant Roadmap for
General Algebraic Sets),  after substituting
$m=0$ and $p = \sqrt{k}$.
\end{proof}

\section{Appendix: computing the limit of bounded points and curve segments}
\label{sec:details}

\subsection{Limit of a bounded point}

Before computing the limit of a bounded point we need to
explain how to perform some useful computations modulo a quasi-monic triangular Thom encoding $\mathcal{F},\sigma$
representing a point $t\in \R^m$.

We associate to $t\in \R^m$ specified by a triangular Thom encoding $\mathcal{F},\sigma$,
$$\mathcal{F}=(f_{[1]},\ldots,f_{[m]}), f_{[i]}\in \D[T_1,\ldots,T_i],$$ 
the ordered domain $\D[t]$ contained in $\R$ and generated by $t$.

We now aim at describing
 the pseudo-inversion of a non-zero element in the domain $\D[t]$ specified by  $\mathcal{F},\sigma$.

\begin{definition}
 A {\rm pseudo-inverse} of $f\in \D[t]$ is an element $g\in \D[t]$ such that $f g\in \D$ is strictly positive.
\end{definition}

 This notion is delicate as the computation of the pseudo-inverse
sometimes
requires us to
update  the 
quasi-monic
triangular Thom encoding specifying $t$, in the spirit of dynamical methods in algebra
(see for example \cite{CLR}).
We start with a motivating example.

\begin{example}
 \label{ex:field}
 We consider $t$, specified as the root of $$f(T)=T^4-T^2-2$$ giving signs $(+,+,+,+)$ to the set ${\rm Der}(f)$ of derivatives of $f$. 
 
Consider $T^2+1$. It is easy to see, using for example  \cite[Algorithm 10.13 (Sign Determination Algorithm)]{BPRbook2} applied to $f$ and the list ${\rm Der}(f),T^2+1$, that the sign of $T^2+1$ at $t$ is positive.
In order to compute its pseudo-inverse, we perform  \cite[Algorithm 8.22 (Extended Signed Subresultant)]{BPRbook2}
of $f$ and $T^2+1$. If $f(T)$ and $T^2+1$ were coprime, we would obtain the pseudo-inverse of $T^2+1$ modulo $f(T)$ since the last subresultant would be a non-zero constant in $\D$. But  $f(T)$ and $T^2+1$ are not coprime and their gcd is $T^2+1$. So we divide $f(T)$ by  $T^2+1$, obtain a new polynomial $g(T)=T^2-2$ and check that the root $t$ of $f(T)$ giving signs $(+,+,+,+)$ to the set ${\rm Der}(f)$
coincides with $\sqrt{2}$ which is the root of $T^2-2$ making the derivative $g'(T)=2T$ positive, using again -for example-
 \cite[Algorithm 10.13 (Sign Determination Algorithm)]{BPRbook2}. It is now possible to 
 compute a pseudo-reduction of  $T^2+1$ modulo $g(T)$, which gives $3$.

In other words, during the process of computing the pseudo-inverse of $T^2+1$ we discovered the factor $g(T)$  of $f(T)$ having $t$ as a root and coprime with $T^2+1$. Using this new description of $t$ we have been able to compute a pseudo-inverse of $T^2+1$.
\end{example}

We can now describe the computation of the pseudo-inverse in general.

\begin{newdescription}
 \label{des:fieldop}
Given $t=(t_1,\ldots,t_m) \in \R^m$ specified by the 
quasi-monic triangular Thom encoding $\mathcal{F}=(f_{[1]},\ldots,f_{[m]}),\sigma=(\sigma_1,\ldots,\sigma_m)$, we describe how to compute a pseudo-inverse of a non-zero element of $\D[t]$.

We proceed by induction on the number $m$ of variables of $\mathcal{F}$.

If $m=0$, there is nothing to do since $\D$ is an ordered domain.

If $m\not= 0$, let $t'=(t_1,\ldots,t_{m-1})$ specified by 
$\mathcal{F}'=(f_{[1]},\ldots,f_{[m-1]}),\sigma=(\sigma_1,\ldots,\sigma_{m-1})$.

We consider $f$ as a polynomial in $T_m$ whose coefficients,
which are elements of 
$$\{h\in \D[T_1,\ldots,T_{m-1}]\mid \deg_{T_i}(h)<\deg_{T_i}(f_{[i]}) , i=1,\ldots,m-1\},$$
represent elements of $\D[t']$.

We first decide the sign of $f$ at $t$, which is done by   \cite[Algorithm 12.19 (Triangular Sign Determination Algorithm)]{BPRbook2}.

If $f(t)\not=0$, we try to pseudo-invert $f$ modulo $\mathcal{F}$.
We perform \cite[Algorithm 8.22 (Extended Signed Subresultant)]{BPRbook2}  for $f$ and $f_{[m]}$, with respect to the
variable $T_m$ and compute a $\gcd(f,f_{[m]})\in \D[t']$ (the last non zero subresultant polynomial)
as well as the cofactors $u,v\in \D[t']$
with $uf+vf_{[m]}=\gcd(f,f_{[m]})$. 
\begin{enumerate}
 \item If $\gcd(f,f_{[m]})$ is of degree 0 in $T_m$, $u$
is a quasi-inverse of $f$.
 \item If $\gcd(f,f_{[m]})$ is of degree $>0$ in $T_m$, we have discovered a factor of $f_{[m]}$.
We define $h$ as the quasi-monic polynomial proportional to
 $f_{[m]}/\gcd(f,f_{[m]})$ obtained  by \cite[Algorithm 8.22 (Extended Signed Subresultant)]{BPRbook2}
 (see \cite[Algorithm 10.1 (Gcd and Gcd-free part)]{BPRbook2}).
 We perform   \cite[Algorithm 12.19 (Triangular Sign Determination)]{BPRbook2}
 applied to $f_{[m]}$ and ${\rm Der}(f_{[m]}),{\rm Der}(h)$ to identify the Thom encoding $\tau$ of $t_{[m]}$
as a root of $h$. We replace $f_{[m]}$ by $h$ and $\sigma_{[m]}$ by $\tau$ in $\mathcal{F}$.
Now $f$ and the new $f_{[m]}$, considered as polynomials in $T_{[m]}$ are coprime, and we can invert
$f$ modulo $f_{[m]}$.
\end{enumerate}

\end{newdescription}

\begin{proposition}
\label{prop:compring}
Let $\D$ be an ordered domain contained in a real closed field $\R$, and  
$t=(t_1,\ldots,t_m)\in \R^m$ be specified by a 
quasi-monic
triangular Thom encoding $\mathcal{F},\sigma$,
$$\mathcal{F}=(f_{[1]},\ldots,f_{[m]}), f_{[i]}\in \D[T_1,\ldots,T_i].$$
Let $d$ be a bound of the degree of $f_{[i]}$ with respect to  each $T_j$, $1\leq j\le i, 1\leq i\leq m$.

a) 
If $g\in \D[T_1,\ldots,T_m]$ is a polynomial of degree $D$,
the complexity  of computing 
a pseudo-reduction $(c,\bar g)$ of $g$ modulo $\mathcal{F}$
is $(Dd)^{O(m)}$ arithmetic operations in  $\D$.

b) The complexity of the computation of the pseudo-inverse of an element of
$\D[t]$ is $d^{O(m)}$ arithmetic operations in  $\D$.
\end{proposition}
\begin{proof}
a)
Suppose that $C_g\in \D$ is such that $C_g T_1^{i_1}\cdots T_m^{i_m} g $ has a reduction in 
$\D$ modulo $\mathcal{F}$
for every $(i_1, \ldots i_m)$
with 
$i_j<{\rm deg}(f_{[j]},T_j)$, $1\leq j\leq m$.
We denote by ${\rm Mat}(C_g g)$ the matrix of multiplication by $C_g g$ modulo 
$(\mathcal{F})$ with respect to monomial bases. 
The entries of ${\rm Mat}(C_g g)$ are in
$\D$. Its rows and columns
are indexed by $(i_1, \ldots, i_m)$,  $i_j<{\rm deg}(f_{[j]},T_j)$, $1\leq j\leq m$
and the 
$(j_1,\ldots,j_m)$-th entry of the column indexed by $(i_1,\ldots,i_m)$ is
the coefficient of
$T_1^{j_1}\cdots T_m^{j_m}$ in
the reduction of
$C_g T_1^{i_1}\cdots T_m^{i_m} g$ modulo $\mathcal{F}$.
Note that ${\rm Mat}(C_g C_h gh)={\rm Mat}(C_g g){\rm Mat}(C_h h)$.
Note also that the entries of the first column of ${\rm Mat}(C_g  g)$ (indexed by 
$(0,\ldots,0)$) 
are the coefficients of
the
reduction of $C_g g$ modulo $\mathcal{F}$.

We first compute $C_{T_j}$ such that 
$C_{T_j} T_1^{i_1}\cdots T_m^{i_m} T_j $ 
has a reduction in $\D$ modulo $\mathcal{F}$
for every $(i_1, \ldots i_m)$,  $i_h<{\rm deg}(f_{[h]},T_h)$, $1\leq
h\leq m$.
The algorithm proceeds by induction on $j$. 

For $j=1$, let $c_1\in \D$ be the leading coefficient of 
$f_{[1]}\in \D[T_1]$,  $d_1={\rm deg}(f_{[1]},T_1)$,
and $C_{T_1}=c_1$.
The matrix ${\rm Mat}(c_1 T_1)$ is simply obtained by replacing each occurence of
$c_1 T_1^{d_1}$ by $f_{[1]}-c_1T_1^{d_1}$ in 
$$c_1  T_1^{d_1-1}  T_{2}^{i_2}\cdots  T_{m}^{i_m} T_1$$
with $i_h<d_h$, $2\le h\le m$ 
and writing the result as a linear combination of the monomials
$ T_{1}^{j_1}\cdots T_{m}^{j_m},$  $j_i<d_i$, $1\leq i \leq  m$.
Compute ${\rm Mat}(C_{T_1}^{h} T_1^{h})={\rm Mat}(c_1 T_1)^{h}$, $h<2d$,
and define $C_1=c_1^{2d-1}$.

Suppose by induction 
that for every monomial $M$ in $T_1,\ldots, T_j$ of degree $<2d$,
 $C_M T_1^{i_1}\cdots T_m^{i_m} M $ has a reduction in $\D$  modulo $\mathcal{F}$
for every $(i_1, \ldots, i_m)$, $i_j<{\rm deg}(f_{[j]},T_j)$, $1\leq j \leq  m$.
Also, suppose that ${\rm Mat}( C_{M} M)$
has been computed. Denote by $C_{j}\in \D$ the product of $C_M$ for all the monomials $M$ 
of degree 
$<2d$ in the $j$ variables
$T_1,\ldots,T_j$.

Let $c_{j+1}\in \D$ be the leading coefficient of $f_{[j+1]}\in \D[T_1,\ldots,T_{j+1}]$ with respect to $T_{j+1}$ and  $d_{j+1}={\rm deg}(f_{[j+1]},T_{j+1})$, and take $C_{T_{j+1}}=c_{j+1} C_{j}$. 
The matrix ${\rm Mat}(C_{T_{j+1}} T_{j+1})$ is obtained by replacing each occurence of
$C_{T_{j+1}} T_{j+1}^{d_{j+1}}$ by $C_{j}f_{[j+1]}-C_{T_{j+1}} T_{j+1}^{d_{j+1}}$ in 
$$
C_{T_{j+1}}  T_{1}^{i_1}\cdots T_{j}^{i_j}T_{j+1}^{d_{j+1}-1} T_{j+2}^{i_{j+2}}\cdots T_{m}^{i_m} T_{j+1},$$
with $i_\ell<d_\ell$.

Notice that the polynomials  obtained this way 
have degrees at most $2d$ in
$T_1,\ldots,T_j$, and degrees $< d_h$ in $T_h$ for $h > j$. Reduce all
such monomials using the 
matrices of multiplication
computed before. 

Finally compute for every monomial $M$ of degree $\le D$ in
$T_1,\ldots,T_m$, $C_M$ and ${\rm Mat}(C_{M} M)$ by taking products of
the $C_{T_i}$ and the matrices ${\rm Mat}(C_{T_i} T_i)$
(respectively), and let $C_g$ be the product of the $C_M$ for all
monomials $M$ of degree $\le D$.  Now determine ${\rm Mat}(C_{g} g)$
by taking an appropriate linear combination of ${\rm Mat}(C_{M} M)$
and thus obtain the reduction of $C_g g$ modulo $\mathcal{F}$.
 
Notice that the complexity of computing the $C_{T_{j+1}}$, and  
${\rm Mat}(C_{T_{j+1}} T_{j+1})$ is bounded by $d^{O(m)}$.
In the last step,
there are $O(D)^m$ monomials of degree at most $D$, and hence at most 
$O(D)^m$ matrix multiplications to perform,
and the sizes of the matrices is
$d_1\cdots d_m \leq d^m$. So the complexity is   $(Dd)^{O(m)}$.

b) The proof  proceeds 
by induction  on the number of variables $m$ of $\mathcal{F}$.

If $m=1$, 
the computation of a gcd
 takes
$(d+1)^{c}$ 
operations in the domain $\D$,
for some universal constant $c  > 0$,
 using the complexity analysis of \cite[Algorithm 8.22 (Extended Signed Subresultant)]{BPRbook2} 
 and
\cite[Algorithm 10.13 (Sign Determination)]{BPRbook2}. 

If $m>1$, let $t=(t',u)$, 
and
we suppose by induction hypothesis that the complexity of arithmetic operations 
including pseudo-inversion
in
$\D[t']$ is $(d+1)^{c(m-1)}$ arithmetic operations in the ordered domain $\D$.
 The claim is clear since the arithmetic operations in the domain $\D[t]$ are using
 $(d+1)^{c}$ operations in the domain $\D[t']$ using the complexity analysis of 
 \cite[Algorithm 8.22 (Extended Signed Subresultant)]{BPRbook2}
 and
\cite[Algorithm 10.13 (Sign Determination)]{BPRbook2}.
\end{proof}

We can now give the description of Algorithm \ref{alg:blocklimits} (Limit of a Bounded Point).
\\

\noindent {\bf Description of Algorithm \ref{alg:blocklimits} (Limit of a Bounded Point)}

The precise input and output of this algorithm appear in Section \ref{subsec:limits}.

\begin{itemize}
\item[{\sc Procedure}]
Remove from 
$g(\eps)(T_1,\ldots,T_m,U)$ 
the coefficients
vanishing at the point 
$(t_1,\ldots,t_m)$, 
using  \cite[Algorithm 12.19 (Triangular Sign Determination)]{BPRbook2}.
Supposing without loss of generality 
that not all the coefficients of 
\[
g(\eps)(t_1,\ldots,t_m,U)
\]
are multiples of $\eps$, 
denote by $g(T_1,\ldots,T_m,U)$ the polynomial obtained by 
substituting $0$ for  $\eps$
in $g(\eps)(T_1,\ldots,T_m,U)$.

Similarly denote by $G(T_1,\ldots,T_m,U)$ the polynomials obtained by substituting $0$ for $\eps$
in $G(\eps)(T_1,\ldots,T_m,U)$.

Compute the set $\Sigma$ of Thom encodings of  roots of $g(t,U)$ using  \cite[Algorithm 12.19 (Triangular Sign Determination)]{BPRbook2}. Denoting by $\mu_\sigma$ the multiplicity of the root of $g(t,U)$ with Thom encoding $\sigma$, define
$G_\sigma$ as the $(\mu_\sigma-1)$-th derivative of $G$ with respect to $U$.

Identify the Thom encoding $\sigma$ and $G_\sigma$ 
representing $z$ using 
 \cite[Algorithm 12.19 (Triangular Sign Determination)]{BPRbook2},
by checking whe\-ther a ball of infinitesimal radius 
$\delta$ ($1\gg \delta\gg \eps>0$) around the point $x$ 
represented by the real univariate representation
$g,\sigma,G_\sigma$ contains 
$z(\eps)$.

Pseudo-invert the leading coefficient of the univariate representation, 
denote by $\mathcal{F}',\sigma'$  the new triangular Thom encoding 
describing $t$ and 
 compute a pseudo-reduction of the output modulo $\mathcal{F}'$.
\end{itemize}

\begin{proof}[Complexity analysis:]
Follows from the complexity of  \cite[Algorithm 12.19 (Triangular Sign Determination)]{BPRbook2}.
\end{proof}

\subsection{Limit of a curve}

Computing the limit of a curve  is not immediate when some part of the curve 
has a vertical limit, as seen in the following example.

\begin{example}
\label{notwell}
Consider the s-a curve 
$\gamma:[0,\eps]\rightarrow \R\langle\eps\rangle^3$, 
parametrized by the $X_1$ coordinate
defined by 
\[
\gamma(x_1) = (x_1, \gamma_2(x_1),\gamma_3(x_1)), x_1 \in [0,\eps]
\]
where $(\gamma_2(x_1),\gamma_3(x_1))$ is the  solution
of the triangular system,
$$
\displaylines{
X_2 - x_1/\eps  = 0, \cr
X_2^2 + X_3^2  - 1 = 0,
}
$$
with Thom encoding $(0,+), (0,+,+)$.

Notice that the image of $\gamma$ is contained in the 
cylinder of unit radius with axis the $X_1$-axis 
and is bounded over $\R$.
The image of $\gamma$ under the $\lim_\eps$ map  is contained in a circle
 in the plane $X_1=0$, and 
can no longer be described as a curve parametrized by the $X_1$-coordinate.

However, it is possible to reparametrize $\gamma$ by the $X_2$-coordinate.
By doing so we obtain another s-a curve 
$\varphi:[0,1] \rightarrow \R\langle\eps\rangle^3$ (having the
same image as $\gamma$) defined by  
\[
\varphi(x_2) = (\varphi_1(x_2), x_2, \varphi_3(x_2)), x_2 \in [0,1]
\]
where $(\varphi_1(x_2), \varphi_3(x_2))$ is the real solution 
of the triangular system 
$$
\displaylines{
 X_1-\eps x_2  = 0, \cr
 X_3^2 + x_2^2  - 1 = 0,
}
$$
with Thom encoding  $(0,-),(0,+,+)$.
Notice that the image under $\lim_\eps$  of the curve which is the graph of $\varphi$
can be easily described
as the curve represented by the following 
triangular system parametrized by $x_2 \in [0,1]$
$$
\displaylines{
X_1 = 0, \cr
X_3^2 + x_2^2  - 1 = 0,
}
$$
and Thom encoding $(0,-1), (0,+,+)$.
\end{example}

This is the reason why some kind of reparametrization is necessary before computing the limit.

\subsubsection{Reparametrization of curve segments}
\label{subsec:reparam}

We  
define the notion of well-parametrized curve, and prove that the limit of a well-parametrized curve is easy to describe.

\begin{definition}
\label{def:well-param}
A 
differentiable s-a curve
\[
\gamma = (\gamma_1,\ldots,\gamma_k) :(a,b)
\rightarrow \R^k
\] 

parametrized by $X_1$ (i.e. $\gamma_1(x_1)=x_1$)
is  {\em well-parametrized}  if for every $x_1 \in (a,b)$,
$$
\displaylines{
\sum_{i=1}^k \left(\frac{\partial \gamma_i}{\partial X_1}\right)^2 \le k.
}
$$

Let $t\in \R^m$ be represented by a triangular Thom encoding $\mathcal{F},\sigma$,
and
$$f_1,\sigma_1, f_2, \sigma_2,g, \tau,G$$
be  a curve segment
with parameter $X_j$ 
over $t$ on $(\alpha_1,\alpha_2)$ where
$\alpha_1$ and $\alpha_2$ are the elements of $\R$ represented by 
the  Thom encodings $f_1, \sigma_1$ and $f_2, \sigma_2$.

The curve segment
$$f_1,\sigma_1, f_2, \sigma_2,g, \tau,G$$
is  {\em well-parametrized} if
the s-a curve 
$\gamma: (\alpha_1,\alpha_2)\rightarrow \R\la\eps\ra^k$ 
defined by 
\[ 
\gamma(x_j) = 
\left(\frac{g_{1} (t,x_j, u(x_j))}{g_0 (t,x_j, u(x_j))},
   \ldots, \frac{g_k (t,x_j, u(x_j))}{g_0 (t,x_j, u(x_j))} \right)
\]
is well-parametrized,
where
$u:(\alpha_1,\alpha_2) \rightarrow \R$ 
maps each $x_j \in (\alpha_1,\alpha_2)$ to the root of $g(t,x_j,U)$ 
with Thom encoding $\tau$.
This means that
$$
\displaylines{
\sum_{i=1}^k\left(\left(\frac{g_i (t,x_j,u(x_j))}{g_0 (t,x_j,u(x_j))}
\right)'\right)^2 \leq k,
}
$$
where the derivative is taken with respect to $x_j$.
\end{definition}

Example \ref{notwell} is not a well-parametrized curve segment.

If a curve segment defined over $\R\la\eps\ra$ is
well-parametrized, and represents a curve bounded over $\R$, then
the image of the curve under the $\lim_\eps$ map can be easily described.
The following proposition
explains why this is true.

\begin{proposition}
\label{prop:limitofwellparametrized}
Let $(a(\eps),b(\eps))\subset \R\la \eps \ra$,
$a(\eps),b(\eps)$ bounded over $\R$, 
$r < j \leq  k$,
$z(\eps)
\in \R\la\eps\ra^{r}$,
and
\[
\gamma(\eps) = 
:(a(\eps),b(\eps))
\rightarrow \{z(\eps)\} \times \R\la\eps\ra^{k - r}
\] 
a s-a 
differentiable curve parametrized by
$X_j$  and  bounded over $\R$.
If $\gamma(\eps)$ is well-parametrized, then:
\begin{enumerate}
\item there exists a continuous extension of $\gamma(\eps)$ to a continuous, 
s-a curve,
\[
\gamma(\eps) = 
:[a(\eps),b(\eps)]
\rightarrow \{z(\eps)\} \times \R\la\eps\ra^{k - r}
\] 
defined over the closed interval $[a(\eps),b(\eps)]$;
\item
for each 
$x \in [\lim_\eps( a(\eps)),\lim_\eps(b(\eps))]$ 
and any 
$x(\eps) \in [a(\eps),b(\eps)]$
with $\lim_\eps (x(\eps)) = x$, 
$\gamma(x):= \lim_\eps (\gamma(\eps)(x)) =  \lim_\eps (\gamma(\eps)(x(\eps)))$;
\item
$\lim_\eps (\gamma(\eps)([a(\eps),b(\eps)])) = \gamma([\lim_\eps (a(\eps)),\lim_\eps (b(\eps)])$,
\end{enumerate}
In other words, the graph of the s-a function $\gamma(-):= \lim_\eps (\gamma(\eps)(-))$ 
is the image under $\lim_\eps$ of the graph of $\gamma(\eps)$.
\end{proposition}

\begin{proof}
  Since $\gamma(\eps)$ is bounded it follows that there exists a
  continuous extension of $\gamma(\eps)$ to the end points of the
  interval $(a(\eps),b(\eps))$.  It also follows from the definition
  of being well-parametrized that $||\gamma(\eps)'(x)|| \leq \sqrt{k}$
  for all $x\in (a(\eps),b(\eps))$.  By the s-a mean value theorem
  \cite[Exercice 3.4]{BPRbook2} we have that for each $x \in
  (a(\eps),b(\eps))\cap\R$ and any $x(\eps) \in (a(\eps),b(\eps))$
  with $\lim_\eps( x(\eps)) = x$,
\[
||\gamma(\eps)(x) - \gamma(\eps)(x(\eps))|| = ||\gamma(\eps)'(w(\eps))|||x-x(\eps)|,
\]
for some $w \in (x,x(\eps))$ (assuming without loss of generality that 
$x < x(\eps)$). Taking the image under $\lim_\eps$ and noticing that
$||\gamma(\eps)'(w(\eps))||$ is bounded over $\R$ by the previous observation, 
we obtain that
$$\lim_\eps (\gamma(\eps)(x)) = \lim_\eps (\gamma(\eps)(x(\eps))),$$ proving (1).
This implies  that the function 
$\gamma:(\lim_\eps (a(\eps)),\lim_\eps (b(\eps))) \rightarrow \R^k$ defined by 
$\gamma(x) = \lim_\eps (\gamma(\eps)(x))$ is a continuous,
bounded (since $\gamma(\eps)$ is boun\-ded over
$\R$)  s-a 
function, and hence
can be extended to a continuous, bounded s-a function 
on the closed interval $[\lim_\eps( a(\eps)), \lim_\eps (b(\eps))]$.
Moreover, it is clear that $\gamma(\lim_\eps (a(\eps))) = \lim_\eps (\gamma(\eps)(a(\eps)))$ and
$\gamma(\lim_\eps (b(\eps))) = \lim_\eps (\gamma(\eps)(b(\eps)))$, since 
{\small$$\gamma(\lim_\eps (a(\eps))),\gamma(\lim_\eps (b(\eps))) \in 
\overline{\gamma((\lim_\eps (a(\eps)),\lim_\eps( b(\eps))))}=\lim_\eps (\gamma(\eps)([a(\eps),b(\eps)])).$$}
It is then clear that (2) follows. 
\end{proof}

A s-a curve  
is  in general not well-parametrized.
However, subdividing if necessary the  curve into several pieces,
 it is possible to choose for each such piece
a parametrizing coordinate 
which makes the piece well-parametrized.
This is what we do in Algorithm \ref{alg:reparam} (Reparametrization of a Curve).

\begin{algorithm} {\bf [Reparametrization of a Curve]}
\label{alg:reparam}
\begin{itemize}
\item[{\sc Input.}] 
\begin{enumerate}
\item $t\in \R^m$ represented by a triangular Thom encoding $\mathcal{F},\sigma$,
\item  a bounded curve $S$ represented by a curve segment,
$$f_{1},\sigma_{1}, f_{2}, \sigma_{2},g, \tau,G$$
with parameter 
$X_1$ 
in $\R^{k}$ over $t$
on $(a,b)$.
\end{enumerate}
All the polynomials in the input
have coefficients in $\D$.

\item[{\sc Output.}] 
\begin{enumerate}
\item
A finite set 
$\mathcal{V} = \{v_1,\ldots,v_{N-1}\}$, 
of real univariate representations
over $(t,c_i)$  where each $c_i$ is 
represented by 
a Thom encoding over $t$ fixing
$X_{m(i)}$.
\item 
A finite set
$\mathcal{W}=\{w_1,\ldots,w_{N}\}$, 
of curve segments  with $w_i$ parametrized by $X_{\ell(i)}$.

Moreover, 
the union of the curves represented by 
$\mathcal{W}$,
and the points represented by 
$\mathcal{V}$
define a partition
of 
$S$.

\end{enumerate}
\item[{\sc Complexity}]
If 
the polynomials occurring in 
the input have degrees bounded by 
$D$, then the complexity of the algorithm is bounded by
$k^{O(1)}D^{O(m)}$.
\end{itemize}
\begin{itemize}
\item[{\sc Procedure.}]
\item[Step 1.] 
Let $g_{1}(X_{1},T) = X_{1} g_0(X_{1},T)$, and for each $i, {1} \leq i \leq k$, let
$$
\displaylines{
F_i :=  \left(\frac{\partial g}{\partial T}\right) 
\left(\frac{\partial g_i}{\partial X_{1}} g_0 - g_i
\frac{\partial g_0}{\partial X_{1}}\right) 
-
\left(\frac{\partial g_i}{\partial T} g_0 - g_i
\frac{\partial g_0}{\partial T}\right) 
\left(\frac{\partial g}{\partial X_{1}}\right) 
}
$$
(which is proportional to the projection on the $i$-th coordinate of the 
tangent vector to the input curve by the chain rule)
and 
$$
\displaylines{
G_i := k F_i^2 - \sum_{j={1}}^k F_j^2.
}
$$
\item[Step 2.]  Computing $\mathrm{RElim}_T(G_i,g)$, ${1}\leq i \leq
  k$, using \cite[Algorithm 11.19 (Restricted Elimination)]{BPRbook2},
  obtain a family $\mathcal{L}$ of polynomials in the ring 
  $\D[T_1,\ldots,T_m,X_1]$.  Subdivide $(a,b)$ in a finite union of
  points and intervals over which the signs of the polynomials in
  $\mathcal{L}$ are fixed using \cite[Algorithm 12.23 (Triangular
  Sampling Points)]{BPRbook2} and get $a=c_1 <\ldots < c_L = b$, where
  each $c_j$ is represented by a Thom encoding $(C_j,\sigma_j)$ over
  $t \in \R^m$, such that for each $j, {1} \leq j \leq L$, there
  exists an $\ell(j), {1}\leq \ell(j) \leq k$, such that for all
  $x_{1} \in (c_{j-1},c_j)$, $G_{\ell(j)}(t,x_{1},u(x_{1})) \geq 0$,
  denoting by $u(x_{1})$ the root of $g(t,x_{1},U)$ with Thom encoding
  $\tau$. For each $j$ fix an $\ell(j)$ satisfying this property.

\item[Step 3.]
For each $j, 1 \leq j\leq L$, reparametrize the segment
of the input curve over the interval $(c_{j-1},c_j)$ using
the coordinate $X_{\ell(j)}$. 
Suppose without loss of generality from here on that $\ell(j)=2$.
 \item[Step 3 a).]
Set 
\[
H := g^2 + (X_{2}\cdot g_0(T,X_{1},U)-g_{2}(T,X_{1},U))^2  \in \D[T,X_1,X_{2},U].
\]
Note that $\ZZ(H(t,-),\R^3)$ is a curve bounded over $\R$ (by
assumption on the input).  Call Algorithm \ref{alg:curvesegments}
(Curve Segments) with input the polynomial $H$, and the triangular
system $\mathcal{F},\sigma$, noticing that $X_{2}$ is now the
parameter.  Let $\mathcal{D}_i$ (respectively, $\mathcal{C}_i$) be the
set of distinguished points (respectively, curves) output by Algorithm
\ref{alg:curvesegments}.

\item[Step 3 b).]
For each element 
\[
(h(T,X_{2},V),\sigma_h ,H(T,X_{2},V))
\in \mathcal{D}_i,
\]
where
$$H(T,X_{2},V)=
(h_0(T,X_{2},V),h_1(T,X_{2},V),h_2(T,X_{2},V)),
$$
use
\cite[Algorithm 12.19 (Triangular Sign Determination)]{BPRbook2}
to check if the point $(x_{1},x_{2},u)$ represented by 
$(h,\sigma_h, (h_0,h_1,X_2 h_0,h_2))$ over
$t$,
coincides with 
$$(x_{1},\frac{g_{2}(t,x_{1},u(x_{1}))}{g_0(t,x_{1},u(x_{1}))},u(x_{1})).$$

Retain only the element
\[(h(T,X_{2},V),\sigma_h, H(T,X_{2},V))
\in \mathcal{D}_i
\] 
for which this is the case, and add to the set
$\mathcal{V}$ the real univariate representation 
$u = (h,\sigma_h, G_H)$ 
(see Notation \ref{not:subst2})
representing a point $v_h \in \R^k$, 
with parameter $X_{2}$ over $t$.

\item[Step 3 c).]
For each element 
\[
(f_{1}(T,V),\sigma_{1}, f_{2}(T,V),\sigma_{2},h(T,X_{2},V), \sigma_h, H(T,X_{2},V))
\in \mathcal{C}_i,
\]
where
$$H(T,X_{2},V)=(
h_0(T,X_{2},V),h_1(T,X_{2},V),h_2(T,X_{2},V)),
$$
use
 \cite[Algorithm 12.19 (Triangular Sign Determination)]{BPRbook2}
to check if the point $(x_{1},x_{2},u)$ represented by 
\[
h(T,X_{2},V),\sigma_h, (h_0,h_1,X_2 h_0,h_2)
\]
over $t$, coincides
with $$(x_{1},\frac{g_{2}(t,x_{1},u(x_{1}))}{g_0(t,x_{1},u(x_{1}))},u(x_{1}))$$
for $x_{2} = (v_{1}+v_{2})/2$ where $v_1,v_2$ are represented by
$(f_1,\sigma_1)$ and $(f_2,\sigma_2)$ respectively.  Retain only the
element of $\mathcal{C}_i$ for which this is the case, and add to the
set $\mathcal{W}$ the curve segment $$(f_1,\sigma_1,
f_2,\sigma_{2},h,\sigma_h,G_H)$$ with parameter $X_{2}$ over $t$ (see
Notation \ref{not:subst2}).

\end{itemize}
\end{algorithm}

\vspace{.1in}
\noindent
{\sc Proof of correctness.} 
Let 
\[
(f_1,\sigma_1, f_2, \sigma_2,g, \tau,G)
\]
be a curve segment parametrized by $X_1$ 
over $t$
representing the curve $\gamma: (a,b) \rightarrow \R^k$.

Let $(c,d)$ be a sub-interval of $(a,b)$ such that
for every $x_1\in (a,b)$
\begin{equation}
\label{eqn:G}
G_\ell(t,x_1,u(x_1)) = k F_\ell^2(t,x_1,u(x_1)) - 
\sum_{j=1}^{k} F_j^2(t,x_1,u(x_1)) \geq 0.
\end{equation}
(using the notation of Step 1 and Step 2).

This implies
$$
\displaylines{
\left\vert \frac{\partial \gamma_\ell}{\partial X_1} \right\vert 
\ge \frac{1}{\sqrt{k}},
}
$$
and hence
the mapping $\gamma_\ell$ from $(c,d)$ to $(c',d')$ 
with $c'=\gamma_\ell(c), d'=\gamma_\ell(d)$ is invertible.
Defining $\bar \gamma (x_\ell)=\gamma(\gamma_\ell^{-1}(x_\ell))$, $\bar \gamma( (c',d'))=\gamma((c,d))$ is well-parametrized
by $X_\ell$.

Moreover, at each point $x_1 \in (a,b)$ such a choice of $\ell$ exists,
since there must exist an  $\ell, 1 \leq \ell \leq k$ such that
$\displaystyle{
\left(\frac{\partial \gamma_\ell}{\partial X_1}\right)^2
}
$ 
is at least the average value
$\displaystyle{
\frac{1}{k}\sum_{i=1}^k \left(\frac{\partial \gamma_i}{\partial X_1}\right)^2
}
$.
Notice also that for such a choice of $\ell$ we have by the chain rule,
\begin{equation}
\label{eqn:well-parametrized}
\sum_{i=1}^k \left(\frac{\partial \gamma_i}{\partial X_\ell}\right)^2 =
\frac{\sum_{i=1}^k \left(\frac{\partial \gamma_i}{\partial X_1}\right)^2}
{\left(\frac{\partial \gamma_\ell}{\partial X_1}\right)^2}\leq k.
\end{equation}
In Step 2 of the algorithm we obtain a partition of the interval $(a,b)$ 
into points and open intervals, such
that over each sub-interval $(c_{j-1},c_j)$ of the partition, there exists 
an index $\ell = \ell(j)$ 
such that (\ref{eqn:G})  is satisfied at each point $v \in (c_{j-1},c_j)$,
and the curve segment over this interval is well-parametrized by $X_\ell$ by
(\ref{eqn:well-parametrized}).

Each curve segment corresponding
to elements of $\mathcal{V}$ output by the algorithm is thus well-parametrized.
The remaining property of the output is a  consequence of the
correctness of Algorithm \ref{alg:curvesegments} (Curve Segments), and
 \cite[Algorithm 12.19 (Triangular Sign Determination)]{BPRbook2}.
\eop

\vspace{.1in}
\noindent
{\sc Complexity analysis.}
Let $D$ be a bound on the degrees of the polynomials 
in the input.
The complexity of Steps 1 and 2 is bounded by 
$ k^{O(1)} D^{O(m)}$ 
from
the complexity of  
\cite[Algorithm 11.19 (Restricted Elimination)]{BPRbook2}, and 
\cite[Algorithm 12.23 (Triangular Sample Points)]{BPRbook2},
noting that the number of polynomials in $\mathcal{L}$ is bounded
by $k^{O(1)} D^{O(m)}$.

In Steps 3-4
the Algorithm \ref{alg:curvesegments} (Curve Segments)
and 
 \cite[Algorithm 12.19 (Triangular Sign Determination)]{BPRbook2}
are both called with a constant number of variables in the input.
Using the complexity analysis of these algorithms, the total complexity
is bounded by 
$k^{O(1)} D^{O(m)}$.
\eop

\subsubsection{Limit of a curve}
\label{subsec:limitofacurve}

We are now ready to describe
Algorithm \ref{alg:limitofacurve} (Limit of 
a Curve). \\

\noindent {\bf Description of Algorithm \ref{alg:limitofacurve} (Limit of 
a Curve)}

The algorithm proceeds by reparametrizing the curve and 
computing the limit of the well-parametrized curve segments so obtained,
as explained below.
Its precise input and output appear in Section \ref{subsec:limits}.

\begin{itemize}
\item[{\sc Procedure}] 
\item[Step 1.]  
Let $T=(T_1,\ldots,T_m)$, $X'=(X_{1},\ldots,X_{r})$. 
Call a slight variant of
\cite[Algorithm 12.18 (Parametrized Bounded Algebraic Sampling)]{BPRbook2},
 computing  pseudo-reductions of the 
 intermediate computations modulo $\mathcal{F}$ 
(using Proposition \ref{prop:compring}), 
with input 
$$
\sum_{A\in \mathcal{H}(\eps)}A^2\in \D[\eps,T,X']
$$
and parameters 
$\eps,T$, and output the set $\mathcal{U}_{\eps}$ of parametrized univariate representations 
 with variable $U$.

For every $(h(\eps),H(\eps)) \in \mathcal{U}(\eps)$,
use \cite[Algorithm 12.20 (Triangular Thom Encoding)]{BPRbook2}
with input the triangular system $(\mathcal{F},h(\eps))$ to compute
the Thom encodings of
the real roots of $h(\eps)(y,U)$.
If 
 $$\mathcal{H}(\eps)=(h_{[1]},\ldots,h_{[r]})$$ with
  $h_{[i]}\in \D[T,X_1,\ldots,X_i]$
substitute the variables $X'$ in  
$$\bigcup_{0,\ldots,r} {\rm Der}_{X_i}(h_{[i]})$$ 
using $H(\eps)$ 
by (Notation \ref{not:subst}) and define a family
$\mathcal{A}$ of polynomials in $\eps,T,U$. 
Using \cite[Algorithm 12. (Triangular Sign Determination)]{BPRbook2}, compute the signs of the polynomials
of $\mathcal{A}$ at the roots of $h(\eps)(y,U)$.
Comparing the Thom encodings, identify a specific
$(h(\eps),\tau(\eps),H(\eps))$ representing $z(\eps)$ over $t$.

 Then apply 
Algorithm \ref{alg:blocklimits}
(Limit of a Bounded Point)
with  input 
\[
(h(\eps),\tau(\eps),H(\eps))
\] representing $z(\eps)$ over $t$
to obtain 
a
real univariate representation $p_z,\rho_z,P_z$
representing
$z$
over $t$.
 
\item[Step 2.] Using Algorithm \ref{alg:reparam} (Reparametrization of a Curve)
reparametrize the input curve segment. 

\item[Step 3.] 
For every well-parametrized curve segment $S(\eps)$ computed in Step 2, and represented by
$$
 (f(\eps)_{1},\sigma(\eps)_{1},f(\eps)_{2},\sigma(\eps)_{2},g(\eps),\tau(\eps),G(\eps)),
$$
do the following. 

First reorder the
variables to ensure that the parameter of $S(\eps)$  is $X_{r+1}$.

Then compute a description of $\lim_\eps(S(\eps))$. 
This process is going to generate  a finite list of open intervals and points above which the 
representation of the restriction of the curve $\lim_\eps(S(\eps))$ 
by a curve segment
is fixed.
This is done as follows.

\item[Step 3 a).] 
Denote by $\alpha(\eps)_{1}$ the element of 
$\R\la \eps\ra$  represented by 
$$f(\eps)_{1}(T,X',X_{r+1}),\sigma(\eps)_{1}$$
over $(t,z(\eps))$.

Call a slight variant of
\cite[Algorithm 12.18 (Parametrized Bounded Algebraic Sampling)]{BPRbook2},
computing  pseudo-reduction of the intermediate computations modulo $\mathcal{F}$ 
of the output modulo $\mathcal{F}$
(using Proposition \ref{prop:compring}), 
with input 
$$
\sum_{A\in \mathcal{H}(\eps)}A^2+f(\eps)_{1}(T,X',X_{r+1})^2 \in \D[\eps,T,X',X_{r+1}]
$$
and parameters 
$\eps,T$,
and output a set $\mathcal{U}'_{\eps}$ of parametrized univariate representations 
 with variable $U$.

For every $(h(\eps),H(\eps)) \in \mathcal{U}'_{\eps}$,
use \cite[Algorithm 12.20 (Triangular Thom Encoding)]{BPRbook2}
with input the triangular system $(\mathcal{F},h(\eps))$ to compute
the Thom encodings of
the real roots of $h(\eps)(y,U)$.

If 
 $$\mathcal{H}(\eps)=(h_{[1]},\ldots,h_{[r]})$$ with
  $h_{[i]}\in \D[T,X_1,\ldots,X_i,]$
substitute the variables $X',X_{r+1}$ in $${\rm Der}_{X_{r+1}}(f(\eps)_{1}(T,X',X_{r+1}),X_{r+1}\cup \bigcup_{1,\ldots,r} {\rm Der}_{X_i}(h_{[i]})$$ 
using  Notation \ref{not:subst} and define a family
$\mathcal{B}$ of polynomials in $\eps,T,U$. 
Using \cite[Algorithm 12. (Triangular Sign Determination)]{BPRbook2}, compute the signs of the polynomials
of $\mathcal{B}$ at the roots of $h(\eps)(y,U)$.
Comparing the Thom encodings, identify a specific
$(h(\eps),\tau(\eps),H(\eps))$ representing $(z(\eps),\alpha(\eps)_{1})$ over $t$.

 Then apply 
Algorithm \ref{alg:blocklimits}
(Limit of a Bounded Point)
with input 
\[
(h(\eps),\tau(\eps),H(\eps))
\] 
representing $(z(\eps),\alpha(\eps)_{1})$ over $t$
to obtain 
a
quasi-monic
real univariate representation
$p_{z,\alpha_1},\rho_{z,\alpha_1},P_{z,\alpha_1}$
representing
$(z,\alpha_1)$
over $t$ with  $\alpha_1=\lim_\eps(\alpha(\eps)_{1})$.
Obtain 
a Thom encoding
over $t$,
of $\alpha_1$
using 
 \cite[Algorithm 15.1 (Projection)]{BPRbook2}.

\noindent Similarly, for $\alpha(\eps)_{2}$ 
the element of $\R\la \eps\ra$  represented by 
$$f(\eps)_{2}(T,X',X_{r+1}),\sigma(\eps)_{2}$$
over $(t,z(\eps))$, compute a Thom encoding over $t$, of
$\alpha_2=\lim_\eps(\alpha(\eps)_{2})$.

\item[Step 3 b).] 
Perform a slight variant of
\cite[Algorithm 12.18 (Parametrized Boun\-ded Algebraic Sampling)]{BPRbook2}, 
computing pseudo-reductions of intermediate computations modulo $\mathcal{F}$ 
of the output modulo $\mathcal{F}$
(using Proposition \ref{prop:compring}), 
with input 
$$
\sum_{A\in \mathcal{H}(\eps)}A^2+g(\eps)(T,X',X_{r+1},V)^2 \in \D[\eps,T,X',X_{r+1},U]
$$
with parameters $\eps,T,X_{r+1}$ and output a set $\mathcal{V}(\eps)$
of paramet\-rized univariate representations with parameter
$\eps,T,X_{r+1}$ and variable $V$.  Denote by $\mathbf{\Theta}(\eps)$
the set of polynomials $\theta(\eps)$ such that there exis\-ts
$\Theta(\eps)$ with $(\theta(\eps),\Theta(\eps))\in
\mathcal{V}(\eps)$.  Note that $\theta(\eps) \in
\D[\eps,T,X_{r+1},V]$.

\item[Step 3 c).]  Compute the family of coefficients
  $\mathcal{C}\subset \D[T,X_{r+1}]$ of the polynomials $\theta(\eps)
  \in \mathbf{\Theta}(\eps)$ considered as elements of
  $\D[T,X_{r+1}][\eps,V]$ and the list $\mathcal{L} \subset \{=0,\neq
  0\}^{\mathcal{C}}$ of non-empty conditions
  $=0,\not= 0$ satisfied by $\mathcal{C}$ in $\R$ using
  \cite[Algorithm 12.23 (Triangular Sample Points)]{BPRbook2}.  Note
  that for every $x_{r+1}$ in the realization of $\tau\in
  \mathcal{L}$, the orders in $\eps$ of the coefficients of the
  polynomials in $\mathbf{\Theta}(\eps)(t,x_{r+1})\subset \D[\eps,V]$
  are fixed.  For every $\theta(\eps) \in \mathbf{\Theta}(\eps) $ we
  denote by $o(\theta(\eps),\tau)$ the minimal order in $\eps$ of the
  coefficients of $\theta(\eps)(t,x_{r+1})$ on the realization of
  $\tau$ and by $\mathbf{\Theta}_\tau \subset \D[T,X_{r+1},V]$ the set
  of polynomials obtained by substituting $0$ for $\eps$ in
  $\eps^{-o(\theta(\eps),\tau)}\theta(\eps)$.

\item[Step 3 d).] 
Define 
$$
\mathbf{\Theta}=\bigcup_{\tau \in \mathcal{L}} \mathbf{\Theta}_\tau \subset \D[T,X_{r+1},V].
$$
Compute 
$$
\mathcal{E}=\mathcal{C} \cup\bigcup_{\theta \in \mathbf{\Theta}} {\rm RElim}_V(\theta,{\rm Der}(\theta))\subset \D[T,X_{r+1}]
$$
using 
 \cite[Algorithm 11.19 (Restricted Elimination)]{BPRbook2},
so that  the 
Thom encodings of the real roots of 
$\theta(t,x_{r+1},V)$ 
are fixed when 
$x_{r+1}$ varies in an open interval  defined by the roots of
the polynomials $\mathcal{E}(t)$.

\item[Step 3 e).] Compute using 
 \cite[Algorithm 12.19 (Triangular Sign Determination)]{BPRbook2}
the Thom encodings of the real roots of the 
polynomials 
in $\mathcal{E}(t)$, and the ordered list $c_1<\cdots< c_{h-1}$ of the roots of 
the polynomials in $\mathcal{E}(t)$ in the interval $(c_0,c_h)$, 
with $c_0=\alpha_1, c_h=\alpha_2$.  Denote by $C_j,\rho_j$ a 
polynomial in $\mathcal{E}(t)$ and a Thom encoding representing $c_j$.

\item[Step 3 f).] 
For every $j$ from 1 to $h-1$, 
and for every 
$\theta\in \mathbf{\Theta}$,
determine using 
 \cite[Algorithm 12.19 (Triangular Sign Determination)]{BPRbook2},
the Thom encoding 
\[
\theta(t,c_j,V),\tau_j
\]
of a root $v_j$
such that $v_j=\lim_\eps(v(\eps))$, 
where $v(\eps)$ is the root of 
$\theta(\eps)(t,c_j,V)$ 
with Thom encoding 
$\tau(\eps)$. 
The 
multiplicity $\mu_j$ of the root $v_j$ is determined by $\tau_j$.

\item[Step 3 g).] 
For every $j$ from 1 to $h$, define $I=(c_{j-1},c_j)$. For every 
$\theta\in \mathbf{\Theta}$ 
determine, using
 \cite[Algorithm 12.19 (Triangular Sign Determination)]{BPRbook2}
the Thom encoding 
$\theta_{I}(t,x_{r+1},V),\tau_{I}$ 
of a root 
$v_{I}(x_{r+1})$, of multiplicity $\mu_I$  
such that for every $x_{r+1}\in I$, 
$v_{I}(x_{r+1})=\lim_\eps(v(\eps))$ 
where $v(\eps)$ is the root of 
$\theta(\eps)(t,x_{r+1},V)$ 
with Thom encoding 
$\tau(\eps)$.
The 
multiplicity $\mu_I$ of the root $v_{I}(x_{r+1})$ is determined by $\tau_{I}$.

\item[Step 3 h).] 
Given 
$(\theta(\eps),\Theta(\eps))$ in $\mathcal{U}(\eps)$ denote by $(g_{\Theta(\eps)},G_{\Theta(\eps)})$ 
the $k-r+1$-tuple of polynomials obtained by substituting in  
$(g(\eps),G(\eps))$  the variables $X',U$ by $F(\eps)$  (see Notation  \ref{not:subst}). 
Denote by 
$\mathcal{V}'(\eps)\subset \D[\eps,T,X_j,V]$ the set of 
$k-r+1$-tuples of polynomials 
$(g_{\Theta(\eps)},G_{\Theta(\eps)})$.

\item[Step 3 i).]  For every $j$ from 1 to $h-1$ and every
  $(h(\eps),H(\eps))\in \mathcal{V}'(\eps)$, with
  $H(\eps)=(h(\eps)_{0},h(\eps)_{r+2},\ldots, h(\eps)_{k})$ determine
  the order in $\eps$
  of $$h_{\eps}(t,c_j,v_j),h(\eps)_{i}(t,c_j,v_j).$$ This is done by
  determining the signs of the coefficient $h_{\ell},h_{i,\ell}$ of
  $\eps^\ell$ in $h(t,c_j,v_j),h_i(t,c_j,v_j)$ using \cite[Algorithm
  12.19 (Triangular Sign Determination)]{BPRbook2}.  Retain those
  $(h(\eps),H(\eps))$ such that $o(h(\eps)_{0})\leq o(h(\eps)_{i})$
  for all $i$ from $r+2$ to $k$ and replace $\eps$ by $0$
  in $$(\eps^{-o(h_{\eps})}h(\eps),\eps^{-o(h(\eps)_{0})}H(\eps)),$$
  which defines a set $\mathcal{H}_j$.  Inspecting every $(h,H)\in
  \mathcal{H}_j$, determine, using \cite[Algorithm 12.19 (Triangular
  Sign Determination)]{BPRbook2}, a $k-r+1$-tuple $(h_j,H_j)$ with the
  following property.  Let $d_j$ be the point represented by the real
  univariate representation
\[
(h_j(T,X_{r+1},V),\tau_j,H_j^{(\mu_j-1)}(T,X_{r+1},V))
\]
over $t,u$.
The image under $\lim_\eps$ of the 
point of $S(\eps)$ 
with $X_{r+1}$-coordinate $(c_j)$
is $(z,c_j,d_j)$.

\item[Step 3 j).] 
For every $j$ from 1 to $h$  define $I=(c_{j-1},c_j)$. For every $(h(\eps),H(\eps))\in 
\mathcal{V}'(\eps)$, with $H(\eps)=(h(\eps)_{0},h(\eps)_{r+2},\ldots, h(\eps)_{k})$
subdivide 
$I$ 
so that  the order in $\eps$ of 
$h(\eps)(t,c_j,v_j)$ and
$h_i(\eps)(t,x_{r+1},v_{I}(x_{r+1}))$ is fixed. This is done by computing
$$
\mathcal{E}_{I}=
\bigcup_{\theta \in \mathbf{\Theta},(h,H(\eps)) \in \mathcal{V}'(\eps),0 \leq \ell \leq  \deg_\eps h_i}  {\rm RElim}_V(\theta,h_{\ell}) \subset \D[T,X_{r+1}],
$$
and
$$
\mathcal{E}_{I,i}=
\bigcup_{\theta \in \mathbf{\Theta},(h,H(\eps)) \in \mathcal{V}'(\eps),0 \leq \ell \leq \deg_\eps h_i} {\rm RElim}_V(\theta,h_{i,\ell})
\subset \D[T,X_{r+1}],
$$
using 
 \cite[Algorithm 11.19 (Restricted Elimination)]{BPRbook2}.

Defining $$\mathcal{E}'_I=\mathcal{E}_I\cup \bigcup_{i\in{0,r+2,\ldots,k}} \mathcal{E}_{I,i},$$
compute
the Thom encodings of the roots of the polynomials in 
$\mathcal{E}'_{I}(t)$,
using 
 \cite[Algorithm 12.19 (Triangular Sign Determination)]{BPRbook2}.
On each 
open
interval $J$  between 
two successive
roots, 
the order in $\eps$, denoted by $o(h_{\eps}),o(h(\eps)_{i})$ of the polynomials  
$$h(\eps)(t,x_{r+1},v_{J}(x_{r+1})),h_i(\eps)(t,x_{r+1},v_{J}(x_{r+1}))$$ 
remains fixed.

Retain those $(h(\eps),H(\eps))$ such that
$o(h(\eps)_{0})\leq o(h(\eps)_{i})$ 
for all $i$ from $r+2$ to $k$ and replace $\eps$ by $0$
in $\eps^{-o(h(\eps)_{0})}(h,H(\eps))$, which defines a set $\mathcal{H}_J$.
Inspecting every $(h,H)\in \mathcal{H}_J$,
determine, using 
 \cite[Algorithm 12.19 (Triangular Sign Determination)]{BPRbook2},
a $k-r+1$-tuple $(h_J,H_J)$ such that the point represented by
$$(h_J(t,x_{r+1},v_{I}),H_J^{(\mu_J-1)}(t,x_{r+1},v_{I}))$$ is the image under $\lim_\eps$ 
of the point of $S(\eps)$ with $X_{r+1}$-coordi\-nate $x_{r+1}$, where $\mu_J$ is the multiplicity of $u_{J}(x_{r+1})$ as a root of $h_{J}(x_{r+1},V)$.

Let
$w_J$ be the curve represented by the curve segment representation
\[
h_{I}(T,X_{r+1},U),\tau_j,H_J^{(\mu_J-1)}(T,X_{r+1},U)
\]
with parameter $X_{r+1}$ over $t,u$.
 
\item[Step 3 k).]
Let $c_1< \cdots< c_{N-1}$ denote the set of all the elements of $\R$ computed in Steps 
2 d), and 2 i) above, and $c_N=c$.  
Re-index each $v_j$ computed in Step 3 h), such that $d_j$ lies above
$c_j$. Similarly, re-index each $w_I$ computed in Step 3 i) by some $j, 1 \leq j \leq N$, so that $w_j$ lies above the interval $(c_{j-1},c_j)$.

Output the lists consisting of 
$d_1,\ldots,d_{N-1}$, and 
$w_1,\ldots,w_N$. 
\end{itemize}

\vspace{.1in}
\noindent
{\sc Proof of correctness.}
Let $\gamma(\eps): (\alpha(\eps)_{1},\alpha(\eps)_{2}) \rightarrow
\R\la\eps\ra^k$ be the curve represented by a well-parametrized curve segment 
$$
f(\eps)_{1},\sigma(\eps)_{1},f(\eps)_{2},\sigma(\eps)_{2},g(\eps),\tau(\eps),G(\eps)
$$
computed in Step 2.

Let $G:(\alpha_1,\alpha_2) \rightarrow \R^k$ be the curve whose image equals
the image of $\gamma(\eps)$ under $\lim_\eps$. 
Since the input curve segment is well-parametrized
it follows from Proposition \ref{prop:limitofwellparametrized}
that in order to compute for any 
$x_1 \in (c_0,c_N)$, $G(x_1)$ it suffices to compute 
$\lim_\eps \gamma(\eps)(x_1)$.
The proof of correctness of the algorithm is then similar to the
proof of correctness of 
Algorithm \ref{alg:blocklimits} (Limit of a Bounded Point).

\eop

\vspace{.1in}
\noindent
{\sc Complexity analysis.}
Let $D$ be a bound on the degrees of all polynomials appearing in the input.
We first bound the degrees in the various 
variables, $\eps, T,X',X_{r+1},U,V$ of the polynomials computed 
in various steps of the algorithm.
In Step 1, the degrees of the polynomials in $\mathcal{U}(\eps)$ are bounded
as follows. The degrees in $\eps,U$ are bounded by $D^{O(r)}$ by the complexity
analysis 
of \cite[Algorithm 12.18 (Parametrized Bounded Algebraic Sampling)]{BPRbook2} 
and the degrees in the $T_i$ are bounded by $D$, because of the 
pseudo-reduction. Moreover, the complexity of this step is bounded by
$D^{O(m+r)}$ from the complexity of 
\cite[Algorithm 12.18 (Parametrized Bounded Algebraic Sampling)]{BPRbook2}
and the complexity of pseudo-reduction (see Definition \ref{triangthom}).

The degrees in $\eps,T_i,X',U$ in the output of Step 2 are all bounded by 
$D^{O(1)}$ and the complexity of Step 2 is bounded by $$(k-r)^{O(1)} D^{O(m+r)}
= k^{O(1)}D^{O(m+r)}$$
using the complexity analysis of Algorithm
\ref{alg:reparam} (Reparametrization of a Curve).

The degrees of the polynomials in Step 3 a are bounded as follows.
In the output of the call to 
\cite[Algorithm 12.18 (Parametrized Bounded Algebraic Sampling)]{BPRbook2},
the degrees in $\eps,U$ are bounded by $D^{O(r)}$,
and the degrees in the $T_i$ are bounded by $D$.
Now, from the complexity analysis of Algorithm \ref{alg:blocklimits}
(Limit of a Bounded Point) it follows that the degrees in the $T_i$ 
of the polynomials output are  bounded by $D$ and those in $\eps,U$ are
bounded by $D^{O(r)}$. Moreover, the complexity of Step 3 a
is bounded by
$D^{O(m+r)}$ from the complexity of 
\cite[Algorithm 12.18 (Parametrized Bounded Algebraic Sampling)]{BPRbook2},
the complexity of 
Algorithm \ref{alg:blocklimits}
(Limit of a Bounded Point)
and the complexity of pseudo-reduction (see Proposition \ref{prop:compring}).

The degrees of the polynomials in Step 3 b are bounded as follows.
In the output of the call to 
\cite[Algorithm 12.18 (Parametrized Bounded Algebraic Sampling)]{BPRbook2},
the degrees in $\eps,X_{r+1},V$ are bounded by $D^{O(r)}$,
and the degrees in the $T_i$ are bounded by $D$.
The complexity of Step 3 b
is bounded by
$D^{O(m+r)}$ from the complexity of 
\cite[Algorithm 12.18 (Parametrized Bounded Algebraic Sampling)]{BPRbook2},
and the complexity of pseudo-reduction (see Definition \ref{triangthom}).

The complexity of Step 3 c is bounded by $D^{O(m+r)}$ using the degree 
bounds from the complexity analysis of the previous steps and the complexity
of \cite[Algorithm 12.23 (Triangular Sample Points)]{BPRbook2}.

It now follows from the complexity analysis of 
\cite[Algorithm 12.19 (Triangular Sign Determination)]{BPRbook2},
 \cite[Algorithm 11.19 (Restricted Elimination)]{BPRbook2},
and the degree estimates proved above that the complexity of the
remaining steps are all bounded by $k^{O(1)} D^{O(m+r)}$.
Thus, the complexity of the algorithm is bounded by $k^{O(1)} D^{O(m+r)}$.
\eop

\section*{Acknowledgement}
We are very grateful to the ananymous referees of the paper for their numerous suggestions.
We are particularly grateful to one of them for pointing out an error in a preliminary version.

\bibliographystyle{plain}
\bibliography{master}

\def\cprime{$'$}
\begin{thebibliography}{10}

\bibitem{BPR99}
S.~Basu, R.~Pollack, and M.-F. Roy.
\newblock Computing roadmaps of semi-algebraic sets on a variety.
\newblock {\em J. Amer. Math. Soc.}, 13(1):55--82, 2000.

\bibitem{BPRbook2}
S.~Basu, R.~Pollack, and M.-F. Roy.
\newblock {\em Algorithms in real algebraic geometry}, volume~10 of {\em
  Algorithms and Computation in Mathematics}.
\newblock Springer-Verlag, Berlin, 2006 (second edition).

\bibitem{BPRbook2posted2}
S.~Basu, R.~Pollack, and M.-F. Roy.
\newblock {\em Algorithms in real algebraic geometry}, volume~10 of {\em
  Algorithms and Computation in Mathematics}.
\newblock Springer-Verlag, Berlin, 2011, online version posted on 3/08/2011,
  available at {\url{http://perso.univ-rennes1.fr/marie-francoise.roy/}}.

\bibitem{BCR}
J.~Bochnak, M.~Coste, and M.-F. Roy.
\newblock {\em G\'eom\'etrie alg\'ebrique r\'eelle (Second edition in english:
  Real Algebraic Geometry)}, volume 12 (36) of {\em Ergebnisse der Mathematik
  und ihrer Grenzgebiete [Results in Mathematics and Related Areas ]}.
\newblock Springer-Verlag, Berlin, 1987 (1998).

\bibitem{Canny87}
J.~Canny.
\newblock {\em The Complexity of Robot Motion Planning}.
\newblock MIT Press, 1987.

\bibitem{Col}
G.~E. Collins.
\newblock Quantifier elimination for real closed fields by cylindric algebraic
  decomposition.
\newblock In {\em Second GI Conference on Automata Theory and Formal
  Languages}, volume~33 of {\em Lecture Notes in Computer Science}, pages
  134--183, Berlin, 1975. Springer- Verlag.

\bibitem{CLR}
Michel Coste, Henri Lombardi, and Marie-Fran{\c{c}}oise Roy.
\newblock Dynamical method in algebra: effective {N}ullstellens\"atze.
\newblock {\em Ann. Pure Appl. Logic}, 111(3):203--256, 2001.

\bibitem{Mohab-Schost2010}
Mohab~Safey el~Din and Eric Schost.
\newblock A baby steps/giant steps probabilistic algorithm for computing
  roadmaps in smooth bounded real hypersurface.
\newblock {\em Discrete Comput. Geom.}, 45(1):181--220, 2010.

\bibitem{GR92}
L.~Gournay and J.~J. Risler.
\newblock Construction of roadmaps of semi-algebraic sets.
\newblock {\em Appl. Algebra Eng. Commun. Comput.}, 4(4):239--252, 1993.

\bibitem{GV92}
D.~Grigoriev and N.~Vorobjov.
\newblock Counting connected components of a semi-algebraic set in
  subexponential time.
\newblock {\em Comput. Complexity}, 2(2):133--186, 1992.

\bibitem{GHRSV90}
D.~Yu. Grigoriev, J.~Heintz, M.-F. Roy, P.~Solern{\'o}, and N.N. Vorobjov, Jr.
\newblock Comptage des composantes connexes d'un ensemble semi-alg\'ebrique en
  temps simplement exponentiel.
\newblock {\em C. R. Acad. Sci. Paris S\'er. I Math.}, 311(13):879--882, 1990.

\bibitem{HRS93}
J.~Heintz, M.-F. Roy, and P.~Solern\`{o}.
\newblock Single exponential path finding in semi-algebraic sets ii: The
  general case.
\newblock In Chandrajit~L. Bajaj, editor, {\em Algebraic geometry and its
  applications}, pages 449--465. Springer-Verlag, 1994.
\newblock Shreeram S. Abhyankar's 60th birthday conference, 1990.

\bibitem{HRS90}
Joos Heintz, Marie-Fran{\c{c}}oise Roy, and Pablo Solern{\'o}.
\newblock Single exponential path finding in semialgebraic sets. {I}. {T}he
  case of a regular bounded hypersurface.
\newblock In {\em Applied algebra, algebraic algorithms and error-correcting
  codes ({T}okyo, 1990)}, volume 508 of {\em Lecture Notes in Comput. Sci.},
  pages 180--196. Springer, Berlin, 1991.

\bibitem{SS}
J.~Schwartz and M.~Sharir.
\newblock On the piano movers' problem ii. general techniques for computing
  topological properties of real algebraic manifolds.
\newblock {\em Adv. Appl. Math.}, 4:298--351, 1983.

\bibitem{GV90}
N.~N. Vorobjov, Jr. and D.~Yu. Grigoriev.
\newblock Determination of the number of connected components of a
  semi-algebraic set in subexponential time.
\newblock {\em Dokl. Akad. Nauk SSSR}, 314(5):1040--1043, 1990.

\end{thebibliography}

\end{document}